\documentclass[a4paper,11pt]{article}

\pdfoutput=1 
\usepackage[utf8]{inputenc}
\usepackage[british]{babel}
\usepackage{amsthm}
\usepackage{amsmath}
\usepackage{amssymb}
\usepackage{tikz}
\usepackage{booktabs}
\usepackage[font={small},width=0.9\textwidth]{caption}
\usetikzlibrary {patterns,patterns.meta}

\newtheorem{theorem}{Theorem}[section]

\newtheorem{lemma}[theorem]{Lemma}

\theoremstyle{definition}
\newtheorem{example}[theorem]{Example}

\newcommand{\mytikz}[2]{\begin{tikzpicture}[scale = #1, baseline={([yshift=-.8ex]current bounding box.center)}]#2\end{tikzpicture}}

\newcommand{\trinone}{\draw (0.000, 0.000) -- (4.000, 0.000) -- (2.000,3.464) -- cycle;}
\newcommand{\tritop}{\draw (0.000, 0.000) -- (4.000, 0.000) -- (2.500, 2.598) -- (1.500, 2.598) -- cycle;}
\newcommand{\trilower}{\draw (2.000, 3.464) -- (0.500, 0.866) -- (1.000, 0.000) -- (3.000, 0.000) -- (3.500, 0.866) -- cycle;}
\newcommand{\triall}{\draw (2.500, 2.598) -- (1.500, 2.598) -- (0.500, 0.866) -- (1.000, 0.000) -- (3.000, 0.000) -- (3.500, 0.866) -- cycle;}
\newcommand{\tridown}{\draw (0.000, 3.464) -- (4.000, 3.464) -- (2.000, 0.000) -- cycle;}

\begin{document}
\begin{center}

{\Large\textbf{The Pattern Complexity of the Sierpi\'{n}ski Triangle}}

\vspace{2ex}
Johan Nilsson
\end{center}

\begin{abstract}
	We give exact formulas for the number of distinct triangular patterns (or subtriangles) of a given size that occur in the Sierpi\'{n}ski Triangle.

\end{abstract}

{\def\thefootnote{}}

\section{Introduction}

The Sierpi\'{n}ski triangle \cite{sierpinski} is a well known fractal structure. It can be created by starting from an equilateral triangle, which is sub-divided into 4 equilateral triangles and where the central one is removed. The removal procedure is now repeated recursively on the tree remaining triangles. The Sierpi\'{n}ski triangle appears in many areas, e.g.\ the attractor in the \emph{chaos game} \cite{barnsley}, in Pascal's triangle \cite[p.~80]{wolfram}, or in Wolfram's rule 90 in \cite[p.~25]{wolfram}, just to mention a few. In the On-Line Encyclopedia of Integer Sequences OEIS \cite{oeis} there are several entries concerning the Sierpi\'{n}ski triangle, such as \texttt{A047999}, and \texttt{A070886}. 

The approach we will apply here, to create the Serpi\'{n}\-ski triangle, is to use the substitution rule $\mu$ defined by 
\begin{equation}
\label{eq:DefOfMu}
\mu : 
\begin{array}{*{11}{c}}
	\mytikz{0.4}{\draw (0.000, 0.000) -- (1.000, 0.000) -- (0.500, 0.866) -- cycle;}
& \mapsto &
\mytikz{0.4}{
	\draw (0.000, 0.000) -- (2.000, 0.000) -- (1.000, 1.732) -- cycle;
	\draw[line join=round] (1.000, 0.000) -- (1.500, 0.866) -- (0.500, 0.866) -- cycle;
},
&&
\mytikz{0.4}{\draw (0.000, 0.866) -- (1.000, 0.866) -- (0.500, 0.000) -- cycle;}
& \mapsto &
\mytikz{0.4}{ 
	\draw (0.000, 1.732) -- (1.000, 0.000) -- (2.000, 1.732) -- cycle;
	\draw[line join=round] (0.500, 0.866) -- (1.500, 0.866) -- (1.000, 1.732) -- cycle;},
&&
\mytikz{0.4}{\filldraw (0.000, 0.000) -- (1.000, 0.000) -- (0.500, 0.866) -- cycle;}
& \mapsto &
\mytikz{0.4}{
	\draw (0.000, 0.000) -- (2.000, 0.000) -- (1.000, 1.732) -- cycle;
	\filldraw[line join=round] (0.000, 0.000) -- (1.000, 0.000) -- (0.500, 0.866) -- cycle;
	\filldraw[line join=round] (1.000, 0.000) -- (2.000, 0.000) -- (1.500, 0.866) -- cycle;
	\filldraw[line join=round] (0.500, 0.866) -- (1.500, 0.866) -- (1.000, 1.732) -- cycle;},
\end{array}
\end{equation}
and where $\mytikz{0.08}{\filldraw (0.000, 0.000) -- (4.000, 0.000) -- (2.000,3.464) -- cycle;}$ is taken as seed. See Figure~\ref{fig:FirstIterationsOfMu} for an illustration of the first iterations of $\mu$. Note that there are no filled downward oriented triangles with side length 1 in \eqref{eq:DefOfMu}, (downwards meaning that precisely one corner is at the bottom. Similarly, if a triangle has precisely one corner at the top, we say that it is oriented upwards). We denote by $T$ the limit structure obtained under iteration of $\mu$ on $\mytikz{0.08}{\filldraw (0.000, 0.000) -- (4.000, 0.000) -- (2.000,3.464) -- cycle;}$, and we call it the Sierpi\'{n}ski triangle.

In this paper we focus on the different triangular patterns (or subtriangles) that occur in $T$, and we prove the following theorem.

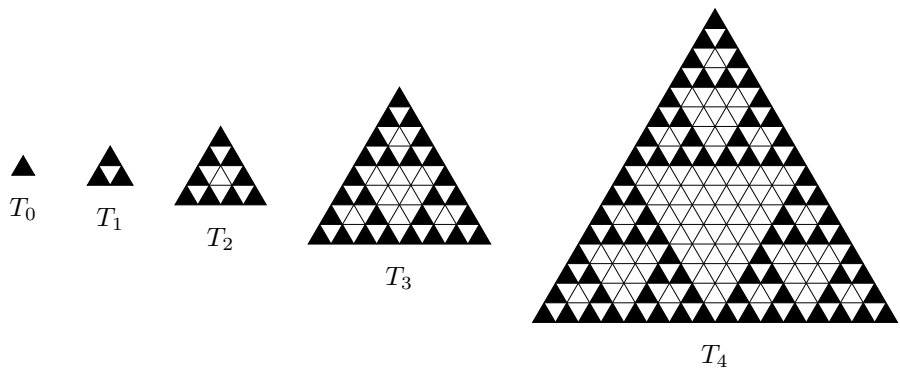
\begin{figure}[ht]
\centering\small
\begin{tabular}{*{5}{c}}
	\begin{tikzpicture}[scale = 0.3, baseline={([yshift=-.8ex]current bounding box.center)}]
\draw[ultra thin] (0.000, 0.000) -- (1.000, 0.000) -- (0.500, 0.866) -- cycle;
\fill[black] (0.500, 0.866) -- (0.000, 0.000) -- (1.000, 0.000) -- cycle;
\draw (0.500, -1.500) node{$T_0$};
\end{tikzpicture} & 
	\begin{tikzpicture}[scale = 0.3, baseline={([yshift=-.8ex]current bounding box.center)}]
\draw ( 0.000,  0.000) -- ( 2.000,  0.000) -- ( 1.000,  1.732) -- cycle;
\draw[ultra thin] ( 0.500,  0.866) -- ( 1.500,  0.866);
\draw[ultra thin] ( 0.500,  0.866) -- ( 1.000,  0.000);
\draw[ultra thin] ( 1.000,  0.000) -- ( 1.500,  0.866);
\fill[black] ( 1.000,  1.732) -- ( 0.500,  0.866) -- ( 1.500,  0.866) -- cycle;
\fill[black] ( 0.500,  0.866) -- ( 0.000,  0.000) -- ( 1.000,  0.000) -- cycle;
\fill[black] ( 1.500,  0.866) -- ( 1.000,  0.000) -- ( 2.000,  0.000) -- cycle;
\draw (1.000, -1.500) node{$T_1$};
\end{tikzpicture} & 
	\begin{tikzpicture}[scale = 0.3, baseline={([yshift=-.8ex]current bounding box.center)}]
\draw ( 0.000,  0.000) -- ( 4.000,  0.000) -- ( 2.000,  3.464) -- cycle;
\draw[ultra thin] ( 0.500,  0.866) -- ( 3.500,  0.866);
\draw[ultra thin] ( 1.000,  1.732) -- ( 3.000,  1.732);
\draw[ultra thin] ( 1.500,  2.598) -- ( 2.500,  2.598);
\draw[ultra thin] ( 0.500,  0.866) -- ( 1.000,  0.000);
\draw[ultra thin] ( 1.000,  1.732) -- ( 2.000,  0.000);
\draw[ultra thin] ( 1.500,  2.598) -- ( 3.000,  0.000);
\draw[ultra thin] ( 1.000,  0.000) -- ( 2.500,  2.598);
\draw[ultra thin] ( 2.000,  0.000) -- ( 3.000,  1.732);
\draw[ultra thin] ( 3.000,  0.000) -- ( 3.500,  0.866);
\fill[black] ( 2.000,  3.464) -- ( 1.500,  2.598) -- ( 2.500,  2.598) -- cycle;
\fill[black] ( 1.500,  2.598) -- ( 1.000,  1.732) -- ( 2.000,  1.732) -- cycle;
\fill[black] ( 2.500,  2.598) -- ( 2.000,  1.732) -- ( 3.000,  1.732) -- cycle;
\fill[black] ( 1.000,  1.732) -- ( 0.500,  0.866) -- ( 1.500,  0.866) -- cycle;
\fill[black] ( 3.000,  1.732) -- ( 2.500,  0.866) -- ( 3.500,  0.866) -- cycle;
\fill[black] ( 0.500,  0.866) -- ( 0.000,  0.000) -- ( 1.000,  0.000) -- cycle;
\fill[black] ( 1.500,  0.866) -- ( 1.000,  0.000) -- ( 2.000,  0.000) -- cycle;
\fill[black] ( 2.500,  0.866) -- ( 2.000,  0.000) -- ( 3.000,  0.000) -- cycle;
\fill[black] ( 3.500,  0.866) -- ( 3.000,  0.000) -- ( 4.000,  0.000) -- cycle;
\draw (2.000, -1.500) node{$T_2$};
\end{tikzpicture} &
	\begin{tikzpicture}[scale = 0.3, baseline={([yshift=-.8ex]current bounding box.center)}]
\draw ( 0.000,  0.000) -- ( 8.000,  0.000) -- ( 4.000,  6.928) -- cycle;
\draw[ultra thin] ( 0.500,  0.866) -- ( 7.500,  0.866);
\draw[ultra thin] ( 1.000,  1.732) -- ( 7.000,  1.732);
\draw[ultra thin] ( 1.500,  2.598) -- ( 6.500,  2.598);
\draw[ultra thin] ( 2.000,  3.464) -- ( 6.000,  3.464);
\draw[ultra thin] ( 2.500,  4.330) -- ( 5.500,  4.330);
\draw[ultra thin] ( 3.000,  5.196) -- ( 5.000,  5.196);
\draw[ultra thin] ( 3.500,  6.062) -- ( 4.500,  6.062);
\draw[ultra thin] ( 0.500,  0.866) -- ( 1.000,  0.000);
\draw[ultra thin] ( 1.000,  1.732) -- ( 2.000,  0.000);
\draw[ultra thin] ( 1.500,  2.598) -- ( 3.000,  0.000);
\draw[ultra thin] ( 2.000,  3.464) -- ( 4.000,  0.000);
\draw[ultra thin] ( 2.500,  4.330) -- ( 5.000,  0.000);
\draw[ultra thin] ( 3.000,  5.196) -- ( 6.000,  0.000);
\draw[ultra thin] ( 3.500,  6.062) -- ( 7.000,  0.000);
\draw[ultra thin] ( 1.000,  0.000) -- ( 4.500,  6.062);
\draw[ultra thin] ( 2.000,  0.000) -- ( 5.000,  5.196);
\draw[ultra thin] ( 3.000,  0.000) -- ( 5.500,  4.330);
\draw[ultra thin] ( 4.000,  0.000) -- ( 6.000,  3.464);
\draw[ultra thin] ( 5.000,  0.000) -- ( 6.500,  2.598);
\draw[ultra thin] ( 6.000,  0.000) -- ( 7.000,  1.732);
\draw[ultra thin] ( 7.000,  0.000) -- ( 7.500,  0.866);
\fill[black] ( 4.000,  6.928) -- ( 3.500,  6.062) -- ( 4.500,  6.062) -- cycle;
\fill[black] ( 3.500,  6.062) -- ( 3.000,  5.196) -- ( 4.000,  5.196) -- cycle;
\fill[black] ( 4.500,  6.062) -- ( 4.000,  5.196) -- ( 5.000,  5.196) -- cycle;
\fill[black] ( 3.000,  5.196) -- ( 2.500,  4.330) -- ( 3.500,  4.330) -- cycle;
\fill[black] ( 5.000,  5.196) -- ( 4.500,  4.330) -- ( 5.500,  4.330) -- cycle;
\fill[black] ( 2.500,  4.330) -- ( 2.000,  3.464) -- ( 3.000,  3.464) -- cycle;
\fill[black] ( 3.500,  4.330) -- ( 3.000,  3.464) -- ( 4.000,  3.464) -- cycle;
\fill[black] ( 4.500,  4.330) -- ( 4.000,  3.464) -- ( 5.000,  3.464) -- cycle;
\fill[black] ( 5.500,  4.330) -- ( 5.000,  3.464) -- ( 6.000,  3.464) -- cycle;
\fill[black] ( 2.000,  3.464) -- ( 1.500,  2.598) -- ( 2.500,  2.598) -- cycle;
\fill[black] ( 6.000,  3.464) -- ( 5.500,  2.598) -- ( 6.500,  2.598) -- cycle;
\fill[black] ( 1.500,  2.598) -- ( 1.000,  1.732) -- ( 2.000,  1.732) -- cycle;
\fill[black] ( 2.500,  2.598) -- ( 2.000,  1.732) -- ( 3.000,  1.732) -- cycle;
\fill[black] ( 5.500,  2.598) -- ( 5.000,  1.732) -- ( 6.000,  1.732) -- cycle;
\fill[black] ( 6.500,  2.598) -- ( 6.000,  1.732) -- ( 7.000,  1.732) -- cycle;
\fill[black] ( 1.000,  1.732) -- ( 0.500,  0.866) -- ( 1.500,  0.866) -- cycle;
\fill[black] ( 3.000,  1.732) -- ( 2.500,  0.866) -- ( 3.500,  0.866) -- cycle;
\fill[black] ( 5.000,  1.732) -- ( 4.500,  0.866) -- ( 5.500,  0.866) -- cycle;
\fill[black] ( 7.000,  1.732) -- ( 6.500,  0.866) -- ( 7.500,  0.866) -- cycle;
\fill[black] ( 0.500,  0.866) -- ( 0.000,  0.000) -- ( 1.000,  0.000) -- cycle;
\fill[black] ( 1.500,  0.866) -- ( 1.000,  0.000) -- ( 2.000,  0.000) -- cycle;
\fill[black] ( 2.500,  0.866) -- ( 2.000,  0.000) -- ( 3.000,  0.000) -- cycle;
\fill[black] ( 3.500,  0.866) -- ( 3.000,  0.000) -- ( 4.000,  0.000) -- cycle;
\fill[black] ( 4.500,  0.866) -- ( 4.000,  0.000) -- ( 5.000,  0.000) -- cycle;
\fill[black] ( 5.500,  0.866) -- ( 5.000,  0.000) -- ( 6.000,  0.000) -- cycle;
\fill[black] ( 6.500,  0.866) -- ( 6.000,  0.000) -- ( 7.000,  0.000) -- cycle;
\fill[black] ( 7.500,  0.866) -- ( 7.000,  0.000) -- ( 8.000,  0.000) -- cycle;
\draw (4.000, -1.500) node{$T_3$};
\end{tikzpicture} &
	\begin{tikzpicture}[scale = 0.3, baseline={([yshift=-.8ex]current bounding box.center)}]
\draw ( 0.000,  0.000) -- (16.000,  0.000) -- ( 8.000, 13.856) -- cycle;
\draw[ultra thin] ( 0.500,  0.866) -- (15.500,  0.866);
\draw[ultra thin] ( 1.000,  1.732) -- (15.000,  1.732);
\draw[ultra thin] ( 1.500,  2.598) -- (14.500,  2.598);
\draw[ultra thin] ( 2.000,  3.464) -- (14.000,  3.464);
\draw[ultra thin] ( 2.500,  4.330) -- (13.500,  4.330);
\draw[ultra thin] ( 3.000,  5.196) -- (13.000,  5.196);
\draw[ultra thin] ( 3.500,  6.062) -- (12.500,  6.062);
\draw[ultra thin] ( 4.000,  6.928) -- (12.000,  6.928);
\draw[ultra thin] ( 4.500,  7.794) -- (11.500,  7.794);
\draw[ultra thin] ( 5.000,  8.660) -- (11.000,  8.660);
\draw[ultra thin] ( 5.500,  9.526) -- (10.500,  9.526);
\draw[ultra thin] ( 6.000, 10.392) -- (10.000, 10.392);
\draw[ultra thin] ( 6.500, 11.258) -- ( 9.500, 11.258);
\draw[ultra thin] ( 7.000, 12.124) -- ( 9.000, 12.124);
\draw[ultra thin] ( 7.500, 12.990) -- ( 8.500, 12.990);
\draw[ultra thin] ( 0.500,  0.866) -- ( 1.000,  0.000);
\draw[ultra thin] ( 1.000,  1.732) -- ( 2.000,  0.000);
\draw[ultra thin] ( 1.500,  2.598) -- ( 3.000,  0.000);
\draw[ultra thin] ( 2.000,  3.464) -- ( 4.000,  0.000);
\draw[ultra thin] ( 2.500,  4.330) -- ( 5.000,  0.000);
\draw[ultra thin] ( 3.000,  5.196) -- ( 6.000,  0.000);
\draw[ultra thin] ( 3.500,  6.062) -- ( 7.000,  0.000);
\draw[ultra thin] ( 4.000,  6.928) -- ( 8.000,  0.000);
\draw[ultra thin] ( 4.500,  7.794) -- ( 9.000,  0.000);
\draw[ultra thin] ( 5.000,  8.660) -- (10.000,  0.000);
\draw[ultra thin] ( 5.500,  9.526) -- (11.000,  0.000);
\draw[ultra thin] ( 6.000, 10.392) -- (12.000,  0.000);
\draw[ultra thin] ( 6.500, 11.258) -- (13.000,  0.000);
\draw[ultra thin] ( 7.000, 12.124) -- (14.000,  0.000);
\draw[ultra thin] ( 7.500, 12.990) -- (15.000,  0.000);
\draw[ultra thin] ( 1.000,  0.000) -- ( 8.500, 12.990);
\draw[ultra thin] ( 2.000,  0.000) -- ( 9.000, 12.124);
\draw[ultra thin] ( 3.000,  0.000) -- ( 9.500, 11.258);
\draw[ultra thin] ( 4.000,  0.000) -- (10.000, 10.392);
\draw[ultra thin] ( 5.000,  0.000) -- (10.500,  9.526);
\draw[ultra thin] ( 6.000,  0.000) -- (11.000,  8.660);
\draw[ultra thin] ( 7.000,  0.000) -- (11.500,  7.794);
\draw[ultra thin] ( 8.000,  0.000) -- (12.000,  6.928);
\draw[ultra thin] ( 9.000,  0.000) -- (12.500,  6.062);
\draw[ultra thin] (10.000,  0.000) -- (13.000,  5.196);
\draw[ultra thin] (11.000,  0.000) -- (13.500,  4.330);
\draw[ultra thin] (12.000,  0.000) -- (14.000,  3.464);
\draw[ultra thin] (13.000,  0.000) -- (14.500,  2.598);
\draw[ultra thin] (14.000,  0.000) -- (15.000,  1.732);
\draw[ultra thin] (15.000,  0.000) -- (15.500,  0.866);
\fill[black] ( 8.000, 13.856) -- ( 7.500, 12.990) -- ( 8.500, 12.990) -- cycle;
\fill[black] ( 7.500, 12.990) -- ( 7.000, 12.124) -- ( 8.000, 12.124) -- cycle;
\fill[black] ( 8.500, 12.990) -- ( 8.000, 12.124) -- ( 9.000, 12.124) -- cycle;
\fill[black] ( 7.000, 12.124) -- ( 6.500, 11.258) -- ( 7.500, 11.258) -- cycle;
\fill[black] ( 9.000, 12.124) -- ( 8.500, 11.258) -- ( 9.500, 11.258) -- cycle;
\fill[black] ( 6.500, 11.258) -- ( 6.000, 10.392) -- ( 7.000, 10.392) -- cycle;
\fill[black] ( 7.500, 11.258) -- ( 7.000, 10.392) -- ( 8.000, 10.392) -- cycle;
\fill[black] ( 8.500, 11.258) -- ( 8.000, 10.392) -- ( 9.000, 10.392) -- cycle;
\fill[black] ( 9.500, 11.258) -- ( 9.000, 10.392) -- (10.000, 10.392) -- cycle;
\fill[black] ( 6.000, 10.392) -- ( 5.500,  9.526) -- ( 6.500,  9.526) -- cycle;
\fill[black] (10.000, 10.392) -- ( 9.500,  9.526) -- (10.500,  9.526) -- cycle;
\fill[black] ( 5.500,  9.526) -- ( 5.000,  8.660) -- ( 6.000,  8.660) -- cycle;
\fill[black] ( 6.500,  9.526) -- ( 6.000,  8.660) -- ( 7.000,  8.660) -- cycle;
\fill[black] ( 9.500,  9.526) -- ( 9.000,  8.660) -- (10.000,  8.660) -- cycle;
\fill[black] (10.500,  9.526) -- (10.000,  8.660) -- (11.000,  8.660) -- cycle;
\fill[black] ( 5.000,  8.660) -- ( 4.500,  7.794) -- ( 5.500,  7.794) -- cycle;
\fill[black] ( 7.000,  8.660) -- ( 6.500,  7.794) -- ( 7.500,  7.794) -- cycle;
\fill[black] ( 9.000,  8.660) -- ( 8.500,  7.794) -- ( 9.500,  7.794) -- cycle;
\fill[black] (11.000,  8.660) -- (10.500,  7.794) -- (11.500,  7.794) -- cycle;
\fill[black] ( 4.500,  7.794) -- ( 4.000,  6.928) -- ( 5.000,  6.928) -- cycle;
\fill[black] ( 5.500,  7.794) -- ( 5.000,  6.928) -- ( 6.000,  6.928) -- cycle;
\fill[black] ( 6.500,  7.794) -- ( 6.000,  6.928) -- ( 7.000,  6.928) -- cycle;
\fill[black] ( 7.500,  7.794) -- ( 7.000,  6.928) -- ( 8.000,  6.928) -- cycle;
\fill[black] ( 8.500,  7.794) -- ( 8.000,  6.928) -- ( 9.000,  6.928) -- cycle;
\fill[black] ( 9.500,  7.794) -- ( 9.000,  6.928) -- (10.000,  6.928) -- cycle;
\fill[black] (10.500,  7.794) -- (10.000,  6.928) -- (11.000,  6.928) -- cycle;
\fill[black] (11.500,  7.794) -- (11.000,  6.928) -- (12.000,  6.928) -- cycle;
\fill[black] ( 4.000,  6.928) -- ( 3.500,  6.062) -- ( 4.500,  6.062) -- cycle;
\fill[black] (12.000,  6.928) -- (11.500,  6.062) -- (12.500,  6.062) -- cycle;
\fill[black] ( 3.500,  6.062) -- ( 3.000,  5.196) -- ( 4.000,  5.196) -- cycle;
\fill[black] ( 4.500,  6.062) -- ( 4.000,  5.196) -- ( 5.000,  5.196) -- cycle;
\fill[black] (11.500,  6.062) -- (11.000,  5.196) -- (12.000,  5.196) -- cycle;
\fill[black] (12.500,  6.062) -- (12.000,  5.196) -- (13.000,  5.196) -- cycle;
\fill[black] ( 3.000,  5.196) -- ( 2.500,  4.330) -- ( 3.500,  4.330) -- cycle;
\fill[black] ( 5.000,  5.196) -- ( 4.500,  4.330) -- ( 5.500,  4.330) -- cycle;
\fill[black] (11.000,  5.196) -- (10.500,  4.330) -- (11.500,  4.330) -- cycle;
\fill[black] (13.000,  5.196) -- (12.500,  4.330) -- (13.500,  4.330) -- cycle;
\fill[black] ( 2.500,  4.330) -- ( 2.000,  3.464) -- ( 3.000,  3.464) -- cycle;
\fill[black] ( 3.500,  4.330) -- ( 3.000,  3.464) -- ( 4.000,  3.464) -- cycle;
\fill[black] ( 4.500,  4.330) -- ( 4.000,  3.464) -- ( 5.000,  3.464) -- cycle;
\fill[black] ( 5.500,  4.330) -- ( 5.000,  3.464) -- ( 6.000,  3.464) -- cycle;
\fill[black] (10.500,  4.330) -- (10.000,  3.464) -- (11.000,  3.464) -- cycle;
\fill[black] (11.500,  4.330) -- (11.000,  3.464) -- (12.000,  3.464) -- cycle;
\fill[black] (12.500,  4.330) -- (12.000,  3.464) -- (13.000,  3.464) -- cycle;
\fill[black] (13.500,  4.330) -- (13.000,  3.464) -- (14.000,  3.464) -- cycle;
\fill[black] ( 2.000,  3.464) -- ( 1.500,  2.598) -- ( 2.500,  2.598) -- cycle;
\fill[black] ( 6.000,  3.464) -- ( 5.500,  2.598) -- ( 6.500,  2.598) -- cycle;
\fill[black] (10.000,  3.464) -- ( 9.500,  2.598) -- (10.500,  2.598) -- cycle;
\fill[black] (14.000,  3.464) -- (13.500,  2.598) -- (14.500,  2.598) -- cycle;
\fill[black] ( 1.500,  2.598) -- ( 1.000,  1.732) -- ( 2.000,  1.732) -- cycle;
\fill[black] ( 2.500,  2.598) -- ( 2.000,  1.732) -- ( 3.000,  1.732) -- cycle;
\fill[black] ( 5.500,  2.598) -- ( 5.000,  1.732) -- ( 6.000,  1.732) -- cycle;
\fill[black] ( 6.500,  2.598) -- ( 6.000,  1.732) -- ( 7.000,  1.732) -- cycle;
\fill[black] ( 9.500,  2.598) -- ( 9.000,  1.732) -- (10.000,  1.732) -- cycle;
\fill[black] (10.500,  2.598) -- (10.000,  1.732) -- (11.000,  1.732) -- cycle;
\fill[black] (13.500,  2.598) -- (13.000,  1.732) -- (14.000,  1.732) -- cycle;
\fill[black] (14.500,  2.598) -- (14.000,  1.732) -- (15.000,  1.732) -- cycle;
\fill[black] ( 1.000,  1.732) -- ( 0.500,  0.866) -- ( 1.500,  0.866) -- cycle;
\fill[black] ( 3.000,  1.732) -- ( 2.500,  0.866) -- ( 3.500,  0.866) -- cycle;
\fill[black] ( 5.000,  1.732) -- ( 4.500,  0.866) -- ( 5.500,  0.866) -- cycle;
\fill[black] ( 7.000,  1.732) -- ( 6.500,  0.866) -- ( 7.500,  0.866) -- cycle;
\fill[black] ( 9.000,  1.732) -- ( 8.500,  0.866) -- ( 9.500,  0.866) -- cycle;
\fill[black] (11.000,  1.732) -- (10.500,  0.866) -- (11.500,  0.866) -- cycle;
\fill[black] (13.000,  1.732) -- (12.500,  0.866) -- (13.500,  0.866) -- cycle;
\fill[black] (15.000,  1.732) -- (14.500,  0.866) -- (15.500,  0.866) -- cycle;
\fill[black] ( 0.500,  0.866) -- ( 0.000,  0.000) -- ( 1.000,  0.000) -- cycle;
\fill[black] ( 1.500,  0.866) -- ( 1.000,  0.000) -- ( 2.000,  0.000) -- cycle;
\fill[black] ( 2.500,  0.866) -- ( 2.000,  0.000) -- ( 3.000,  0.000) -- cycle;
\fill[black] ( 3.500,  0.866) -- ( 3.000,  0.000) -- ( 4.000,  0.000) -- cycle;
\fill[black] ( 4.500,  0.866) -- ( 4.000,  0.000) -- ( 5.000,  0.000) -- cycle;
\fill[black] ( 5.500,  0.866) -- ( 5.000,  0.000) -- ( 6.000,  0.000) -- cycle;
\fill[black] ( 6.500,  0.866) -- ( 6.000,  0.000) -- ( 7.000,  0.000) -- cycle;
\fill[black] ( 7.500,  0.866) -- ( 7.000,  0.000) -- ( 8.000,  0.000) -- cycle;
\fill[black] ( 8.500,  0.866) -- ( 8.000,  0.000) -- ( 9.000,  0.000) -- cycle;
\fill[black] ( 9.500,  0.866) -- ( 9.000,  0.000) -- (10.000,  0.000) -- cycle;
\fill[black] (10.500,  0.866) -- (10.000,  0.000) -- (11.000,  0.000) -- cycle;
\fill[black] (11.500,  0.866) -- (11.000,  0.000) -- (12.000,  0.000) -- cycle;
\fill[black] (12.500,  0.866) -- (12.000,  0.000) -- (13.000,  0.000) -- cycle;
\fill[black] (13.500,  0.866) -- (13.000,  0.000) -- (14.000,  0.000) -- cycle;
\fill[black] (14.500,  0.866) -- (14.000,  0.000) -- (15.000,  0.000) -- cycle;
\fill[black] (15.500,  0.866) -- (15.000,  0.000) -- (16.000,  0.000) -- cycle;
\draw (8.000, -1.500) node{$T_4$};
\end{tikzpicture} 
\end{tabular}
\caption{The first iterations of the substitution $\mu$ on \mytikz{0.07}{\protect\filldraw (0.000, 0.000) -- (4.000, 0.000) -- (2.000,3.464) -- cycle;}, generating the Sierpi\'{n}ski triangle, where 
$T_n = \mu^n(\mytikz{0.07}{\protect\filldraw (0.000, 0.000) -- (4.000, 0.000) -- (2.000,3.464) -- cycle;})$.}
\label{fig:FirstIterationsOfMu}
\end{figure}

\begin{theorem}
\label{thm:main}
Let $A_n$ be the number of unique upwards oriented triangular patterns of side length $n$ that occur in the Sierpi\'{n}ski triangle $T$. Then 
\begin{equation}
	\label{eq:MainUpwards}
	A_n = 4n^2 - 6n + 4
\end{equation}
for $n \geq 1$. Similarly, let $A'_n$ be the number of downwards oriented triangular patterns of side length $n$ that occur in $T$. Then $A'_1=1$, and 
\begin{equation}
	\label{eq:MainDownwards}
	A'_n = n^2 - 3n + 4
\end{equation}
for $n \geq 2$. 
\end{theorem}

The value given by $A'_n$ in Theorem~\ref{thm:main} is also the maximal number of regions obtained when dividing the plane with $n-1$ circles, see sequences \texttt{A014206}, and \texttt{A386480}  in OEIS \cite{oeis}, and also \cite{parabola}. If there is a direct connection between this division of the plane and the patterns of $T$ is at the moment not known to the author. 

Example of similar results to Theorem~\ref{thm:main} are Allouche's result of the pattern complexity in paper-folding sequences \cite{allouche}, Nilsson's generalisation of the paper-folding structures into 2 dimensions \cite{nilssonPaperfoling}, and properties of the squiral tiling given in \cite{nilssonSquiral}.

The outline of this paper is as follows. In the next section we give definitions and provide some initial results. Thereafter, we turn to looking at properties of sets of patterns and give some tools for enumerating such sets. Succeeding this, in section~\ref{sec:recursions} we give a list of recursions describing the size of sets of patterns, and in the final section we tie everything up and prove Theorem~\ref{thm:main}.

\section{Preliminaries}

In this section we give basis notations and definitions, we will state and prove a couple of some initial results. 

Recall the definition of the substitution $\mu$ in \eqref{eq:DefOfMu}. We use the notation $\mu^n := \mu \circ \mu^{n-1}$ for $n\geq1$ and $\mu^{0} = Id$. An element, or structure, of the form $\mu^n(x)$ where 
$ 
x \in \{ 
\mytikz{0.08}{\trinone},
\mytikz{0.08}{\tridown},
\mytikz{0.08}{\filldraw (0.000, 0.000) -- (4.000, 0.000) -- (2.000,3.464) -- cycle;}
\}
$ is called a \emph{super-tile}. We define the special super-tiles $T_n$ by $T_n := \mu^n(\mytikz{0.08}{\filldraw (0.000, 0.000) -- (4.000, 0.000) -- (2.000,3.464) -- cycle;})$ for $n\geq 0$. From the definition of $\mu$ in \eqref{eq:DefOfMu} we see that $T_n$ can be given as a recursive block substitution 
\begin{equation}
\label{eq:DefTnBlock}
T_{n+1} 
= \mu^{n+1}(\mytikz{0.08}{\filldraw (0.000, 0.000) -- (4.000, 0.000) -- (2.000,3.464) -- cycle;})
= \mytikz{1.6}{
	\draw (0.000, 0.000) -- (2.000, 0.000) -- (1.000, 1.732) -- cycle;
	\draw[line join=round] (1.000, 0.000) -- (1.500, 0.866) -- (0.500, 0.866) -- cycle;
	\draw(0.5,0.289) node{\small$T_{n}$};
	\draw(1.5,0.289) node{\small$T_{n}$};
	\draw(1.0,1.155) node{\small$T_{n}$};
	\draw(1.0,0.866) node[below]{\small$\mu^n($\begin{tikzpicture}[scale=0.07] \tridown \end{tikzpicture}$)$};
}
\end{equation}
See Figure~\ref{fig:FirstIterationsOfMu} for a visualisation of the first $T_n$s. The limit of the sequence of the $T_n$s the Sierpi\'{n}ski triangle and we denote it by $T$, as also already mentioned in the introduction. 

By the notion \emph{pattern} we shall mean a triangular region without holes that occur somewhere in $T$. We also say that $T$ is a pattern (an infinite one). Note that a we do not cut the triangles of unit side length to create a pattern. Hence, any finite pattern is an equilateral triangle. A pattern with $n$ rows of unit triangles is said to be of size $n$. We consider two main kind of patterns, the patterns that have precisely one corner at the top and those with precisely one at the bottom. We say the former are upwards oriented and the latter downwards oriented. 

The unit length triangles in a pattern $x$ can be indexed via a pair $(r,c)$ where $r$ is the row counted from the top starting with 0, and $c$ the column counted from the left starting with 0. We allow us to loosen the definition of patterns to also include an equilateral region where we cut off one unit length triangle at one or more corners. The unit length triangles in such a pattern are indexed in the same way as in a pattern with all corners uncut. That is, if we cut off the top triangle in an upwards oriented pattern $p$ then the first row in $p$ will have row-number 1. Note that we say that a pattern is of size $n$ even if we have cut off a corner, see Figure~\ref{fig:indexingpatterns}.

\begin{figure}[ht]
\centering
\begin{tabular}{ccc}
\begin{tikzpicture}[scale = 0.65, baseline={([yshift=-.8ex]current bounding box.center)}]
\draw (0.000, 0.000) -- (5.000, 0.000) -- (2.500, 4.330) -- cycle;
\draw (0.500, 0.866) -- (1.000, 0.000);
\draw (3.000, 3.464) -- (1.000, 0.000);
\draw (0.500, 0.866) -- (4.500, 0.866);
\draw (1.000, 1.732) -- (2.000, 0.000);
\draw (3.500, 2.598) -- (2.000, 0.000);
\draw (1.000, 1.732) -- (4.000, 1.732);
\draw (1.500, 2.598) -- (3.000, 0.000);
\draw (4.000, 1.732) -- (3.000, 0.000);
\draw (1.500, 2.598) -- (3.500, 2.598);
\draw (2.000, 3.464) -- (4.000, 0.000);
\draw (4.500, 0.866) -- (4.000, 0.000);
\draw (2.000, 3.464) -- (3.000, 3.464);
\draw (2.500, 3.811) node{\footnotesize$0$};
\draw (2.000, 2.944) node{\footnotesize$0$};
\draw (2.500, 3.118) node{\footnotesize$1$};
\draw (3.000, 2.944) node{\footnotesize$2$};
\draw (1.500, 2.078) node{\footnotesize$0$};
\draw (2.000, 2.252) node{\footnotesize$1$};
\draw (2.500, 2.078) node{\footnotesize$2$};
\draw (3.000, 2.252) node{\footnotesize$3$};
\draw (3.500, 2.078) node{\footnotesize$4$};
\draw (1.000, 1.212) node{\footnotesize$0$};
\draw (1.500, 1.386) node{\footnotesize$1$};
\draw (2.000, 1.212) node{\footnotesize$2$};
\draw (2.500, 1.386) node{\footnotesize$3$};
\draw (3.000, 1.212) node{\footnotesize$4$};
\draw (3.500, 1.386) node{\footnotesize$5$};
\draw (4.000, 1.212) node{\footnotesize$6$};
\draw (0.500, 0.346) node{\footnotesize$0$};
\draw (1.000, 0.520) node{\footnotesize$1$};
\draw (1.500, 0.346) node{\footnotesize$2$};
\draw (2.000, 0.520) node{\footnotesize$3$};
\draw (2.500, 0.346) node{\footnotesize$4$};
\draw (3.000, 0.520) node{\footnotesize$5$};
\draw (3.500, 0.346) node{\footnotesize$6$};
\draw (4.000, 0.520) node{\footnotesize$7$};
\draw (4.500, 0.346) node{\footnotesize$8$};
\draw (-0.500, 0.433) node{\footnotesize$4$};
\draw (0.000, 1.299) node{\footnotesize$3$};
\draw (0.500, 2.165) node{\footnotesize$2$};
\draw (1.000, 3.031) node{\footnotesize$1$};
\draw (1.500, 3.897) node{\footnotesize$0$};
\draw (2.500, -1.000) node{\small(a)};
\end{tikzpicture}
&
\begin{tikzpicture}[scale = 0.65, baseline={([yshift=-.8ex]current bounding box.center)}]
\draw (0.000, 4.330) -- (5.000, 4.330) -- (2.500, 0.000) -- cycle;
\draw (1.000, 4.330) -- (3.000, 0.866);
\draw (1.000, 4.330) -- (0.500, 3.464);
\draw (2.000, 0.866) -- (3.000, 0.866);
\draw (2.000, 4.330) -- (3.500, 1.732);
\draw (2.000, 4.330) -- (1.000, 2.598);
\draw (1.500, 1.732) -- (3.500, 1.732);
\draw (3.000, 4.330) -- (4.000, 2.598);
\draw (3.000, 4.330) -- (1.500, 1.732);
\draw (1.000, 2.598) -- (4.000, 2.598);
\draw (4.000, 4.330) -- (4.500, 3.464);
\draw (4.000, 4.330) -- (2.000, 0.866);
\draw (0.500, 3.464) -- (4.500, 3.464);
\draw (2.500, 0.520) node{\footnotesize$0$};
\draw (2.000, 1.386) node{\footnotesize$0$};
\draw (2.500, 1.212) node{\footnotesize$1$};
\draw (3.000, 1.386) node{\footnotesize$2$};
\draw (1.500, 2.252) node{\footnotesize$0$};
\draw (2.000, 2.078) node{\footnotesize$1$};
\draw (2.500, 2.252) node{\footnotesize$2$};
\draw (3.000, 2.078) node{\footnotesize$3$};
\draw (3.500, 2.252) node{\footnotesize$4$};
\draw (1.000, 3.118) node{\footnotesize$0$};
\draw (1.500, 2.944) node{\footnotesize$1$};
\draw (2.000, 3.118) node{\footnotesize$2$};
\draw (2.500, 2.944) node{\footnotesize$3$};
\draw (3.000, 3.118) node{\footnotesize$4$};
\draw (3.500, 2.944) node{\footnotesize$5$};
\draw (4.000, 3.118) node{\footnotesize$6$};
\draw (0.500, 3.984) node{\footnotesize$0$};
\draw (1.000, 3.811) node{\footnotesize$1$};
\draw (1.500, 3.984) node{\footnotesize$2$};
\draw (2.000, 3.811) node{\footnotesize$3$};
\draw (2.500, 3.984) node{\footnotesize$4$};
\draw (3.000, 3.811) node{\footnotesize$5$};
\draw (3.500, 3.984) node{\footnotesize$6$};
\draw (4.000, 3.811) node{\footnotesize$7$};
\draw (4.500, 3.984) node{\footnotesize$8$};
\draw (-0.500, 3.897) node{\footnotesize$0$};
\draw (0.000, 3.031) node{\footnotesize$1$};
\draw (0.500, 2.165) node{\footnotesize$2$};
\draw (1.000, 1.299) node{\footnotesize$3$};
\draw (1.500, 0.433) node{\footnotesize$4$};
\draw (2.500, -1.000) node{\small(b)};
\end{tikzpicture}
&
\begin{tikzpicture}[scale = 0.65, baseline={([yshift=-.8ex]current bounding box.center)}]
\draw (0.500, 0.866) -- (2.000, 3.464) -- (3.000, 3.464) -- (4.500, 0.866) -- (4.000, 0.000) -- (1.000, 0.000) -- cycle;
\draw (3.000, 3.464) -- (1.000, 0.000);
\draw (0.500, 0.866) -- (4.500, 0.866);
\draw (1.000, 1.732) -- (2.000, 0.000);
\draw (3.500, 2.598) -- (2.000, 0.000);
\draw (1.000, 1.732) -- (4.000, 1.732);
\draw (1.500, 2.598) -- (3.000, 0.000);
\draw (4.000, 1.732) -- (3.000, 0.000);
\draw (1.500, 2.598) -- (3.500, 2.598);
\draw (2.000, 3.464) -- (4.000, 0.000);
\draw (2.000, 2.944) node{\footnotesize$0$};
\draw (2.500, 3.118) node{\footnotesize$1$};
\draw (3.000, 2.944) node{\footnotesize$2$};
\draw (1.500, 2.078) node{\footnotesize$0$};
\draw (2.000, 2.252) node{\footnotesize$1$};
\draw (2.500, 2.078) node{\footnotesize$2$};
\draw (3.000, 2.252) node{\footnotesize$3$};
\draw (3.500, 2.078) node{\footnotesize$4$};
\draw (1.000, 1.212) node{\footnotesize$0$};
\draw (1.500, 1.386) node{\footnotesize$1$};
\draw (2.000, 1.212) node{\footnotesize$2$};
\draw (2.500, 1.386) node{\footnotesize$3$};
\draw (3.000, 1.212) node{\footnotesize$4$};
\draw (3.500, 1.386) node{\footnotesize$5$};
\draw (4.000, 1.212) node{\footnotesize$6$};
\draw (1.000, 0.520) node{\footnotesize$1$};
\draw (1.500, 0.346) node{\footnotesize$2$};
\draw (2.000, 0.520) node{\footnotesize$3$};
\draw (2.500, 0.346) node{\footnotesize$4$};
\draw (3.000, 0.520) node{\footnotesize$5$};
\draw (3.500, 0.346) node{\footnotesize$6$};
\draw (4.000, 0.520) node{\footnotesize$7$};
\draw (-0.500, 0.433) node{\footnotesize$4$};
\draw (0.000, 1.299) node{\footnotesize$3$};
\draw (0.500, 2.165) node{\footnotesize$2$};
\draw (1.000, 3.031) node{\footnotesize$1$};
\draw (2.500, -1.000) node{\small(c)};
\clip (0.000, 0.000) rectangle (5.000, 4.330);
\end{tikzpicture}
\end{tabular}

\caption{Indexing unit length triangles in patterns of size 5. (a) Indexing an upwards oriented pattern. (b) Indexing a downwards oriented pattern. (c) Indexing an upwards oriented pattern where corners are cut off.}
\label{fig:indexingpatterns}
\end{figure}
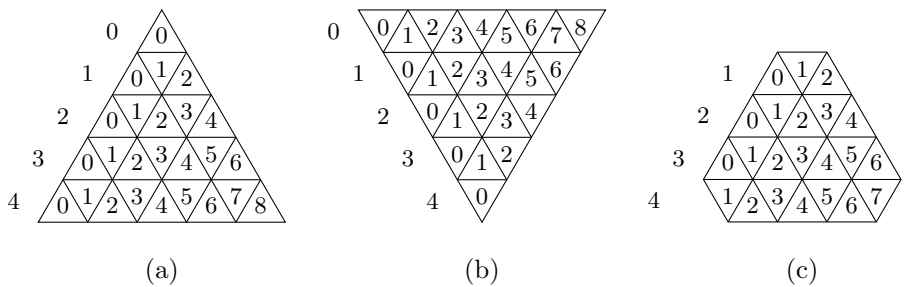

Let $x$ be a pattern. Then the notation $u:=x[\alpha,r,c,n]$ denotes the sub\-pattern $u$ of $x$ that has its top row and leftmost column at row $r$ and column $c$ in $x$, is of size $n$, and where
\begin{equation}
\label{eq:defOfTypes}
\alpha \in \mathcal{A} := \left\{ 
\mytikz{0.1}{\trinone}, 
\mytikz{0.1}{\tritop},
\mytikz{0.1}{\trilower},
\mytikz{0.1}{\triall},
\mytikz{0.1}{\tridown} \,
\right\}
\end{equation}
symbolises the kind of subpattern we are considering; upwards or downwards oriented, and with or without cut corners. See Figure~\ref{figure:defofsubpattern}. Note that we can not define a subpattern for all indexes $(r,c)$; for example if $x$ is an upward oriented pattern we can not define an upward oriented subpattern $u$ of $x$ for an odd $c$.

\begin{figure}[ht]
\centering
\begin{tikzpicture}[scale = 0.65, baseline={([yshift=-.8ex]current bounding box.center)}]
\fill[gray!40] (1.500, 0.866) -- (2.500, 2.598) -- (3.500, 2.598) -- (4.500, 0.866) --cycle;
\draw (0.000, 0.000) -- (5.000, 0.000) -- (2.500, 4.330) -- cycle;
\draw[thin] (0.500, 0.866) -- (4.500, 0.866);
\draw[thin] (1.000, 1.732) -- (4.000, 1.732);
\draw[thin] (1.500, 2.598) -- (3.500, 2.598);
\draw[thin] (2.000, 3.464) -- (3.000, 3.464);
\draw[thin] (0.500, 0.866) -- (1.000, 0.000);
\draw[thin] (1.000, 1.732) -- (2.000, 0.000);
\draw[thin] (1.500, 2.598) -- (3.000, 0.000);
\draw[thin] (2.000, 3.464) -- (4.000, 0.000);
\draw[thin] (1.000, 0.000) -- (3.000, 3.464);
\draw[thin] (2.000, 0.000) -- (3.500, 2.598);
\draw[thin] (3.000, 0.000) -- (4.000, 1.732);
\draw[thin] (4.000, 0.000) -- (4.500, 0.866);
\draw (2.500, 3.811) node{\small$0$};
\draw (2.000, 2.944) node{\small$0$};
\draw (2.500, 3.118) node{\small$1$};
\draw (3.000, 2.944) node{\small$2$};
\draw (1.500, 2.078) node{\small$0$};
\draw (2.000, 2.252) node{\small$1$};
\draw (2.500, 2.078) node{\small$2$};
\draw (3.000, 2.252) node{\small$3$};
\draw (3.500, 2.078) node{\small$4$};
\draw (1.000, 1.212) node{\small$0$};
\draw (1.500, 1.386) node{\small$1$};
\draw (2.000, 1.212) node{\small$2$};
\draw (2.500, 1.386) node{\small$3$};
\draw (3.000, 1.212) node{\small$4$};
\draw (3.500, 1.386) node{\small$5$};
\draw (4.000, 1.212) node{\small$6$};
\draw (0.500, 0.346) node{\small$0$};
\draw (1.000, 0.520) node{\small$1$};
\draw (1.500, 0.346) node{\small$2$};
\draw (2.000, 0.520) node{\small$3$};
\draw (2.500, 0.346) node{\small$4$};
\draw (3.000, 0.520) node{\small$5$};
\draw (3.500, 0.346) node{\small$6$};
\draw (4.000, 0.520) node{\small$7$};
\draw (4.500, 0.346) node{\small$8$};
\draw (1.500, 3.897) node{\small$0$};
\draw (1.000, 3.031) node{\small$1$};
\draw (0.500, 2.165) node{\small$2$};
\draw (0.000, 1.299) node{\small$3$};
\draw (-0.500, 0.433) node{\small$4$};
\draw (2.500, -0.750) node{$x$};
\node (A) at (4.000, 2.000) {};
\node (B) at (5.000, 2.598) [right] {$x[\mytikz{0.08}{\tritop},1,2,3]$};
\draw[<-, thick] (A) to[out=30, in=185] (B);
\end{tikzpicture}
\caption{A subpattern $u$ (gray shaded) of a pattern $x$ is denoted by its starting row, leftmost column, and size, that is; $u = x[\mytikz{0.07}{\protect\tritop},1,2,3]$. (Note that the row count includes the top triangle of $u$ that has been cut off.)} 
\label{figure:defofsubpattern}
\end{figure}
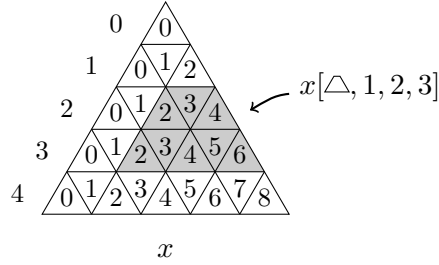

For $n\geq 1$ define 
\begin{equation}
\label{eq:defOfP}
	P(\alpha, T, n) := \left\{ T[\alpha, r,c,n]: r,c \in\mathbb{N}\right\}
\end{equation}
to be the set of all possible patterns of size $n$ and of type $\alpha \in \mathcal{A}$ that occurs somewhere in $T$ (and of course where such pattern is possible to define). From \eqref{eq:DefTnBlock} we now have.

\begin{lemma}
	\label{lemma:inclusion}
	Let $n\geq 0$. Then $T_n \in P(\mytikz{0.08}{\trinone}, T_{n+1}, 2^n)$.\hfill$\square$ 
\end{lemma}

The Lemma~\ref{lemma:inclusion} shows that the chain of nested sets of subpatterns,
\[
   P(\mytikz{0.08}{\trinone}, T_{0}, m) \subseteq \cdots \subseteq 
   P(\mytikz{0.08}{\trinone}, T_{n}, m) \subseteq 
   P(\mytikz{0.08}{\trinone}, T_{n+1}, m ) \subseteq \cdots, 
\]
is monotonic including in $n$, (if $m \leq 2^n$). Next, we show that when the sets are non-empty the chain is strictly monotonic including until all possible subpatterns are contained. 

\begin{lemma}
\label{lemma:uniqueplateau}
Let $m\geq1$. If there is an $n\geq0$ such that $P(\mytikz{0.08}{\trinone}, T_{n}, m)$ is non-empty and 
\begin{equation}
	\label{eq:plateauassumption}
	P(\mytikz{0.08}{\trinone}, T_{n}, m) = P(\mytikz{0.08}{\trinone}, T_{n+1}, m)
\end{equation}
then 
\begin{equation}
	\label{eq:plateauresult}
	P(\mytikz{0.08}{\trinone}, T_{n}, m) = P(\mytikz{0.08}{\trinone}, T_{n+k}, m)
\end{equation}
for all integers $k \geq 1$, and in particular $P(\mytikz{0.08}{\trinone}, T_{n}, m) = P(\mytikz{0.08}{\trinone}, T, m)$. 
\end{lemma}

\begin{proof}
We give a proof by induction on $k$ in \eqref{eq:plateauresult}. The basis case, $k = 1$, is direct from the assumption \eqref{eq:plateauassumption}. Assume for induction that \eqref{eq:plateauresult} holds for $1\leq k \leq p$. 
	
For the induction step, $k = p+1$, take a pattern $a \in P(\mytikz{0.08}{\trinone}, T_{n+p+1}, m)$. Then there is a pattern $b \in P(\mytikz{0.08}{\trinone}, T_{n+p}, m)$ such that $a$ is a subpattern of $\mu(b)$. By the induction assumption we have that $b \in P(\mytikz{0.08}{\trinone},T_{n+p-1}, m)$. This implies
\[
	a \in P(\mytikz{0.08}{\trinone},\mu(b), m) \subseteq P(\mytikz{0.08}{\trinone}, T_{n+p}, m).
\]
Therefore $P(\mytikz{0.08}{\trinone}, T_{n+p}, m) \supseteq P(\mytikz{0.08}{\trinone}, T_{n+p+1}, m)$, and by Lemma~\ref{lemma:inclusion} it follows that 
\[
	P(\mytikz{0.08}{\trinone}, T_{n+p}, m) = P(\mytikz{0.08}{\trinone}, T_{n+p+1}, m),
\]
which completes the induction. 
\end{proof}

\begin{example}
\label{ex:patterncount}
By inspection, we find
\[
	P(\mytikz{0.08}{\trinone}, T_3, 2) = P(\mytikz{0.08}{\trinone}, T_4, 2),
\]
with $|P(\mytikz{0.08}{\trinone},T_2, 2)| = 8$. Lemma~\ref{lemma:uniqueplateau} now implies that 
\[
	P(\mytikz{0.08}{\trinone}, T_3, 2) = P(\mytikz{0.08}{\trinone}, T, 2),
\]
so we can find all upwards oriented patterns of size 2 in $T$ by just looking at patterns in $T_3$. In the same way, continuing the enumeration and applying Lemma~\ref{lemma:uniqueplateau}, we find 
\[
	P(\mytikz{0.08}{\trinone}, T_4, 4) = P(\mytikz{0.08}{\trinone}, T_5, 4) = P(\mytikz{0.08}{\trinone}, T, 4),
\]
with $|P(\mytikz{0.08}{\trinone}, T, 4)| = 44$. As a consequence, we clearly also have $P(\mytikz{0.08}{\trinone}, T_4, 3) = P(\mytikz{0.08}{\trinone}, T, 3)$ without any further enumerations. This because $T_4$ contains all patterns of size 4, and therefore it must also contain all patterns of size 3. \hfill$\diamond$
\end{example}

The above Lemma~\ref{lemma:uniqueplateau}, and as seen in Example~\ref{ex:patterncount}, give us a way to find the sets of patterns of a given size that occurs in $T$. See Figure~\ref{fig:AllUpPatterns2}, Figure~\ref{fig:AllUpPatterns3}, Figure~\ref{fig:AllDownPatterns2} and Figure~\ref{fig:AllDownPatterns3} for lists of patterns of small sizes.

\begin{figure}[ht]
\centering
\begin{tabular}{*{8}{c}}
\begin{tikzpicture}[scale = 0.3, baseline={([yshift=-.8ex]current bounding box.center)}]
\clip (-0.2,-2.0) rectangle (2.2, 2.2);
\draw (0.000, 0.000) -- (2.000, 0.000) -- (1.000, 1.732) -- cycle;
\draw[ultra thin] (0.500, 0.866) -- (1.500, 0.866);
\draw[ultra thin] (0.500, 0.866) -- (1.000, 0.000);
\draw[ultra thin] (1.000, 0.000) -- (1.500, 0.866);
\draw (1,-1) node{\footnotesize 1};
\end{tikzpicture}
&
\begin{tikzpicture}[scale = 0.3, baseline={([yshift=-.8ex]current bounding box.center)}]
\clip (-0.2,-2.0) rectangle (2.2, 2.2);
\draw (0.000, 0.000) -- (2.000, 0.000) -- (1.000, 1.732) -- cycle;
\draw[ultra thin] (0.500, 0.866) -- (1.500, 0.866);
\draw[ultra thin] (0.500, 0.866) -- (1.000, 0.000);
\draw[ultra thin] (1.000, 0.000) -- (1.500, 0.866);
\fill[black] (1.500, 0.866) -- (1.000, 0.000) -- (2.000, 0.000) -- cycle;
\draw (1,-1) node{\footnotesize 2};
\end{tikzpicture}
&
\begin{tikzpicture}[scale = 0.3, baseline={([yshift=-.8ex]current bounding box.center)}]
\clip (-0.2,-2.0) rectangle (2.2, 2.2);
\draw (0.000, 0.000) -- (2.000, 0.000) -- (1.000, 1.732) -- cycle;
\draw[ultra thin] (0.500, 0.866) -- (1.500, 0.866);
\draw[ultra thin] (0.500, 0.866) -- (1.000, 0.000);
\draw[ultra thin] (1.000, 0.000) -- (1.500, 0.866);
\fill[black] (0.500, 0.866) -- (0.000, 0.000) -- (1.000, 0.000) -- cycle;
\draw (1,-1) node{\footnotesize 3};
\end{tikzpicture}
&
\begin{tikzpicture}[scale = 0.3, baseline={([yshift=-.8ex]current bounding box.center)}]
\clip (-0.2,-2.0) rectangle (2.2, 2.2);
\draw (0.000, 0.000) -- (2.000, 0.000) -- (1.000, 1.732) -- cycle;
\draw[ultra thin] (0.500, 0.866) -- (1.500, 0.866);
\draw[ultra thin] (0.500, 0.866) -- (1.000, 0.000);
\draw[ultra thin] (1.000, 0.000) -- (1.500, 0.866);
\fill[black] (0.500, 0.866) -- (0.000, 0.000) -- (1.000, 0.000) -- cycle;
\fill[black] (1.500, 0.866) -- (1.000, 0.000) -- (2.000, 0.000) -- cycle;
\draw (1,-1) node{\footnotesize 4};
\end{tikzpicture}
&
\begin{tikzpicture}[scale = 0.3, baseline={([yshift=-.8ex]current bounding box.center)}]
\clip (-0.2,-2.0) rectangle (2.2, 2.2);
\draw (0.000, 0.000) -- (2.000, 0.000) -- (1.000, 1.732) -- cycle;
\draw[ultra thin] (0.500, 0.866) -- (1.500, 0.866);
\draw[ultra thin] (0.500, 0.866) -- (1.000, 0.000);
\draw[ultra thin] (1.000, 0.000) -- (1.500, 0.866);
\fill[black] (1.000, 1.732) -- (0.500, 0.866) -- (1.500, 0.866) -- cycle;
\draw (1,-1) node{\footnotesize 5};
\end{tikzpicture}
&
\begin{tikzpicture}[scale = 0.3, baseline={([yshift=-.8ex]current bounding box.center)}]
\clip (-0.2,-2.0) rectangle (2.2, 2.2);
\draw (0.000, 0.000) -- (2.000, 0.000) -- (1.000, 1.732) -- cycle;
\draw[ultra thin] (0.500, 0.866) -- (1.500, 0.866);
\draw[ultra thin] (0.500, 0.866) -- (1.000, 0.000);
\draw[ultra thin] (1.000, 0.000) -- (1.500, 0.866);
\fill[black] (1.000, 1.732) -- (0.500, 0.866) -- (1.500, 0.866) -- cycle;
\fill[black] (1.500, 0.866) -- (1.000, 0.000) -- (2.000, 0.000) -- cycle;
\draw (1,-1) node{\footnotesize 6};
\end{tikzpicture}
&
\begin{tikzpicture}[scale = 0.3, baseline={([yshift=-.8ex]current bounding box.center)}]
\clip (-0.2,-2.0) rectangle (2.2, 2.2);
\draw (0.000, 0.000) -- (2.000, 0.000) -- (1.000, 1.732) -- cycle;
\draw[ultra thin] (0.500, 0.866) -- (1.500, 0.866);
\draw[ultra thin] (0.500, 0.866) -- (1.000, 0.000);
\draw[ultra thin] (1.000, 0.000) -- (1.500, 0.866);
\fill[black] (1.000, 1.732) -- (0.500, 0.866) -- (1.500, 0.866) -- cycle;
\fill[black] (0.500, 0.866) -- (0.000, 0.000) -- (1.000, 0.000) -- cycle;
\draw (1,-1) node{\footnotesize 7};
\end{tikzpicture}
&
\begin{tikzpicture}[scale = 0.3, baseline={([yshift=-.8ex]current bounding box.center)}]
\clip (-0.2,-2.0) rectangle (2.2, 2.2);
\draw (0.000, 0.000) -- (2.000, 0.000) -- (1.000, 1.732) -- cycle;
\draw[ultra thin] (0.500, 0.866) -- (1.500, 0.866);
\draw[ultra thin] (0.500, 0.866) -- (1.000, 0.000);
\draw[ultra thin] (1.000, 0.000) -- (1.500, 0.866);
\fill[black] (1.000, 1.732) -- (0.500, 0.866) -- (1.500, 0.866) -- cycle;
\fill[black] (0.500, 0.866) -- (0.000, 0.000) -- (1.000, 0.000) -- cycle;
\fill[black] (1.500, 0.866) -- (1.000, 0.000) -- (2.000, 0.000) -- cycle;
\draw (1,-1) node{\footnotesize 8};
\end{tikzpicture}
\end{tabular}
\caption{The 8 different upwards oriented patterns of size 2, that is, the elements of $P(\mytikz{0.07}{\protect\trinone},T,2)$.} 
\label{fig:AllUpPatterns2}
\end{figure}

\begin{figure}[ht]
\centering
\input{Allpatterns3.txt}
\caption{The 22 different upwards oriented patterns of size 3, that is, the elements of $P(\mytikz{0.07}{\protect\trinone},T,3)$.} 
\label{fig:AllUpPatterns3}
\end{figure}

\begin{figure}[ht]
\centering
\begin{tabular}{*{8}{c}}
\begin{tikzpicture}[scale = 0.3, baseline={([yshift=-.8ex]current bounding box.center)}]
\clip (-0.2,-2.0) rectangle (2.2, 2.2);
\draw (0.000, 1.732) -- (1.000, 0.000) -- (2.000, 1.732) -- cycle;
\draw[ultra thin] (0.500, 0.866) -- (1.500, 0.866);
\draw[ultra thin] (0.500, 0.866) -- (1.000, 1.732);
\draw[ultra thin] (1.000, 1.732) -- (1.500, 0.866);
\draw (1,-1) node{\footnotesize 1};
\end{tikzpicture}
&
\begin{tikzpicture}[scale = 0.3, baseline={([yshift=-.8ex]current bounding box.center)}]
\clip (-0.2,-2.0) rectangle (2.2, 2.2);
\draw (0.000, 1.732) -- (1.000, 0.000) -- (2.000, 1.732) -- cycle;
\draw[ultra thin] (0.500, 0.866) -- (1.500, 0.866);
\draw[ultra thin] (0.500, 0.866) -- (1.000, 1.732);
\draw[ultra thin] (1.000, 1.732) -- (1.500, 0.866);
\fill[black] (1.000, 1.732) -- (0.500, 0.866) -- (1.500, 0.866) -- cycle;
\draw (1,-1) node{\footnotesize 2};
\end{tikzpicture}
\end{tabular}
\caption{The 2 different downwards oriented patterns of size 2, that is, the elements of $P(\mytikz{0.07}{\protect\tridown},T,2)$.}
\label{fig:AllDownPatterns2} 
\end{figure}

\begin{figure}[ht]
\centering
\begin{tabular}{*{8}{c}}
\begin{tikzpicture}[scale = 0.3, baseline={([yshift=-.8ex]current bounding box.center)}]
\clip (-0.2,-2.0) rectangle (3.2, 3.2);
\draw (0.000, 2.598) -- (1.500, 0.000) -- (3.000, 2.598) -- cycle;
\draw[ultra thin] (0.500, 1.732) -- (2.500, 1.732);
\draw[ultra thin] (1.000, 0.866) -- (2.000, 0.866);
\draw[ultra thin] (0.500, 1.732) -- (1.000, 2.598);
\draw[ultra thin] (1.000, 0.866) -- (2.000, 2.598);
\draw[ultra thin] (1.000, 2.598) -- (2.000, 0.866);
\draw[ultra thin] (2.000, 2.598) -- (2.500, 1.732);
\draw (1.5,-1) node{\footnotesize 1};
\end{tikzpicture}
&
\begin{tikzpicture}[scale = 0.3, baseline={([yshift=-.8ex]current bounding box.center)}]
\clip (-0.2,-2.0) rectangle (3.2, 3.2);
\draw (0.000, 2.598) -- (1.500, 0.000) -- (3.000, 2.598) -- cycle;
\draw[ultra thin] (0.500, 1.732) -- (2.500, 1.732);
\draw[ultra thin] (1.000, 0.866) -- (2.000, 0.866);
\draw[ultra thin] (0.500, 1.732) -- (1.000, 2.598);
\draw[ultra thin] (1.000, 0.866) -- (2.000, 2.598);
\draw[ultra thin] (1.000, 2.598) -- (2.000, 0.866);
\draw[ultra thin] (2.000, 2.598) -- (2.500, 1.732);
\fill[black] (2.000, 2.598) -- (1.500, 1.732) -- (2.500, 1.732) -- cycle;
\fill[black] (1.500, 1.732) -- (1.000, 0.866) -- (2.000, 0.866) -- cycle;
\draw (1.5,-1) node{\footnotesize 2};
\end{tikzpicture}
&
\begin{tikzpicture}[scale = 0.3, baseline={([yshift=-.8ex]current bounding box.center)}]
\clip (-0.2,-2.0) rectangle (3.2, 3.2);
\draw (0.000, 2.598) -- (1.500, 0.000) -- (3.000, 2.598) -- cycle;
\draw[ultra thin] (0.500, 1.732) -- (2.500, 1.732);
\draw[ultra thin] (1.000, 0.866) -- (2.000, 0.866);
\draw[ultra thin] (0.500, 1.732) -- (1.000, 2.598);
\draw[ultra thin] (1.000, 0.866) -- (2.000, 2.598);
\draw[ultra thin] (1.000, 2.598) -- (2.000, 0.866);
\draw[ultra thin] (2.000, 2.598) -- (2.500, 1.732);
\fill[black] (1.000, 2.598) -- (0.500, 1.732) -- (1.500, 1.732) -- cycle;
\fill[black] (1.500, 1.732) -- (1.000, 0.866) -- (2.000, 0.866) -- cycle;
\draw (1.5,-1) node{\footnotesize 3};
\end{tikzpicture}
&
\begin{tikzpicture}[scale = 0.3, baseline={([yshift=-.8ex]current bounding box.center)}]
\clip (-0.2,-2.0) rectangle (3.2, 3.2);
\draw (0.000, 2.598) -- (1.500, 0.000) -- (3.000, 2.598) -- cycle;
\draw[ultra thin] (0.500, 1.732) -- (2.500, 1.732);
\draw[ultra thin] (1.000, 0.866) -- (2.000, 0.866);
\draw[ultra thin] (0.500, 1.732) -- (1.000, 2.598);
\draw[ultra thin] (1.000, 0.866) -- (2.000, 2.598);
\draw[ultra thin] (1.000, 2.598) -- (2.000, 0.866);
\draw[ultra thin] (2.000, 2.598) -- (2.500, 1.732);
\fill[black] (1.000, 2.598) -- (0.500, 1.732) -- (1.500, 1.732) -- cycle;
\fill[black] (2.000, 2.598) -- (1.500, 1.732) -- (2.500, 1.732) -- cycle;
\draw (1.5,-1) node{\footnotesize 4};
\end{tikzpicture}
\end{tabular}
\caption{The 4 different downwards oriented patterns of size 3, that is, the elements of $P(\mytikz{0.07}{\protect\tridown},T,3)$.} 
\label{fig:AllDownPatterns3}
\end{figure}

\section{Intersections}

In this section we discuss properties of sets of patterns. Recall the definition of the particular sets of patterns given in \eqref{eq:defOfP}. These sets can be split into subsets (not necessarily disjoint) depending on their pattern's position in the underlying structure of super-tiles of size 2. Let us start by defining a set of indices (pair of integers) by 
\[	
	I : = \{ (0,0), (1,0), (1,2), (2,2) \}.
\]
For indices $(r,c) \in I$, and $n\geq2$ we introduce the following set of patterns 
\begin{equation}
\label{eq:defofPrc}
	P_{r,c}(\alpha, T, n) 
	:= \big\{ \mu(x)[\alpha, r,c, n] : x \in P(\alpha,T, n) \big\},
\end{equation}
where $\alpha\in\mathcal{A}$ is the type of pattern, see \eqref{eq:defOfTypes}. The definition in \eqref{eq:defofPrc} can be extend to further indices via
\begin{equation}
\label{eq:defofPrcExtension}
	P_{r+2s,c+4t}(\alpha, T, n) := P_{r,c}(\alpha, T, n), 
\end{equation}
for $s,t\in \mathbb{N}$. The above definitions lead to 
\begin{equation}
\label{eq:PequalsPrc}
	P(\alpha, T, n) = \bigcup_{(r,c) \in I } P_{r,c}(\alpha, T, n).
\end{equation}
The sets on the right-hand side of \eqref{eq:PequalsPrc} are not pairwise disjoint, as we will show in the following theorem, which is also the main result of this section.

\begin{theorem}
\label{tmm:additivity}
Let $n\geq 2$, and $\alpha\in\mathcal{A}$. Then 
\begin{equation}
\label{eq:additivity}
	\left|P(\alpha, T,n)\right| = - 6 + \sum_{(r,c) \in I} \left|P_{r,c}(\alpha, T,n)\right|.
\end{equation}
\end{theorem}

The proof of Theorem~\ref{tmm:additivity} is based on several lemmas, given below. In order to state the results of the lemmas we need to introduce a list of notation for special types of patterns. 

\begin{itemize}
\item Let
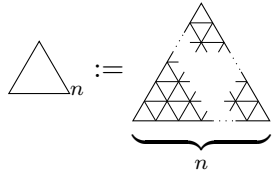
\[
\mytikz{0.2}{\trinone}_{n}
:= \underbrace{
\mytikz{0.3}{
\draw (3.500, 0.000) -- (0.000, 0.000) -- (1.750, 3.031) (2.250, 3.897) -- (3.000, 5.196) -- (4.250, 3.031) (4.750, 2.165) -- (6.000, 0.000) -- (4.500, 0.000);
\draw[dotted] (3.500, 0.000) -- (4.500, 0.000);
\draw[dotted] (1.750, 3.031) -- (2.250, 3.897);
\draw[dotted] (4.250, 3.031) -- (4.750, 2.165);
\draw[ultra thin] (0.500, 0.866) -- (3.000, 0.866) (4.000, 0.866) -- (5.500, 0.866); 
\draw[ultra thin] (1.000, 1.732) -- (2.500, 1.732) (4.500, 1.732) -- (5.000, 1.732);
\draw[ultra thin] (1.500, 2.598) -- (2.000, 2.598);
\draw[ultra thin] (2.500, 3.464) -- (4.000, 3.464);
\draw[ultra thin] (2.500, 4.330) -- (3.500, 4.330);
\draw[ultra thin] (0.500, 0.866) -- (1.000, 0.000);
\draw[ultra thin] (1.000, 1.732) -- (2.000, 0.000);
\draw[ultra thin] (1.500, 2.598) -- (3.000, 0.000);
\draw[ultra thin] (2.500, 4.330) -- (3.250, 3.031) (4.250, 1.299) -- (5.000, 0.000);
\draw[ultra thin] (1.000, 0.000) -- (2.250, 2.165) (2.750, 3.031) -- (3.500, 4.330);
\draw[ultra thin] (2.000, 0.000) -- (2.750, 1.299) (3.750, 3.031) -- (4.000, 3.464);
\draw[ultra thin] (3.000, 0.000) -- (3.250, 0.433); 
\draw[ultra thin] (4.250, 0.433) -- (5.000, 1.732);
\draw[ultra thin] (5.000, 0.000) -- (5.500, 0.866);
\clip (0.000, -0.500) rectangle (6.000, 5.196);
}}_{n}
\]
represent an upwards oriented pattern of size $n$ consisting of unfilled unit triangles.

\item Let
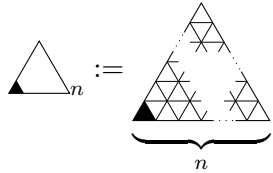
\[
\mytikz{0.2}{\trinone \fill (0.000, 0.000) -- (1.000, 0.000) -- (0.500, 0.866) -- cycle;}_{n}
:= \underbrace{
\mytikz{0.3}{
\draw (3.500, 0.000) -- (0.000, 0.000) -- (1.750, 3.031) (2.250, 3.897) -- (3.000, 5.196) -- (4.250, 3.031) (4.750, 2.165) -- (6.000, 0.000) -- (4.500, 0.000);
\draw[dotted] (3.500, 0.000) -- (4.500, 0.000);
\draw[dotted] (1.750, 3.031) -- (2.250, 3.897);
\draw[dotted] (4.250, 3.031) -- (4.750, 2.165);
\draw[ultra thin] (0.500, 0.866) -- (3.000, 0.866) (4.000, 0.866) -- (5.500, 0.866); 
\draw[ultra thin] (1.000, 1.732) -- (2.500, 1.732) (4.500, 1.732) -- (5.000, 1.732);
\draw[ultra thin] (1.500, 2.598) -- (2.000, 2.598);
\draw[ultra thin] (2.500, 3.464) -- (4.000, 3.464);
\draw[ultra thin] (2.500, 4.330) -- (3.500, 4.330);
\draw[ultra thin] (0.500, 0.866) -- (1.000, 0.000);
\draw[ultra thin] (1.000, 1.732) -- (2.000, 0.000);
\draw[ultra thin] (1.500, 2.598) -- (3.000, 0.000);
\draw[ultra thin] (2.500, 4.330) -- (3.250, 3.031) (4.250, 1.299) -- (5.000, 0.000);
\draw[ultra thin] (1.000, 0.000) -- (2.250, 2.165) (2.750, 3.031) -- (3.500, 4.330);
\draw[ultra thin] (2.000, 0.000) -- (2.750, 1.299) (3.750, 3.031) -- (4.000, 3.464);
\draw[ultra thin] (3.000, 0.000) -- (3.250, 0.433); 
\draw[ultra thin] (4.250, 0.433) -- (5.000, 1.732);
\draw[ultra thin] (5.000, 0.000) -- (5.500, 0.866);
\fill[black] (0.500, 0.866) -- (0.000, 0.000) -- (1.000, 0.000) -- cycle;
\clip (0.000, -0.500) rectangle (6.000, 5.196);
}}_{n}
\]
represent an upwards oriented pattern of size $n$ consisting of unfilled unit triangles, except that the marked corner contains one filled unit triangle. The notation extends to marking one or more corners.

\item Let 
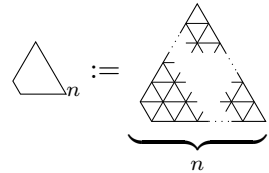
\[
\mytikz{0.2}{\draw (2.000, 3.464) -- (0.500, 0.866) -- (1.000, 0.000) -- (4.000, 0.000) -- cycle;}_{n} 
:= \underbrace{
\mytikz{0.3}{
\draw (3.500, 0.000) -- (1.000, 0.000) -- (0.500, 0.866) -- (1.750, 3.031) (2.250, 3.897) -- (3.000, 5.196) -- (4.250, 3.031) (4.750, 2.165) -- (6.000, 0.000) -- (4.500, 0.000);
\draw[dotted] (3.500, 0.000) -- (4.500, 0.000);
\draw[dotted] (1.750, 3.031) -- (2.250, 3.897);
\draw[dotted] (4.250, 3.031) -- (4.750, 2.165);
\draw[ultra thin] (0.500, 0.866) -- (3.000, 0.866) (4.000, 0.866) -- (5.500, 0.866); 
\draw[ultra thin] (1.000, 1.732) -- (2.500, 1.732) (4.500, 1.732) -- (5.000, 1.732);
\draw[ultra thin] (1.500, 2.598) -- (2.000, 2.598);
\draw[ultra thin] (2.500, 3.464) -- (4.000, 3.464);
\draw[ultra thin] (2.500, 4.330) -- (3.500, 4.330);
\draw[ultra thin] (1.000, 1.732) -- (2.000, 0.000);
\draw[ultra thin] (1.500, 2.598) -- (3.000, 0.000);
\draw[ultra thin] (2.500, 4.330) -- (3.250, 3.031) (4.250, 1.299) -- (5.000, 0.000);
\draw[ultra thin] (1.000, 0.000) -- (2.250, 2.165) (2.750, 3.031) -- (3.500, 4.330);
\draw[ultra thin] (2.000, 0.000) -- (2.750, 1.299) (3.750, 3.031) -- (4.000, 3.464);
\draw[ultra thin] (3.000, 0.000) -- (3.250, 0.433); 
\draw[ultra thin] (4.250, 0.433) -- (5.000, 1.732);
\draw[ultra thin] (5.000, 0.000) -- (5.500, 0.866);
\clip (0.000, -0.500) rectangle (6.000, 5.196);
}}_{n}
\]
represent an upwards oriented pattern of size $n$ consisting of unfilled unit triangles, and at the indicated corner one triangle of unit length has be cut off. The notation extends to cutting off one or more corners.

\item Let 
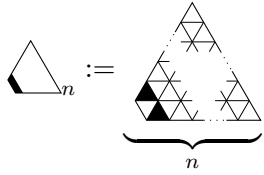
\[
\mytikz{0.2}{
\draw (2.000, 3.464) -- (0.500, 0.866) -- (1.000, 0.000) -- (4.000, 0.000) -- cycle;
\fill (0.500, 0.866) -- (1.000, 0.000) -- (1.500, 0.000) -- (0.750, 1.299) -- cycle;
}_{n} 
:= \underbrace{
\mytikz{0.3}{
\draw (3.500, 0.000) -- (1.000, 0.000) -- (0.500, 0.866) -- (1.750, 3.031) (2.250, 3.897) -- (3.000, 5.196) -- (4.250, 3.031) (4.750, 2.165) -- (6.000, 0.000) -- (4.500, 0.000);
\draw[dotted] (3.500, 0.000) -- (4.500, 0.000);
\draw[dotted] (1.750, 3.031) -- (2.250, 3.897);
\draw[dotted] (4.250, 3.031) -- (4.750, 2.165);
\draw[ultra thin] (0.500, 0.866) -- (3.000, 0.866) (4.000, 0.866) -- (5.500, 0.866); 
\draw[ultra thin] (1.000, 1.732) -- (2.500, 1.732) (4.500, 1.732) -- (5.000, 1.732);
\draw[ultra thin] (1.500, 2.598) -- (2.000, 2.598);
\draw[ultra thin] (2.500, 3.464) -- (4.000, 3.464);
\draw[ultra thin] (2.500, 4.330) -- (3.500, 4.330);
\draw[ultra thin] (1.000, 1.732) -- (2.000, 0.000);
\draw[ultra thin] (1.500, 2.598) -- (3.000, 0.000);
\draw[ultra thin] (2.500, 4.330) -- (3.250, 3.031) (4.250, 1.299) -- (5.000, 0.000);
\draw[ultra thin] (1.000, 0.000) -- (2.250, 2.165) (2.750, 3.031) -- (3.500, 4.330);
\draw[ultra thin] (2.000, 0.000) -- (2.750, 1.299) (3.750, 3.031) -- (4.000, 3.464);
\draw[ultra thin] (3.000, 0.000) -- (3.250, 0.433); 
\draw[ultra thin] (4.250, 0.433) -- (5.000, 1.732);
\draw[ultra thin] (5.000, 0.000) -- (5.500, 0.866);
\fill[black] (1.000, 1.732) -- (0.500, 0.866) -- (1.500, 0.866) -- cycle;
\fill[black] (1.500, 0.866) -- (1.000, 0.000) -- (2.000, 0.000) -- cycle;
\clip (0.000, -0.500) rectangle (6.000, 5.196);
}}_{n}
\]
represent an upwards oriented pattern of size $n$ consisting of unfilled unit triangles and at the indicated corner one triangle of unit length has be cut of and the two upwards oriented unit triangles next to the cut are filled. The notation of cutting off and marking extends to one or more corners.

\item Let
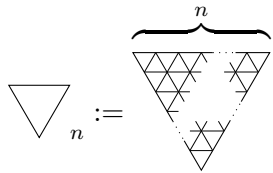
\[
\mytikz{0.2}{
\tridown
}_{n} 
:= \overbrace{
\mytikz{0.3}{
\draw (3.500, 5.196) -- (0.000, 5.196) -- (1.750, 2.165) (2.250, 1.299) -- (3.000, 0.000) -- (4.250, 2.165) (4.750, 3.031) -- (6.000, 5.196) -- (4.500, 5.196);
\draw[dotted] (3.500, 5.196) -- (4.500, 5.196);
\draw[dotted] (1.750, 2.165) -- (2.250, 1.299);
\draw[dotted] (4.250, 2.165) -- (4.750, 3.031);
\draw[ultra thin] (0.500, 4.339) -- (3.000, 4.339) (4.000, 4.339) -- (5.500, 4.339); 
\draw[ultra thin] (1.000, 3.464) -- (2.500, 3.464) (4.500, 3.464) -- (5.000, 3.464);
\draw[ultra thin] (1.500, 2.598) -- (2.000, 2.598);
\draw[ultra thin] (2.500, 1.732) -- (4.000, 1.732);
\draw[ultra thin] (2.500, 0.866) -- (3.500, 0.866);
\draw[ultra thin] (0.500, 4.330) -- (1.000, 5.196);
\draw[ultra thin] (1.000, 3.464) -- (2.000, 5.196);
\draw[ultra thin] (1.500, 2.598) -- (3.000, 5.196);
\draw[ultra thin] (2.500, 0.866) -- (3.250, 2.165) (4.250, 3.897) -- (5.000, 5.196);
\draw[ultra thin] (1.000, 5.196) -- (2.250, 3.031) (2.750, 2.165) -- (3.500, 0.866);
\draw[ultra thin] (2.000, 5.196) -- (2.750, 3.897) (3.750, 2.165) -- (4.000, 1.732);
\draw[ultra thin] (3.000, 5.196) -- (3.250, 4.763); 
\draw[ultra thin] (4.250, 4.763) -- (5.000, 3.464);
\draw[ultra thin] (5.000, 5.196) -- (5.500, 4.330);
\clip (0.000, -0.500) rectangle (6.000, 5.696);
}}^{n}
\]
represent a downwards oriented pattern of size $n$ consisting of unfilled unit triangles.

\item Let
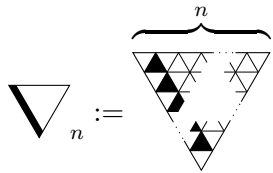
\[
\mytikz{0.2}{
\tridown \fill (0.000, 3.464) -- (0.500, 3.464) -- (2.250, 0.433) -- (2.000, 0.000) -- cycle;
}_{n} 
:= \overbrace{
\mytikz{0.3}{
\draw (3.500, 5.196) -- (0.000, 5.196) -- (1.750, 2.165) (2.250, 1.299) -- (3.000, 0.000) -- (4.250, 2.165) (4.750, 3.031) -- (6.000, 5.196) -- (4.500, 5.196);
\draw[dotted] (3.500, 5.196) -- (4.500, 5.196);
\draw[dotted] (1.750, 2.165) -- (2.250, 1.299);
\draw[dotted] (4.250, 2.165) -- (4.750, 3.031);
\draw[ultra thin] (0.500, 4.339) -- (3.000, 4.339) (4.000, 4.339) -- (5.500, 4.339); 
\draw[ultra thin] (1.000, 3.464) -- (2.500, 3.464) (4.500, 3.464) -- (5.000, 3.464);
\draw[ultra thin] (1.500, 2.598) -- (2.000, 2.598);
\draw[ultra thin] (2.500, 1.732) -- (4.000, 1.732);
\draw[ultra thin] (2.500, 0.866) -- (3.500, 0.866);
\draw[ultra thin] (0.500, 4.330) -- (1.000, 5.196);
\draw[ultra thin] (1.000, 3.464) -- (2.000, 5.196);
\draw[ultra thin] (1.500, 2.598) -- (3.000, 5.196);
\draw[ultra thin] (2.500, 0.866) -- (3.250, 2.165) (4.250, 3.897) -- (5.000, 5.196);
\draw[ultra thin] (1.000, 5.196) -- (2.250, 3.031) (2.750, 2.165) -- (3.500, 0.866);
\draw[ultra thin] (2.000, 5.196) -- (2.750, 3.897) (3.750, 2.165) -- (4.000, 1.732);
\draw[ultra thin] (3.000, 5.196) -- (3.250, 4.763); 
\draw[ultra thin] (4.250, 4.763) -- (5.000, 3.464);
\draw[ultra thin] (5.000, 5.196) -- (5.500, 4.330);
\fill[black] (1.000, 5.196) -- (0.500, 4.330) -- (1.500, 4.330) -- cycle;
\fill[black] (1.500, 4.330) -- (1.000, 3.464) -- (2.000, 3.464) -- cycle;
\fill[black] (2.000, 3.464) -- (1.500, 2.598) -- (2.000, 2.598) -- (2.250, 3.031) -- cycle;
\fill[black] (2.750, 2.165) -- (2.500, 1.732) -- (3.000, 1.732) -- cycle;
\fill[black] (3.000, 1.732) -- (2.500, 0.866) -- (3.500, 0.866) -- cycle;
\clip (0.000, -0.500) rectangle (6.000, 5.696);
}}^{n}
\]
represent a downwards oriented pattern of size $n$ consisting of unfilled unit triangles, and at the marked side the upwards oriented unit triangles are filled.
\end{itemize}

The above notations describe the corners and sides in a pattern, and a pattern may consist of different types of corners; cut or uncut, marked or unmarked, and so on. With the help of the above definitions we can now give a string of lemmas leading up to the proof of Theorem~\ref{tmm:additivity}.

\begin{lemma}
\label{lemma:instersectionEven}
Let $n\geq1$. Then 
\begin{align}
P_{0,0}(\mytikz{0.08}{\trinone},T,2n) \ \bigcap \ P_{1,0}(\mytikz{0.08}{\trinone},T,2n) &= \left\{\mytikz{0.2}{\trinone}_{2n} \right\}, \label{eq:instersectionEven1}\\
P_{0,0}(\mytikz{0.08}{\trinone},T,2n) \ \bigcap \ P_{1,2}(\mytikz{0.08}{\trinone},T,2n) &= \left\{\mytikz{0.2}{\trinone}_{2n} \right\}, \label{eq:instersectionEven2}\\
P_{0,0}(\mytikz{0.08}{\trinone},T,2n) \ \bigcap \ P_{2,2}(\mytikz{0.08}{\trinone},T,2n) &= \left\{\mytikz{0.2}{\trinone}_{2n} \right\}, \label{eq:instersectionEven3}\\
P_{1,0}(\mytikz{0.08}{\trinone},T,2n) \ \bigcap \ P_{1,2}(\mytikz{0.08}{\trinone},T,2n) &= 
\left\{\mytikz{0.2}{\trinone}_{2n}, \mytikz{0.2}{\trinone \fill (1.500, 2.598) -- (2.500, 2.598) -- (2.000, 3.464) -- cycle;}_{2n} \right\}, \label{eq:instersectionEven4} \\
P_{1,0}(\mytikz{0.08}{\trinone},T,2n) \ \bigcap \ P_{2,2}(\mytikz{0.08}{\trinone},T,2n) &= 
\left\{\mytikz{0.2}{\trinone}_{2n}, \mytikz{0.2}{\trinone \fill (0.000, 0.000) -- (1.000, 0.000) -- (0.500, 0.866) -- cycle;}_{2n} \right\},
\label{eq:instersectionEven5} \\
P_{1,2}(\mytikz{0.08}{\trinone},T,2n) \ \bigcap \ P_{2,2}(\mytikz{0.08}{\trinone},T,2n) &= 
\left\{\mytikz{0.2}{\trinone}_{2n}, \mytikz{0.2}{\trinone \fill (3.000, 0.000) -- (4.000, 0.000) -- (3.500, 0.866) -- cycle;}_{2n} \right\}.
\label{eq:instersectionEven6}
\end{align}
\end{lemma}

\begin{proof}
Let us consider the equality in \eqref{eq:instersectionEven1}. We prove it by induction on $n$. The basis cases $n = 1,2$ are seen via a straight forward enumeration. Assume \eqref{eq:instersectionEven1} holds for $n\leq p$. For the induction step, $n = p+1$, we take a pattern
\[
	x \in P_{0,0}(\mytikz{0.08}{\trinone},T,2p+2) \ \bigcap \ P_{1,0}(\mytikz{0.08}{\trinone},T,2p+2).
\]
Such a pattern exists, since the intersection is non-empty; it contains at-least the pattern consisting of unfilled unit triangles. 
By the definition of the $P_{r,c}$ sets from \eqref{eq:defofPrcExtension} and the induction assumption we obtain 
 \[
	\renewcommand{\arraystretch}{1.3}
	\left\{\begin{array}{@{}l@{}}	
		x[\mytikz{0.08}{\trinone},0,0,2p], \\
		x[\mytikz{0.08}{\trinone},2,0,2p], \\
		x[\mytikz{0.08}{\trinone},2,4,2p]
	\end{array}\right\}
\quad \subseteq \quad P_{0,0}(T,2p) \ \bigcap \ P_{1,0}(\mytikz{0.08}{\trinone}, T,2p).
\]
Therefore we may conclude 
\[	
	x = \mytikz{0.2}{\trinone}_{2p+2}\,,
\]
which concludes the induction. The remaining intersections in \eqref{eq:instersectionEven2}, \eqref{eq:instersectionEven3}, \eqref{eq:instersectionEven4}, \eqref{eq:instersectionEven5}, and \eqref{eq:instersectionEven6} are dealt with in the same way. 
\end{proof}

The following lemmas are similar and can be proven in the same way as Lemma~\ref{lemma:instersectionEven}.

\begin{lemma}
Let $n\geq1$. Then 
\begin{align*}
P_{0,0}(\mytikz{0.08}{\trinone},T,2n+1) \ \bigcap \ P_{1,0}(\mytikz{0.08}{\trinone},T,2n+1) &= 
\left\{\mytikz{0.2}{\trinone}_{2n+1}, \mytikz{0.2}{\trinone \fill (3.000, 0.000) -- (4.000, 0.000) -- (3.500, 0.866) -- cycle;}_{2n+1} \right\}, \\
P_{0,0}(\mytikz{0.08}{\trinone},T,2n+1) \ \bigcap \ P_{1,2}(\mytikz{0.08}{\trinone},T,2n+1) &= 
\left\{\mytikz{0.2}{\trinone}_{2n+1}, \mytikz{0.2}{\trinone \fill (0.000, 0.000) -- (1.000, 0.000) -- (0.500, 0.866) -- cycle;}_{2n+1} \right\}, \\
P_{0,0}(\mytikz{0.08}{\trinone},T,2n+1) \ \bigcap \ P_{2,2}(\mytikz{0.08}{\trinone},T,2n+1) &= 
\left\{\mytikz{0.2}{\trinone}_{2n+1} \right\}, \\
P_{1,0}(\mytikz{0.08}{\trinone},T,2n+1) \ \bigcap \ P_{1,2}(\mytikz{0.08}{\trinone},T,2n+1) &= 
\left\{\mytikz{0.2}{\trinone}_{2n+1}, \mytikz{0.2}{\trinone \fill (1.500, 2.598) -- (2.500, 2.598) -- (2.000, 3.464) -- cycle;}_{2n+1} \right\}, \\
P_{1,0}(\mytikz{0.08}{\trinone},T,2n+1) \ \bigcap \ P_{2,2}(\mytikz{0.08}{\trinone},T,2n+1) &= 
\left\{\mytikz{0.2}{\trinone}_{2n+1} \right\}, \\
P_{1,2}(\mytikz{0.08}{\trinone},T,2n+1) \ \bigcap \ P_{2,2}(\mytikz{0.08}{\trinone},T,2n+1) &= 
\left\{\mytikz{0.2}{\trinone}_{2n+1} \right\}. 
\end{align*}
\end{lemma}

\begin{lemma}
Let $n\geq2$. Then 
\begin{align*}
P_{0,0}(\mytikz{0.08}{\tritop},T,2n) \ \bigcap \ P_{1,0}(\mytikz{0.08}{\tritop},T,2n) &= 
\left\{\mytikz{0.2}{\tritop}_{2n} \right\}, \\
P_{0,0}(\mytikz{0.08}{\tritop},T,2n) \ \bigcap \ P_{1,2}(\mytikz{0.08}{\tritop},T,2n) &= 
\left\{\mytikz{0.2}{\tritop}_{2n} \right\}, \\
P_{0,0}(\mytikz{0.08}{\tritop},T,2n) \ \bigcap \ P_{2,2}(\mytikz{0.08}{\tritop},T,2n) &= 
\left\{\mytikz{0.2}{\tritop}_{2n}, \mytikz{0.2}{\tritop \fill (1.500, 2.598) -- (1.250, 2.165) -- (2.750, 2.165) -- (2.500, 2.598) -- cycle;}_{2n} \right\}, \\
P_{1,0}(\mytikz{0.08}{\tritop},T,2n) \ \bigcap \ P_{1,2}(\mytikz{0.08}{\tritop},T,2n) &= 
\left\{\mytikz{0.2}{\tritop}_{2n} \right\}, \\
P_{1,0}(\mytikz{0.08}{\tritop},T,2n) \ \bigcap \ P_{2,2}(\mytikz{0.08}{\tritop},T,2n) &=
\left\{\mytikz{0.2}{\tritop}_{2n}, \mytikz{0.2}{\tritop \fill (0.000, 0.000) -- (1.000, 0.000) -- (0.500, 0.866) -- cycle;}_{2n} \right\}, \\
P_{1,2}(\mytikz{0.08}{\tritop},T,2n) \ \bigcap \ P_{2,2}(\mytikz{0.08}{\tritop},T,2n) &=
\left\{\mytikz{0.2}{\tritop}_{2n}, \mytikz{0.2}{\tritop \fill (3.000, 0.000) -- (4.000, 0.000) -- (3.500, 0.866) -- cycle;}_{2n} \right\}.
\end{align*}
\end{lemma}

\begin{lemma}
Let $n\geq2$. Then 
\begin{align*}
P_{0,0}(\mytikz{0.08}{\tritop},T,2n+1) \ \bigcap \ P_{1,0}(\mytikz{0.08}{\tritop},T,2n+1) &= 
\left\{\mytikz{0.2}{\tritop}_{2n+1}, \mytikz{0.2}{\tritop \fill (3.000, 0.000) -- (4.000, 0.000) -- (3.500, 0.866) -- cycle;}_{2n+1} \right\}, \\
P_{0,0}(\mytikz{0.08}{\tritop},T,2n+1) \ \bigcap \ P_{1,2}(\mytikz{0.08}{\tritop},T,2n+1) &= 
\left\{\mytikz{0.2}{\tritop}_{2n+1}, \mytikz{0.2}{\tritop \fill (0.000, 0.000) -- (1.000, 0.000) -- (0.500, 0.866) -- cycle;}_{2n+1} \right\}, \\
P_{0,0}(\mytikz{0.08}{\tritop},T,2n+1) \ \bigcap \ P_{2,2}(\mytikz{0.08}{\tritop},T,2n+1) &= 
\left\{\mytikz{0.2}{\tritop}_{2n+1}, \mytikz{0.2}{\tritop \fill (1.500, 2.598) -- (1.250, 2.165) -- (2.750, 2.165) -- (2.500, 2.598) -- cycle;}_{2n+1} \right\}, \\
P_{1,0}(\mytikz{0.08}{\tritop},T,2n+1) \ \bigcap \ P_{1,2}(\mytikz{0.08}{\tritop},T,2n+1) &= \left\{\mytikz{0.2}{\tritop}_{2n+1} \right\}, \\
P_{1,0}(\mytikz{0.08}{\tritop},T,2n+1) \ \bigcap \ P_{2,2}(\mytikz{0.08}{\tritop},T,2n+1) &= \left\{\mytikz{0.2}{\tritop}_{2n+1} \right\}, \\
P_{1,2}(\mytikz{0.08}{\tritop},T,2n+1) \ \bigcap \ P_{2,2}(\mytikz{0.08}{\tritop},T,2n+1) &= \left\{\mytikz{0.2}{\tritop}_{2n+1} \right\}.
\end{align*}
\end{lemma}

\begin{lemma}
Let $n\geq2$. Then 
\begin{align*}
P_{0,0}(\mytikz{0.08}{\trilower},T,2n) \ \bigcap \ P_{1,0}(\mytikz{0.08}{\trilower},T,2n) &= 
\left\{\mytikz{0.2}{\trilower}_{2n}, \mytikz{0.2}{\trilower \fill (3.500, 0.866) -- (3.000, 0.000) -- (2.500, 0.000) -- (3.250, 1.299) -- cycle;}_{2n} \right\}, \\
P_{0,0}(\mytikz{0.08}{\trilower},T,2n) \ \bigcap \ P_{1,2}(\mytikz{0.08}{\trilower},T,2n) &=
\left\{\mytikz{0.2}{\trilower}_{2n}, \mytikz{0.2}{\trilower \fill (0.500, 0.866) -- (1.000, 0.000) -- (1.500, 0.000) -- (0.750, 1.299) -- cycle;}_{2n} \right\}, \\
P_{0,0}(\mytikz{0.08}{\trilower},T,2n) \ \bigcap \ P_{2,2}(\mytikz{0.08}{\trilower},T,2n) &= 
\left\{\mytikz{0.2}{\trilower}_{2n} \right\}, \\
P_{1,0}(\mytikz{0.08}{\trilower},T,2n) \ \bigcap \ P_{1,2}(\mytikz{0.08}{\trilower},T,2n) &= 
\left\{\mytikz{0.2}{\trilower}_{2n}, \mytikz{0.2}{\trilower \fill (1.500, 2.598) -- (2.500, 2.598) -- (2.000, 3.464) -- cycle;}_{2n} \right\}, \\
P_{1,0}(\mytikz{0.08}{\trilower},T,2n) \ \bigcap \ P_{2,2}(\mytikz{0.08}{\trilower},T,2n) &= 
\left\{\mytikz{0.2}{\trilower}_{2n} \right\}, \\
P_{1,2}(\mytikz{0.08}{\trilower},T,2n) \ \bigcap \ P_{2,2}(\mytikz{0.08}{\trilower},T,2n) &= 
\left\{\mytikz{0.2}{\trilower}_{2n} \right\}.
\end{align*}
\end{lemma}

\begin{lemma}
Let $n\geq2$. Then 
\begin{align*}
P_{0,0}(\mytikz{0.08}{\trilower},T,2n+1) \ \bigcap \ P_{1,0}(\mytikz{0.08}{\trilower},T,2n+1) &= 
\left\{\mytikz{0.2}{\trilower}_{2n+1} \right\}, \\
P_{0,0}(\mytikz{0.08}{\trilower},T,2n+1) \ \bigcap \ P_{1,2}(\mytikz{0.08}{\trilower},T,2n+1) &= 
\left\{\mytikz{0.2}{\trilower}_{2n+1} \right\}, \\
P_{0,0}(\mytikz{0.08}{\trilower},T,2n+1) \ \bigcap \ P_{2,2}(\mytikz{0.08}{\trilower},T,2n+1) &= 
\left\{\mytikz{0.2}{\trilower}_{2n+1} \right\}, \\
P_{1,0}(\mytikz{0.08}{\trilower},T,2n+1) \ \bigcap \ P_{1,2}(\mytikz{0.08}{\trilower},T,2n+1) &= 
\left\{\mytikz{0.2}{\trilower}_{2n+1}, \mytikz{0.2}{\trilower \fill (1.500, 2.598) -- (2.500, 2.598) -- (2.000, 3.464) -- cycle;}_{2n+1} \right\}, \\
P_{1,0}(\mytikz{0.08}{\trilower},T,2n+1) \ \bigcap \ P_{2,2}(\mytikz{0.08}{\trilower},T,2n+1) &=
\left\{\mytikz{0.2}{\trilower}_{2n+1}, \mytikz{0.2}{\trilower \fill (0.500, 0.866) -- (1.000, 0.000) -- (1.500, 0.000) -- (0.750, 1.299) -- cycle;}_{2n+1} \right\}, \\
P_{1,2}(\mytikz{0.08}{\trilower},T,2n+1) \ \bigcap \ P_{2,2}(\mytikz{0.08}{\trilower},T,2n+1) &=
\left\{\mytikz{0.2}{\trilower}_{2n+1}, \mytikz{0.2}{\trilower \fill (3.500, 0.866) -- (3.000, 0.000) -- (2.500, 0.000) -- (3.250, 1.299) -- cycle;}_{2n+1} \right\}.
\end{align*}
\end{lemma}

\begin{lemma}
Let $n\geq2$. Then 
\begin{align*}
P_{0,0}(\mytikz{0.08}{\triall},T,2n) \ \bigcap \ P_{1,0}(\mytikz{0.08}{\triall},T,2n) &= 
\left\{\mytikz{0.2}{\triall}_{2n}, \mytikz{0.2}{\triall \fill (3.500, 0.866) -- (3.000, 0.000) -- (2.500, 0.000) -- (3.250, 1.299) -- cycle;}_{2n} \right\}, \\
P_{0,0}(\mytikz{0.08}{\triall},T,2n) \ \bigcap \ P_{1,2}(\mytikz{0.08}{\triall},T,2n) &=
\left\{\mytikz{0.2}{\triall}_{2n}, \mytikz{0.2}{\triall \fill (0.500, 0.866) -- (1.000, 0.000) -- (1.500, 0.000) -- (0.750, 1.299) -- cycle;}_{2n} \right\}, \\
P_{0,0}(\mytikz{0.08}{\triall},T,2n) \ \bigcap \ P_{2,2}(\mytikz{0.08}{\triall},T,2n) &= 
\left\{\mytikz{0.2}{\triall}_{2n}, \mytikz{0.2}{\triall \fill (1.500, 2.598) -- (1.250, 2.165) -- (2.750, 2.165) -- (2.500, 2.598) -- cycle;}_{2n} \right\}, \\
P_{1,0}(\mytikz{0.08}{\triall},T,2n) \ \bigcap \ P_{1,2}(\mytikz{0.08}{\triall},T,2n) &= 
\left\{\mytikz{0.2}{\triall}_{2n} \right\},\\
P_{1,0}(\mytikz{0.08}{\triall},T,2n) \ \bigcap \ P_{2,2}(\mytikz{0.08}{\triall},T,2n) &= 
\left\{\mytikz{0.2}{\triall}_{2n} \right\},\\
P_{1,2}(\mytikz{0.08}{\triall},T,2n) \ \bigcap \ P_{2,2}(\mytikz{0.08}{\triall},T,2n) &= 
\left\{\mytikz{0.2}{\triall}_{2n} \right\}.
\end{align*}
\end{lemma}

\begin{lemma}
Let $n\geq2$. Then 
\begin{align*}
P_{0,0}(\mytikz{0.08}{\triall},T,2n+1) \ \bigcap \ P_{1,0}(\mytikz{0.08}{\triall},T,2n+1) &= 
\left\{\mytikz{0.2}{\triall}_{2n+1} \right\},\\
P_{0,0}(\mytikz{0.08}{\triall},T,2n+1) \ \bigcap \ P_{1,2}(\mytikz{0.08}{\triall},T,2n+1) &= 
\left\{\mytikz{0.2}{\triall}_{2n+1} \right\},\\
P_{0,0}(\mytikz{0.08}{\triall},T,2n+1) \ \bigcap \ P_{2,2}(\mytikz{0.08}{\triall},T,2n+1) &= 
\left\{\mytikz{0.2}{\triall}_{2n+1}, \mytikz{0.2}{\triall \fill (1.500, 2.598) -- (1.250, 2.165) -- (2.750, 2.165) -- (2.500, 2.598) -- cycle;}_{2n+1} \right\}, \\
P_{1,0}(\mytikz{0.08}{\triall},T,2n+1) \ \bigcap \ P_{1,2}(\mytikz{0.08}{\triall},T,2n+1) &= 
\left\{\mytikz{0.2}{\triall}_{2n+1} \right\},\\
P_{1,0}(\mytikz{0.08}{\triall},T,2n+1) \ \bigcap \ P_{2,2}(\mytikz{0.08}{\triall},T,2n+1) &= 
\left\{\mytikz{0.2}{\triall}_{2n+1}, \mytikz{0.2}{\triall \fill (0.500, 0.866) -- (1.000, 0.000) -- (1.500, 0.000) -- (0.750, 1.299) -- cycle;}_{2n+1} \right\}, \\
P_{1,2}(\mytikz{0.08}{\triall},T,2n+1) \ \bigcap \ P_{2,2}(\mytikz{0.08}{\triall},T,2n+1) &= 
\left\{\mytikz{0.2}{\triall}_{2n+1}, \mytikz{0.2}{\triall \fill (3.500, 0.866) -- (3.000, 0.000) -- (2.500, 0.000) -- (3.250, 1.299) -- cycle;}_{2n+1} \right\}.
\end{align*}
\end{lemma}

\begin{lemma}
Let $n\geq2$. Then 
\begin{align*}
P_{0,0}(\mytikz{0.08}{\tridown},T,2n) \ \bigcap \ P_{1,0}(\mytikz{0.08}{\tridown},T,2n) &= 
\left\{\mytikz{0.2}{\tridown}_{2n} \right\},\\
P_{0,0}(\mytikz{0.08}{\tridown},T,2n) \ \bigcap \ P_{1,2}(\mytikz{0.08}{\tridown},T,2n) &= 
\left\{\mytikz{0.2}{\tridown}_{2n} \right\},\\
P_{0,0}(\mytikz{0.08}{\tridown},T,2n) \ \bigcap \ P_{2,2}(\mytikz{0.08}{\tridown},T,2n) &= 
\left\{\mytikz{0.2}{\tridown}_{2n}, \mytikz{0.2}{\tridown \fill (0.000, 3.464) -- (4.000, 3.464) -- (3.750, 3.031) -- (0.250, 3.031) -- cycle;}_{2n} \right\}, \\
P_{1,0}(\mytikz{0.08}{\tridown},T,2n) \ \bigcap \ P_{1,2}(\mytikz{0.08}{\tridown},T,2n) &= 
\left\{\mytikz{0.2}{\tridown}_{2n} \right\},\\
P_{1,0}(\mytikz{0.08}{\tridown},T,2n) \ \bigcap \ P_{2,2}(\mytikz{0.08}{\tridown},T,2n) &= 
\left\{\mytikz{0.2}{\tridown}_{2n}, \mytikz{0.2}{\tridown \fill (4.000, 3.464) -- (2.000, 0.000) -- (1.750, 0.433) -- (3.500, 3.464) -- cycle;}_{2n} \right\}, \\
P_{1,2}(\mytikz{0.08}{\tridown},T,2n) \ \bigcap \ P_{2,2}(\mytikz{0.08}{\tridown},T,2n) &= 
\left\{\mytikz{0.2}{\tridown}_{2n}, \mytikz{0.2}{\tridown \fill (0.000, 3.464) -- (0.500, 3.464) -- (2.250, 0.433) -- (2.000, 0.000) -- cycle;}_{2n} \right\}.
\end{align*}
\end{lemma}

\begin{lemma}
\label{lemma:instersectionOddDown}
Let $n\geq2$. Then 
\begin{align*}
P_{0,0}(\mytikz{0.08}{\tridown},T,2n+1) \ \bigcap \ P_{1,0}(\mytikz{0.08}{\tridown},T,2n+1) &= 
\left\{\mytikz{0.2}{\tridown}_{2n+1} \right\},\\
P_{0,0}(\mytikz{0.08}{\tridown},T,2n+1) \ \bigcap \ P_{1,2}(\mytikz{0.08}{\tridown},T,2n+1) &= 
\left\{\mytikz{0.2}{\tridown}_{2n+1}, \mytikz{0.2}{\tridown \fill (4.000, 3.464) -- (2.000, 0.000) -- (1.750, 0.433) -- (3.500, 3.464) -- cycle;}_{2n+1} \right\},\\
P_{0,0}(\mytikz{0.08}{\tridown},T,2n+1) \ \bigcap \ P_{2,2}(\mytikz{0.08}{\tridown},T,2n+1) &= 
\left\{\mytikz{0.2}{\tridown}_{2n+1} \right\}, \\
P_{1,0}(\mytikz{0.08}{\tridown},T,2n+1) \ \bigcap \ P_{1,2}(\mytikz{0.08}{\tridown},T,2n+1) &= 
\left\{\mytikz{0.2}{\tridown}_{2n+1}, \mytikz{0.2}{\tridown \fill (0.000, 3.464) -- (4.000, 3.464) -- (3.750, 3.031) -- (0.250, 3.031) -- cycle;}_{2n+1} \right\},\\
P_{1,0}(\mytikz{0.08}{\tridown},T,2n+1) \ \bigcap \ P_{2,2}(\mytikz{0.08}{\tridown},T,2n+1) &= 
\left\{\mytikz{0.2}{\tridown}_{2n+1} \right\}, \\
P_{1,2}(\mytikz{0.08}{\tridown},T,2n+1) \ \bigcap \ P_{2,2}(\mytikz{0.08}{\tridown},T,2n+1) &= 
\left\{\mytikz{0.2}{\tridown}_{2n+1}, \mytikz{0.2}{\tridown \fill (0.000, 3.464) -- (0.500, 3.464) -- (2.250, 0.433) -- (2.000, 0.000) -- cycle;}_{2n+1} \right\}.
\end{align*}
\end{lemma}

We have now given the lemmas needed to prove the main result of this section. 

\begin{proof}[Proof of Theorem~\ref{tmm:additivity}]
From the lemmas above, (Lemma~\ref{lemma:instersectionEven} -- Lemma~\ref{lemma:instersectionOddDown}) we see that any of the sets $P_{r,c}(\alpha, T, n)$ contains the pattern with only unfilled unit-length triangles, meaning that this pattern is counted 4 times in the sum
\begin{equation}
\label{eq:sumPaTnm}
	\sum_{(r,c) \in I}|P_{r,c}(\alpha, T, n)|,
\end{equation}
that is, 3 times too many. Moreover, in each of the lemmas we see that there are precisely 3 intersections where 1 extra pattern is counted, besides the one with only unfilled triangles. Hence there is an additional over-count of 3 patterns in \eqref{eq:sumPaTnm}. This leads to the equality in \eqref{eq:additivity}.
\end{proof}

\section{Recursions}
\label{sec:recursions}

The aim of this section is to give a list of recursion relations for the size of the sets $P_{r,c}(\alpha, T, n)$. We start by introducing the following short hand 
\begin{equation}
\label{eq:defofABCD}
\begin{split}
	A_n &:= |P(\mytikz{0.08}{\trinone},  T, n)|, \\
	B_n &:= |P(\mytikz{0.08}{\tritop},   T, n)|, \\
	C_n &:= |P(\mytikz{0.08}{\trilower}, T, n)|, \\
	D_n &:= |P(\mytikz{0.08}{\triall},   T, n)|, 
\end{split}
\end{equation}
and 
\begin{equation}
\label{eq:defofAprim}
	A'_n := |P(\mytikz{0.08}{\tridown}, T, n)|. 
\end{equation}
The two quantities above, $A_n$ and $A_n'$, are the ones used in Theorem~\ref{thm:main}.

With the help of Theorem~\ref{tmm:additivity}, symmetry, and the definition of the substitution $\mu$ from \eqref{eq:DefOfMu} we now have for $n \geq 2$
\begin{equation}
\label{eq:recAeven}
\begin{split}
A_{2n} 
	&= -6 + \sum_{(r,c)\in I} |P_{r,c}(\mytikz{0.08}{\trinone},T, 2n)| \\
	&= -6 + A_{n} + B_{n+1} + B_{n+1} + B_{n+1}.
\end{split}
\end{equation}
See Figure~\ref{fig:recAeven} for a visualisation of the deduction of this recursion.

\begin{figure}[ht]
\begin{center}
\begin{tabular}{cc}
	\begin{tikzpicture}[scale = 0.4]
		\draw[color=black, fill=cyan] (6.000, 10.392) -- (3.000, 5.196) -- (9.000, 5.196) -- cycle;
		\draw[pattern={north east lines}] (6.000, 10.392) -- (3.000, 5.196) -- (9.000, 5.196) -- cycle;
		\draw[thin,dotted] (11.173, 0.300) -- (11.700, 1.212);
\draw[thin] (10.173, 0.300) -- (11.700, 2.944);
\draw[thin,dotted] (9.173, 0.300) -- (11.700, 4.677);
\draw[thin] (8.173, 0.300) -- (11.700, 6.409);
\draw[thin,dotted] (7.173, 0.300) -- (11.700, 8.141);
\draw[thin] (6.173, 0.300) -- (11.700, 9.873);
\draw[thin,dotted] (5.173, 0.300) -- (11.700, 11.605);
\draw[thin] (4.173, 0.300) -- (10.755, 11.700);
\draw[thin,dotted] (3.173, 0.300) -- (9.755, 11.700);
\draw[thin] (2.173, 0.300) -- (8.755, 11.700);
\draw[thin,dotted] (1.173, 0.300) -- (7.755, 11.700);
\draw[thin] (0.300, 0.520) -- (6.755, 11.700);
\draw[thin,dotted] (0.300, 2.252) -- (5.755, 11.700);
\draw[thin] (0.300, 3.984) -- (4.755, 11.700);
\draw[thin,dotted] (0.300, 5.716) -- (3.755, 11.700);
\draw[thin] (0.300, 7.448) -- (2.755, 11.700);
\draw[thin,dotted] (0.300, 9.180) -- (1.755, 11.700);
\draw[thin] (0.300, 10.912) -- (0.755, 11.700);
\draw[thin,dotted] (0.300, 1.212) -- (0.827, 0.300);
\draw[thin] (0.300, 2.944) -- (1.827, 0.300);
\draw[thin,dotted] (0.300, 4.677) -- (2.827, 0.300);
\draw[thin] (0.300, 6.409) -- (3.827, 0.300);
\draw[thin,dotted] (0.300, 8.141) -- (4.827, 0.300);
\draw[thin] (0.300, 9.873) -- (5.827, 0.300);
\draw[thin,dotted] (0.300, 11.605) -- (6.827, 0.300);
\draw[thin] (1.245, 11.700) -- (7.827, 0.300);
\draw[thin,dotted] (2.245, 11.700) -- (8.827, 0.300);
\draw[thin] (3.245, 11.700) -- (9.827, 0.300);
\draw[thin,dotted] (4.245, 11.700) -- (10.827, 0.300);
\draw[thin] (5.245, 11.700) -- (11.700, 0.520);
\draw[thin,dotted] (6.245, 11.700) -- (11.700, 2.252);
\draw[thin] (7.245, 11.700) -- (11.700, 3.984);
\draw[thin,dotted] (8.245, 11.700) -- (11.700, 5.716);
\draw[thin] (9.245, 11.700) -- (11.700, 7.448);
\draw[thin,dotted] (10.245, 11.700) -- (11.700, 9.180);
\draw[thin] (11.245, 11.700) -- (11.700, 10.912);
\draw[thin,dotted] (0.300, 0.866) -- (11.700, 0.866);
\draw[thin] (0.300, 1.732) -- (11.700, 1.732);
\draw[thin,dotted] (0.300, 2.598) -- (11.700, 2.598);
\draw[thin] (0.300, 3.464) -- (11.700, 3.464);
\draw[thin,dotted] (0.300, 4.330) -- (11.700, 4.330);
\draw[thin] (0.300, 5.196) -- (11.700, 5.196);
\draw[thin,dotted] (0.300, 6.062) -- (11.700, 6.062);
\draw[thin] (0.300, 6.928) -- (11.700, 6.928);
\draw[thin,dotted] (0.300, 7.794) -- (11.700, 7.794);
\draw[thin] (0.300, 8.660) -- (11.700, 8.660);
\draw[thin,dotted] (0.300, 9.526) -- (11.700, 9.526);
\draw[thin] (0.300, 10.392) -- (11.700, 10.392);
\draw[thin,dotted] (0.300, 11.258) -- (11.700, 11.258);		
		\draw(6,-0.75) node{\footnotesize $(r,c) = (0,0)$};	
	\end{tikzpicture}
&
	\begin{tikzpicture}[scale = 0.4]
		\draw[color=black, fill=cyan] (5.500, 9.526) -- (2.500, 4.330) -- (8.500, 4.330) -- cycle;
		\draw[pattern={north east lines}] (6.000,10.392) -- (2.000, 3.464) -- (8.000, 3.464) -- (9.000, 5.196) -- cycle;
		\draw[thin,dotted] (11.173, 0.300) -- (11.700, 1.212);
\draw[thin] (10.173, 0.300) -- (11.700, 2.944);
\draw[thin,dotted] (9.173, 0.300) -- (11.700, 4.677);
\draw[thin] (8.173, 0.300) -- (11.700, 6.409);
\draw[thin,dotted] (7.173, 0.300) -- (11.700, 8.141);
\draw[thin] (6.173, 0.300) -- (11.700, 9.873);
\draw[thin,dotted] (5.173, 0.300) -- (11.700, 11.605);
\draw[thin] (4.173, 0.300) -- (10.755, 11.700);
\draw[thin,dotted] (3.173, 0.300) -- (9.755, 11.700);
\draw[thin] (2.173, 0.300) -- (8.755, 11.700);
\draw[thin,dotted] (1.173, 0.300) -- (7.755, 11.700);
\draw[thin] (0.300, 0.520) -- (6.755, 11.700);
\draw[thin,dotted] (0.300, 2.252) -- (5.755, 11.700);
\draw[thin] (0.300, 3.984) -- (4.755, 11.700);
\draw[thin,dotted] (0.300, 5.716) -- (3.755, 11.700);
\draw[thin] (0.300, 7.448) -- (2.755, 11.700);
\draw[thin,dotted] (0.300, 9.180) -- (1.755, 11.700);
\draw[thin] (0.300, 10.912) -- (0.755, 11.700);
\draw[thin,dotted] (0.300, 1.212) -- (0.827, 0.300);
\draw[thin] (0.300, 2.944) -- (1.827, 0.300);
\draw[thin,dotted] (0.300, 4.677) -- (2.827, 0.300);
\draw[thin] (0.300, 6.409) -- (3.827, 0.300);
\draw[thin,dotted] (0.300, 8.141) -- (4.827, 0.300);
\draw[thin] (0.300, 9.873) -- (5.827, 0.300);
\draw[thin,dotted] (0.300, 11.605) -- (6.827, 0.300);
\draw[thin] (1.245, 11.700) -- (7.827, 0.300);
\draw[thin,dotted] (2.245, 11.700) -- (8.827, 0.300);
\draw[thin] (3.245, 11.700) -- (9.827, 0.300);
\draw[thin,dotted] (4.245, 11.700) -- (10.827, 0.300);
\draw[thin] (5.245, 11.700) -- (11.700, 0.520);
\draw[thin,dotted] (6.245, 11.700) -- (11.700, 2.252);
\draw[thin] (7.245, 11.700) -- (11.700, 3.984);
\draw[thin,dotted] (8.245, 11.700) -- (11.700, 5.716);
\draw[thin] (9.245, 11.700) -- (11.700, 7.448);
\draw[thin,dotted] (10.245, 11.700) -- (11.700, 9.180);
\draw[thin] (11.245, 11.700) -- (11.700, 10.912);
\draw[thin,dotted] (0.300, 0.866) -- (11.700, 0.866);
\draw[thin] (0.300, 1.732) -- (11.700, 1.732);
\draw[thin,dotted] (0.300, 2.598) -- (11.700, 2.598);
\draw[thin] (0.300, 3.464) -- (11.700, 3.464);
\draw[thin,dotted] (0.300, 4.330) -- (11.700, 4.330);
\draw[thin] (0.300, 5.196) -- (11.700, 5.196);
\draw[thin,dotted] (0.300, 6.062) -- (11.700, 6.062);
\draw[thin] (0.300, 6.928) -- (11.700, 6.928);
\draw[thin,dotted] (0.300, 7.794) -- (11.700, 7.794);
\draw[thin] (0.300, 8.660) -- (11.700, 8.660);
\draw[thin,dotted] (0.300, 9.526) -- (11.700, 9.526);
\draw[thin] (0.300, 10.392) -- (11.700, 10.392);
\draw[thin,dotted] (0.300, 11.258) -- (11.700, 11.258);		
		\draw(6,-0.75) node{\footnotesize $(r,c) = (1,0)$};	
	\end{tikzpicture}
\\
	\begin{tikzpicture}[scale = 0.4]
		\draw[color=black, fill=cyan] (6.500, 9.526) -- (3.500, 4.330) -- (9.500, 4.330) -- cycle;
		\draw[pattern={north east lines}] (6.000, 10.392) -- (3.000, 5.196) -- (4.000, 3.464) -- (10.000, 3.464) -- cycle;
		\draw[thin,dotted] (11.173, 0.300) -- (11.700, 1.212);
\draw[thin] (10.173, 0.300) -- (11.700, 2.944);
\draw[thin,dotted] (9.173, 0.300) -- (11.700, 4.677);
\draw[thin] (8.173, 0.300) -- (11.700, 6.409);
\draw[thin,dotted] (7.173, 0.300) -- (11.700, 8.141);
\draw[thin] (6.173, 0.300) -- (11.700, 9.873);
\draw[thin,dotted] (5.173, 0.300) -- (11.700, 11.605);
\draw[thin] (4.173, 0.300) -- (10.755, 11.700);
\draw[thin,dotted] (3.173, 0.300) -- (9.755, 11.700);
\draw[thin] (2.173, 0.300) -- (8.755, 11.700);
\draw[thin,dotted] (1.173, 0.300) -- (7.755, 11.700);
\draw[thin] (0.300, 0.520) -- (6.755, 11.700);
\draw[thin,dotted] (0.300, 2.252) -- (5.755, 11.700);
\draw[thin] (0.300, 3.984) -- (4.755, 11.700);
\draw[thin,dotted] (0.300, 5.716) -- (3.755, 11.700);
\draw[thin] (0.300, 7.448) -- (2.755, 11.700);
\draw[thin,dotted] (0.300, 9.180) -- (1.755, 11.700);
\draw[thin] (0.300, 10.912) -- (0.755, 11.700);
\draw[thin,dotted] (0.300, 1.212) -- (0.827, 0.300);
\draw[thin] (0.300, 2.944) -- (1.827, 0.300);
\draw[thin,dotted] (0.300, 4.677) -- (2.827, 0.300);
\draw[thin] (0.300, 6.409) -- (3.827, 0.300);
\draw[thin,dotted] (0.300, 8.141) -- (4.827, 0.300);
\draw[thin] (0.300, 9.873) -- (5.827, 0.300);
\draw[thin,dotted] (0.300, 11.605) -- (6.827, 0.300);
\draw[thin] (1.245, 11.700) -- (7.827, 0.300);
\draw[thin,dotted] (2.245, 11.700) -- (8.827, 0.300);
\draw[thin] (3.245, 11.700) -- (9.827, 0.300);
\draw[thin,dotted] (4.245, 11.700) -- (10.827, 0.300);
\draw[thin] (5.245, 11.700) -- (11.700, 0.520);
\draw[thin,dotted] (6.245, 11.700) -- (11.700, 2.252);
\draw[thin] (7.245, 11.700) -- (11.700, 3.984);
\draw[thin,dotted] (8.245, 11.700) -- (11.700, 5.716);
\draw[thin] (9.245, 11.700) -- (11.700, 7.448);
\draw[thin,dotted] (10.245, 11.700) -- (11.700, 9.180);
\draw[thin] (11.245, 11.700) -- (11.700, 10.912);
\draw[thin,dotted] (0.300, 0.866) -- (11.700, 0.866);
\draw[thin] (0.300, 1.732) -- (11.700, 1.732);
\draw[thin,dotted] (0.300, 2.598) -- (11.700, 2.598);
\draw[thin] (0.300, 3.464) -- (11.700, 3.464);
\draw[thin,dotted] (0.300, 4.330) -- (11.700, 4.330);
\draw[thin] (0.300, 5.196) -- (11.700, 5.196);
\draw[thin,dotted] (0.300, 6.062) -- (11.700, 6.062);
\draw[thin] (0.300, 6.928) -- (11.700, 6.928);
\draw[thin,dotted] (0.300, 7.794) -- (11.700, 7.794);
\draw[thin] (0.300, 8.660) -- (11.700, 8.660);
\draw[thin,dotted] (0.300, 9.526) -- (11.700, 9.526);
\draw[thin] (0.300, 10.392) -- (11.700, 10.392);
\draw[thin,dotted] (0.300, 11.258) -- (11.700, 11.258);		
		\draw(6,-0.75) node{\footnotesize $(r,c) = (1,2)$};	
	\end{tikzpicture}
&
	\begin{tikzpicture}[scale = 0.4]
		\draw[color=black, fill=cyan] (6.000, 8.660) -- (3.000, 3.464) -- (9.000, 3.464) -- cycle;
		\draw[pattern={north east lines}] (5.000, 8.660) -- (2.000, 3.464) -- (10.000, 3.464) -- (7.000, 8.660) -- cycle;
		\draw[thin,dotted] (11.173, 0.300) -- (11.700, 1.212);
\draw[thin] (10.173, 0.300) -- (11.700, 2.944);
\draw[thin,dotted] (9.173, 0.300) -- (11.700, 4.677);
\draw[thin] (8.173, 0.300) -- (11.700, 6.409);
\draw[thin,dotted] (7.173, 0.300) -- (11.700, 8.141);
\draw[thin] (6.173, 0.300) -- (11.700, 9.873);
\draw[thin,dotted] (5.173, 0.300) -- (11.700, 11.605);
\draw[thin] (4.173, 0.300) -- (10.755, 11.700);
\draw[thin,dotted] (3.173, 0.300) -- (9.755, 11.700);
\draw[thin] (2.173, 0.300) -- (8.755, 11.700);
\draw[thin,dotted] (1.173, 0.300) -- (7.755, 11.700);
\draw[thin] (0.300, 0.520) -- (6.755, 11.700);
\draw[thin,dotted] (0.300, 2.252) -- (5.755, 11.700);
\draw[thin] (0.300, 3.984) -- (4.755, 11.700);
\draw[thin,dotted] (0.300, 5.716) -- (3.755, 11.700);
\draw[thin] (0.300, 7.448) -- (2.755, 11.700);
\draw[thin,dotted] (0.300, 9.180) -- (1.755, 11.700);
\draw[thin] (0.300, 10.912) -- (0.755, 11.700);
\draw[thin,dotted] (0.300, 1.212) -- (0.827, 0.300);
\draw[thin] (0.300, 2.944) -- (1.827, 0.300);
\draw[thin,dotted] (0.300, 4.677) -- (2.827, 0.300);
\draw[thin] (0.300, 6.409) -- (3.827, 0.300);
\draw[thin,dotted] (0.300, 8.141) -- (4.827, 0.300);
\draw[thin] (0.300, 9.873) -- (5.827, 0.300);
\draw[thin,dotted] (0.300, 11.605) -- (6.827, 0.300);
\draw[thin] (1.245, 11.700) -- (7.827, 0.300);
\draw[thin,dotted] (2.245, 11.700) -- (8.827, 0.300);
\draw[thin] (3.245, 11.700) -- (9.827, 0.300);
\draw[thin,dotted] (4.245, 11.700) -- (10.827, 0.300);
\draw[thin] (5.245, 11.700) -- (11.700, 0.520);
\draw[thin,dotted] (6.245, 11.700) -- (11.700, 2.252);
\draw[thin] (7.245, 11.700) -- (11.700, 3.984);
\draw[thin,dotted] (8.245, 11.700) -- (11.700, 5.716);
\draw[thin] (9.245, 11.700) -- (11.700, 7.448);
\draw[thin,dotted] (10.245, 11.700) -- (11.700, 9.180);
\draw[thin] (11.245, 11.700) -- (11.700, 10.912);
\draw[thin,dotted] (0.300, 0.866) -- (11.700, 0.866);
\draw[thin] (0.300, 1.732) -- (11.700, 1.732);
\draw[thin,dotted] (0.300, 2.598) -- (11.700, 2.598);
\draw[thin] (0.300, 3.464) -- (11.700, 3.464);
\draw[thin,dotted] (0.300, 4.330) -- (11.700, 4.330);
\draw[thin] (0.300, 5.196) -- (11.700, 5.196);
\draw[thin,dotted] (0.300, 6.062) -- (11.700, 6.062);
\draw[thin] (0.300, 6.928) -- (11.700, 6.928);
\draw[thin,dotted] (0.300, 7.794) -- (11.700, 7.794);
\draw[thin] (0.300, 8.660) -- (11.700, 8.660);
\draw[thin,dotted] (0.300, 9.526) -- (11.700, 9.526);
\draw[thin] (0.300, 10.392) -- (11.700, 10.392);
\draw[thin,dotted] (0.300, 11.258) -- (11.700, 11.258);		
		\draw(6,-0.75) node{\footnotesize $(r,c) = (2,2)$};	
	\end{tikzpicture}
\end{tabular}
\end{center}
\caption{The blue shaded regions represent elements of $P_{r,c}(\mytikz{0.08}{\protect\trinone},T, 2n)$, for the different values of $(r,c)$. The illustrations show how these are modified to reach the hatched regions, from which we can deduce the recursive expression in \eqref{eq:recAeven}.}
\label{fig:recAeven}
\end{figure}

In the same way, for odd side length, we have 
\begin{equation}
\label{eq:recAodd}
\begin{split}
A_{2n+1} 
	&= -6 + \sum_{(r,c)\in I} |P_{r,c}(\mytikz{0.08}{\trinone},T, 2n+1)| \\
	&= -6 + A_{n+1} + A_{n+1} + A_{n+1} + D_{n+2}.
\end{split}
\end{equation}
See Figure~\ref{fig:recAodd} for a visualisation of the deduction of this recursion.

\begin{figure}[ht]
\begin{center}
\begin{tabular}{cc}
	\begin{tikzpicture}[scale = 0.4]
		\draw[color=black, fill=cyan] (6.000, 10.392) -- (2.500, 4.330) -- (9.500, 4.330) -- cycle;
		\draw[pattern={north east lines}] (6.000, 10.392) -- (2.000, 3.464) -- (10.000, 3.464) -- cycle;
		\draw[thin,dotted] (11.173, 0.300) -- (11.700, 1.212);
\draw[thin] (10.173, 0.300) -- (11.700, 2.944);
\draw[thin,dotted] (9.173, 0.300) -- (11.700, 4.677);
\draw[thin] (8.173, 0.300) -- (11.700, 6.409);
\draw[thin,dotted] (7.173, 0.300) -- (11.700, 8.141);
\draw[thin] (6.173, 0.300) -- (11.700, 9.873);
\draw[thin,dotted] (5.173, 0.300) -- (11.700, 11.605);
\draw[thin] (4.173, 0.300) -- (10.755, 11.700);
\draw[thin,dotted] (3.173, 0.300) -- (9.755, 11.700);
\draw[thin] (2.173, 0.300) -- (8.755, 11.700);
\draw[thin,dotted] (1.173, 0.300) -- (7.755, 11.700);
\draw[thin] (0.300, 0.520) -- (6.755, 11.700);
\draw[thin,dotted] (0.300, 2.252) -- (5.755, 11.700);
\draw[thin] (0.300, 3.984) -- (4.755, 11.700);
\draw[thin,dotted] (0.300, 5.716) -- (3.755, 11.700);
\draw[thin] (0.300, 7.448) -- (2.755, 11.700);
\draw[thin,dotted] (0.300, 9.180) -- (1.755, 11.700);
\draw[thin] (0.300, 10.912) -- (0.755, 11.700);
\draw[thin,dotted] (0.300, 1.212) -- (0.827, 0.300);
\draw[thin] (0.300, 2.944) -- (1.827, 0.300);
\draw[thin,dotted] (0.300, 4.677) -- (2.827, 0.300);
\draw[thin] (0.300, 6.409) -- (3.827, 0.300);
\draw[thin,dotted] (0.300, 8.141) -- (4.827, 0.300);
\draw[thin] (0.300, 9.873) -- (5.827, 0.300);
\draw[thin,dotted] (0.300, 11.605) -- (6.827, 0.300);
\draw[thin] (1.245, 11.700) -- (7.827, 0.300);
\draw[thin,dotted] (2.245, 11.700) -- (8.827, 0.300);
\draw[thin] (3.245, 11.700) -- (9.827, 0.300);
\draw[thin,dotted] (4.245, 11.700) -- (10.827, 0.300);
\draw[thin] (5.245, 11.700) -- (11.700, 0.520);
\draw[thin,dotted] (6.245, 11.700) -- (11.700, 2.252);
\draw[thin] (7.245, 11.700) -- (11.700, 3.984);
\draw[thin,dotted] (8.245, 11.700) -- (11.700, 5.716);
\draw[thin] (9.245, 11.700) -- (11.700, 7.448);
\draw[thin,dotted] (10.245, 11.700) -- (11.700, 9.180);
\draw[thin] (11.245, 11.700) -- (11.700, 10.912);
\draw[thin,dotted] (0.300, 0.866) -- (11.700, 0.866);
\draw[thin] (0.300, 1.732) -- (11.700, 1.732);
\draw[thin,dotted] (0.300, 2.598) -- (11.700, 2.598);
\draw[thin] (0.300, 3.464) -- (11.700, 3.464);
\draw[thin,dotted] (0.300, 4.330) -- (11.700, 4.330);
\draw[thin] (0.300, 5.196) -- (11.700, 5.196);
\draw[thin,dotted] (0.300, 6.062) -- (11.700, 6.062);
\draw[thin] (0.300, 6.928) -- (11.700, 6.928);
\draw[thin,dotted] (0.300, 7.794) -- (11.700, 7.794);
\draw[thin] (0.300, 8.660) -- (11.700, 8.660);
\draw[thin,dotted] (0.300, 9.526) -- (11.700, 9.526);
\draw[thin] (0.300, 10.392) -- (11.700, 10.392);
\draw[thin,dotted] (0.300, 11.258) -- (11.700, 11.258);		
		\draw(6,-0.75) node{\footnotesize $(r,c) = (0,0)$};	
	\end{tikzpicture}
&
	\begin{tikzpicture}[scale = 0.4]
		\draw[color=black, fill=cyan] (5.500, 9.526) -- (2.000, 3.464) -- (9.000, 3.464) -- cycle;
		\draw[pattern={north east lines}] (6.000, 10.392) -- (2.000, 3.464) -- (10.000, 3.464) -- cycle;
		\draw[thin,dotted] (11.173, 0.300) -- (11.700, 1.212);
\draw[thin] (10.173, 0.300) -- (11.700, 2.944);
\draw[thin,dotted] (9.173, 0.300) -- (11.700, 4.677);
\draw[thin] (8.173, 0.300) -- (11.700, 6.409);
\draw[thin,dotted] (7.173, 0.300) -- (11.700, 8.141);
\draw[thin] (6.173, 0.300) -- (11.700, 9.873);
\draw[thin,dotted] (5.173, 0.300) -- (11.700, 11.605);
\draw[thin] (4.173, 0.300) -- (10.755, 11.700);
\draw[thin,dotted] (3.173, 0.300) -- (9.755, 11.700);
\draw[thin] (2.173, 0.300) -- (8.755, 11.700);
\draw[thin,dotted] (1.173, 0.300) -- (7.755, 11.700);
\draw[thin] (0.300, 0.520) -- (6.755, 11.700);
\draw[thin,dotted] (0.300, 2.252) -- (5.755, 11.700);
\draw[thin] (0.300, 3.984) -- (4.755, 11.700);
\draw[thin,dotted] (0.300, 5.716) -- (3.755, 11.700);
\draw[thin] (0.300, 7.448) -- (2.755, 11.700);
\draw[thin,dotted] (0.300, 9.180) -- (1.755, 11.700);
\draw[thin] (0.300, 10.912) -- (0.755, 11.700);
\draw[thin,dotted] (0.300, 1.212) -- (0.827, 0.300);
\draw[thin] (0.300, 2.944) -- (1.827, 0.300);
\draw[thin,dotted] (0.300, 4.677) -- (2.827, 0.300);
\draw[thin] (0.300, 6.409) -- (3.827, 0.300);
\draw[thin,dotted] (0.300, 8.141) -- (4.827, 0.300);
\draw[thin] (0.300, 9.873) -- (5.827, 0.300);
\draw[thin,dotted] (0.300, 11.605) -- (6.827, 0.300);
\draw[thin] (1.245, 11.700) -- (7.827, 0.300);
\draw[thin,dotted] (2.245, 11.700) -- (8.827, 0.300);
\draw[thin] (3.245, 11.700) -- (9.827, 0.300);
\draw[thin,dotted] (4.245, 11.700) -- (10.827, 0.300);
\draw[thin] (5.245, 11.700) -- (11.700, 0.520);
\draw[thin,dotted] (6.245, 11.700) -- (11.700, 2.252);
\draw[thin] (7.245, 11.700) -- (11.700, 3.984);
\draw[thin,dotted] (8.245, 11.700) -- (11.700, 5.716);
\draw[thin] (9.245, 11.700) -- (11.700, 7.448);
\draw[thin,dotted] (10.245, 11.700) -- (11.700, 9.180);
\draw[thin] (11.245, 11.700) -- (11.700, 10.912);
\draw[thin,dotted] (0.300, 0.866) -- (11.700, 0.866);
\draw[thin] (0.300, 1.732) -- (11.700, 1.732);
\draw[thin,dotted] (0.300, 2.598) -- (11.700, 2.598);
\draw[thin] (0.300, 3.464) -- (11.700, 3.464);
\draw[thin,dotted] (0.300, 4.330) -- (11.700, 4.330);
\draw[thin] (0.300, 5.196) -- (11.700, 5.196);
\draw[thin,dotted] (0.300, 6.062) -- (11.700, 6.062);
\draw[thin] (0.300, 6.928) -- (11.700, 6.928);
\draw[thin,dotted] (0.300, 7.794) -- (11.700, 7.794);
\draw[thin] (0.300, 8.660) -- (11.700, 8.660);
\draw[thin,dotted] (0.300, 9.526) -- (11.700, 9.526);
\draw[thin] (0.300, 10.392) -- (11.700, 10.392);
\draw[thin,dotted] (0.300, 11.258) -- (11.700, 11.258);		
		\draw(6,-0.75) node{\footnotesize $(r,c) = (1,0)$};	
	\end{tikzpicture}
\\
	\begin{tikzpicture}[scale = 0.4]
		\draw[color=black, fill=cyan] (6.500, 9.526) -- (3.000, 3.464) -- (10.000, 3.464) -- cycle;
		\draw[pattern={north east lines}] (6.000, 10.392) -- (2.000, 3.464) -- (10.000, 3.464) -- cycle;
		\draw[thin,dotted] (11.173, 0.300) -- (11.700, 1.212);
\draw[thin] (10.173, 0.300) -- (11.700, 2.944);
\draw[thin,dotted] (9.173, 0.300) -- (11.700, 4.677);
\draw[thin] (8.173, 0.300) -- (11.700, 6.409);
\draw[thin,dotted] (7.173, 0.300) -- (11.700, 8.141);
\draw[thin] (6.173, 0.300) -- (11.700, 9.873);
\draw[thin,dotted] (5.173, 0.300) -- (11.700, 11.605);
\draw[thin] (4.173, 0.300) -- (10.755, 11.700);
\draw[thin,dotted] (3.173, 0.300) -- (9.755, 11.700);
\draw[thin] (2.173, 0.300) -- (8.755, 11.700);
\draw[thin,dotted] (1.173, 0.300) -- (7.755, 11.700);
\draw[thin] (0.300, 0.520) -- (6.755, 11.700);
\draw[thin,dotted] (0.300, 2.252) -- (5.755, 11.700);
\draw[thin] (0.300, 3.984) -- (4.755, 11.700);
\draw[thin,dotted] (0.300, 5.716) -- (3.755, 11.700);
\draw[thin] (0.300, 7.448) -- (2.755, 11.700);
\draw[thin,dotted] (0.300, 9.180) -- (1.755, 11.700);
\draw[thin] (0.300, 10.912) -- (0.755, 11.700);
\draw[thin,dotted] (0.300, 1.212) -- (0.827, 0.300);
\draw[thin] (0.300, 2.944) -- (1.827, 0.300);
\draw[thin,dotted] (0.300, 4.677) -- (2.827, 0.300);
\draw[thin] (0.300, 6.409) -- (3.827, 0.300);
\draw[thin,dotted] (0.300, 8.141) -- (4.827, 0.300);
\draw[thin] (0.300, 9.873) -- (5.827, 0.300);
\draw[thin,dotted] (0.300, 11.605) -- (6.827, 0.300);
\draw[thin] (1.245, 11.700) -- (7.827, 0.300);
\draw[thin,dotted] (2.245, 11.700) -- (8.827, 0.300);
\draw[thin] (3.245, 11.700) -- (9.827, 0.300);
\draw[thin,dotted] (4.245, 11.700) -- (10.827, 0.300);
\draw[thin] (5.245, 11.700) -- (11.700, 0.520);
\draw[thin,dotted] (6.245, 11.700) -- (11.700, 2.252);
\draw[thin] (7.245, 11.700) -- (11.700, 3.984);
\draw[thin,dotted] (8.245, 11.700) -- (11.700, 5.716);
\draw[thin] (9.245, 11.700) -- (11.700, 7.448);
\draw[thin,dotted] (10.245, 11.700) -- (11.700, 9.180);
\draw[thin] (11.245, 11.700) -- (11.700, 10.912);
\draw[thin,dotted] (0.300, 0.866) -- (11.700, 0.866);
\draw[thin] (0.300, 1.732) -- (11.700, 1.732);
\draw[thin,dotted] (0.300, 2.598) -- (11.700, 2.598);
\draw[thin] (0.300, 3.464) -- (11.700, 3.464);
\draw[thin,dotted] (0.300, 4.330) -- (11.700, 4.330);
\draw[thin] (0.300, 5.196) -- (11.700, 5.196);
\draw[thin,dotted] (0.300, 6.062) -- (11.700, 6.062);
\draw[thin] (0.300, 6.928) -- (11.700, 6.928);
\draw[thin,dotted] (0.300, 7.794) -- (11.700, 7.794);
\draw[thin] (0.300, 8.660) -- (11.700, 8.660);
\draw[thin,dotted] (0.300, 9.526) -- (11.700, 9.526);
\draw[thin] (0.300, 10.392) -- (11.700, 10.392);
\draw[thin,dotted] (0.300, 11.258) -- (11.700, 11.258);		
		\draw(6,-0.75) node{\footnotesize $(r,c) = (1,2)$};	
	\end{tikzpicture}
&
	\begin{tikzpicture}[scale = 0.4]
		\draw[color=black, fill=cyan] (6.000, 8.660) -- (2.500, 2.598) -- (9.500, 2.598) -- cycle;
		\draw[pattern={north east lines}] (5.000, 8.660) -- (2.000, 3.464) -- (3.000, 1.732) -- (9.000, 1.732) -- (10.000, 3.464) -- (7.000, 8.660) -- cycle;
		\draw[thin,dotted] (11.173, 0.300) -- (11.700, 1.212);
\draw[thin] (10.173, 0.300) -- (11.700, 2.944);
\draw[thin,dotted] (9.173, 0.300) -- (11.700, 4.677);
\draw[thin] (8.173, 0.300) -- (11.700, 6.409);
\draw[thin,dotted] (7.173, 0.300) -- (11.700, 8.141);
\draw[thin] (6.173, 0.300) -- (11.700, 9.873);
\draw[thin,dotted] (5.173, 0.300) -- (11.700, 11.605);
\draw[thin] (4.173, 0.300) -- (10.755, 11.700);
\draw[thin,dotted] (3.173, 0.300) -- (9.755, 11.700);
\draw[thin] (2.173, 0.300) -- (8.755, 11.700);
\draw[thin,dotted] (1.173, 0.300) -- (7.755, 11.700);
\draw[thin] (0.300, 0.520) -- (6.755, 11.700);
\draw[thin,dotted] (0.300, 2.252) -- (5.755, 11.700);
\draw[thin] (0.300, 3.984) -- (4.755, 11.700);
\draw[thin,dotted] (0.300, 5.716) -- (3.755, 11.700);
\draw[thin] (0.300, 7.448) -- (2.755, 11.700);
\draw[thin,dotted] (0.300, 9.180) -- (1.755, 11.700);
\draw[thin] (0.300, 10.912) -- (0.755, 11.700);
\draw[thin,dotted] (0.300, 1.212) -- (0.827, 0.300);
\draw[thin] (0.300, 2.944) -- (1.827, 0.300);
\draw[thin,dotted] (0.300, 4.677) -- (2.827, 0.300);
\draw[thin] (0.300, 6.409) -- (3.827, 0.300);
\draw[thin,dotted] (0.300, 8.141) -- (4.827, 0.300);
\draw[thin] (0.300, 9.873) -- (5.827, 0.300);
\draw[thin,dotted] (0.300, 11.605) -- (6.827, 0.300);
\draw[thin] (1.245, 11.700) -- (7.827, 0.300);
\draw[thin,dotted] (2.245, 11.700) -- (8.827, 0.300);
\draw[thin] (3.245, 11.700) -- (9.827, 0.300);
\draw[thin,dotted] (4.245, 11.700) -- (10.827, 0.300);
\draw[thin] (5.245, 11.700) -- (11.700, 0.520);
\draw[thin,dotted] (6.245, 11.700) -- (11.700, 2.252);
\draw[thin] (7.245, 11.700) -- (11.700, 3.984);
\draw[thin,dotted] (8.245, 11.700) -- (11.700, 5.716);
\draw[thin] (9.245, 11.700) -- (11.700, 7.448);
\draw[thin,dotted] (10.245, 11.700) -- (11.700, 9.180);
\draw[thin] (11.245, 11.700) -- (11.700, 10.912);
\draw[thin,dotted] (0.300, 0.866) -- (11.700, 0.866);
\draw[thin] (0.300, 1.732) -- (11.700, 1.732);
\draw[thin,dotted] (0.300, 2.598) -- (11.700, 2.598);
\draw[thin] (0.300, 3.464) -- (11.700, 3.464);
\draw[thin,dotted] (0.300, 4.330) -- (11.700, 4.330);
\draw[thin] (0.300, 5.196) -- (11.700, 5.196);
\draw[thin,dotted] (0.300, 6.062) -- (11.700, 6.062);
\draw[thin] (0.300, 6.928) -- (11.700, 6.928);
\draw[thin,dotted] (0.300, 7.794) -- (11.700, 7.794);
\draw[thin] (0.300, 8.660) -- (11.700, 8.660);
\draw[thin,dotted] (0.300, 9.526) -- (11.700, 9.526);
\draw[thin] (0.300, 10.392) -- (11.700, 10.392);
\draw[thin,dotted] (0.300, 11.258) -- (11.700, 11.258);
		\draw(6,-0.75) node{\footnotesize $(r,c) = (2,2)$};	
	\end{tikzpicture}
\end{tabular}
\end{center}
\caption{The blue shaded regions represent elements of $P_{r,c}(\mytikz{0.07}{\protect\trinone},T, 2n+1)$, for the different values of $(r,c)$. The illustrations show how these are modified to reach the hatched regions, from which we can deduce the recursive expression in \eqref{eq:recAodd}.}
\label{fig:recAodd}
\end{figure}

\begin{equation}
\label{eq:recBeven}
\begin{split}
B_{2n} 
&= -6 + \sum_{(r,c)\in I} |P_{r,c}(\mytikz{0.08}{\tritop},T, 2n)| \\
&= -6 + A_{n} + C_{n+1} + C_{n+1} + B_{n+1}. 
\end{split}
\end{equation}
See Figure~\ref{fig:recBeven} for a visualisation of the deduction of this recursion.

\begin{figure}[ht]
\begin{center}
\begin{tabular}{cc}
	\begin{tikzpicture}[scale = 0.4]
		\draw[color=black, fill=cyan] (5.500, 9.526) -- (3.000, 5.196) -- (9.000, 5.196) -- (6.500, 9.526) -- cycle;
		\draw[pattern={north east lines}] (6.000, 10.392) -- (3.000, 5.196) -- (9.000, 5.196) -- cycle;
		\draw[thin,dotted] (11.173, 0.300) -- (11.700, 1.212);
\draw[thin] (10.173, 0.300) -- (11.700, 2.944);
\draw[thin,dotted] (9.173, 0.300) -- (11.700, 4.677);
\draw[thin] (8.173, 0.300) -- (11.700, 6.409);
\draw[thin,dotted] (7.173, 0.300) -- (11.700, 8.141);
\draw[thin] (6.173, 0.300) -- (11.700, 9.873);
\draw[thin,dotted] (5.173, 0.300) -- (11.700, 11.605);
\draw[thin] (4.173, 0.300) -- (10.755, 11.700);
\draw[thin,dotted] (3.173, 0.300) -- (9.755, 11.700);
\draw[thin] (2.173, 0.300) -- (8.755, 11.700);
\draw[thin,dotted] (1.173, 0.300) -- (7.755, 11.700);
\draw[thin] (0.300, 0.520) -- (6.755, 11.700);
\draw[thin,dotted] (0.300, 2.252) -- (5.755, 11.700);
\draw[thin] (0.300, 3.984) -- (4.755, 11.700);
\draw[thin,dotted] (0.300, 5.716) -- (3.755, 11.700);
\draw[thin] (0.300, 7.448) -- (2.755, 11.700);
\draw[thin,dotted] (0.300, 9.180) -- (1.755, 11.700);
\draw[thin] (0.300, 10.912) -- (0.755, 11.700);
\draw[thin,dotted] (0.300, 1.212) -- (0.827, 0.300);
\draw[thin] (0.300, 2.944) -- (1.827, 0.300);
\draw[thin,dotted] (0.300, 4.677) -- (2.827, 0.300);
\draw[thin] (0.300, 6.409) -- (3.827, 0.300);
\draw[thin,dotted] (0.300, 8.141) -- (4.827, 0.300);
\draw[thin] (0.300, 9.873) -- (5.827, 0.300);
\draw[thin,dotted] (0.300, 11.605) -- (6.827, 0.300);
\draw[thin] (1.245, 11.700) -- (7.827, 0.300);
\draw[thin,dotted] (2.245, 11.700) -- (8.827, 0.300);
\draw[thin] (3.245, 11.700) -- (9.827, 0.300);
\draw[thin,dotted] (4.245, 11.700) -- (10.827, 0.300);
\draw[thin] (5.245, 11.700) -- (11.700, 0.520);
\draw[thin,dotted] (6.245, 11.700) -- (11.700, 2.252);
\draw[thin] (7.245, 11.700) -- (11.700, 3.984);
\draw[thin,dotted] (8.245, 11.700) -- (11.700, 5.716);
\draw[thin] (9.245, 11.700) -- (11.700, 7.448);
\draw[thin,dotted] (10.245, 11.700) -- (11.700, 9.180);
\draw[thin] (11.245, 11.700) -- (11.700, 10.912);
\draw[thin,dotted] (0.300, 0.866) -- (11.700, 0.866);
\draw[thin] (0.300, 1.732) -- (11.700, 1.732);
\draw[thin,dotted] (0.300, 2.598) -- (11.700, 2.598);
\draw[thin] (0.300, 3.464) -- (11.700, 3.464);
\draw[thin,dotted] (0.300, 4.330) -- (11.700, 4.330);
\draw[thin] (0.300, 5.196) -- (11.700, 5.196);
\draw[thin,dotted] (0.300, 6.062) -- (11.700, 6.062);
\draw[thin] (0.300, 6.928) -- (11.700, 6.928);
\draw[thin,dotted] (0.300, 7.794) -- (11.700, 7.794);
\draw[thin] (0.300, 8.660) -- (11.700, 8.660);
\draw[thin,dotted] (0.300, 9.526) -- (11.700, 9.526);
\draw[thin] (0.300, 10.392) -- (11.700, 10.392);
\draw[thin,dotted] (0.300, 11.258) -- (11.700, 11.258);		
		\draw(6,-0.75) node{\footnotesize $(r,c) = (0,0)$};	
	\end{tikzpicture}
&
	\begin{tikzpicture}[scale = 0.4]	
		\draw[color=black, fill=cyan] (5.000, 8.660) -- (2.500, 4.330) -- (8.500, 4.330) -- (6.000, 8.660) -- cycle;
		\draw[pattern={north east lines}] (2.000, 3.464) -- (8.000, 3.464) -- (9.000, 5.196) -- (7.000, 8.660) -- (5.000, 8.660) -- cycle;
		\draw[thin,dotted] (11.173, 0.300) -- (11.700, 1.212);
\draw[thin] (10.173, 0.300) -- (11.700, 2.944);
\draw[thin,dotted] (9.173, 0.300) -- (11.700, 4.677);
\draw[thin] (8.173, 0.300) -- (11.700, 6.409);
\draw[thin,dotted] (7.173, 0.300) -- (11.700, 8.141);
\draw[thin] (6.173, 0.300) -- (11.700, 9.873);
\draw[thin,dotted] (5.173, 0.300) -- (11.700, 11.605);
\draw[thin] (4.173, 0.300) -- (10.755, 11.700);
\draw[thin,dotted] (3.173, 0.300) -- (9.755, 11.700);
\draw[thin] (2.173, 0.300) -- (8.755, 11.700);
\draw[thin,dotted] (1.173, 0.300) -- (7.755, 11.700);
\draw[thin] (0.300, 0.520) -- (6.755, 11.700);
\draw[thin,dotted] (0.300, 2.252) -- (5.755, 11.700);
\draw[thin] (0.300, 3.984) -- (4.755, 11.700);
\draw[thin,dotted] (0.300, 5.716) -- (3.755, 11.700);
\draw[thin] (0.300, 7.448) -- (2.755, 11.700);
\draw[thin,dotted] (0.300, 9.180) -- (1.755, 11.700);
\draw[thin] (0.300, 10.912) -- (0.755, 11.700);
\draw[thin,dotted] (0.300, 1.212) -- (0.827, 0.300);
\draw[thin] (0.300, 2.944) -- (1.827, 0.300);
\draw[thin,dotted] (0.300, 4.677) -- (2.827, 0.300);
\draw[thin] (0.300, 6.409) -- (3.827, 0.300);
\draw[thin,dotted] (0.300, 8.141) -- (4.827, 0.300);
\draw[thin] (0.300, 9.873) -- (5.827, 0.300);
\draw[thin,dotted] (0.300, 11.605) -- (6.827, 0.300);
\draw[thin] (1.245, 11.700) -- (7.827, 0.300);
\draw[thin,dotted] (2.245, 11.700) -- (8.827, 0.300);
\draw[thin] (3.245, 11.700) -- (9.827, 0.300);
\draw[thin,dotted] (4.245, 11.700) -- (10.827, 0.300);
\draw[thin] (5.245, 11.700) -- (11.700, 0.520);
\draw[thin,dotted] (6.245, 11.700) -- (11.700, 2.252);
\draw[thin] (7.245, 11.700) -- (11.700, 3.984);
\draw[thin,dotted] (8.245, 11.700) -- (11.700, 5.716);
\draw[thin] (9.245, 11.700) -- (11.700, 7.448);
\draw[thin,dotted] (10.245, 11.700) -- (11.700, 9.180);
\draw[thin] (11.245, 11.700) -- (11.700, 10.912);
\draw[thin,dotted] (0.300, 0.866) -- (11.700, 0.866);
\draw[thin] (0.300, 1.732) -- (11.700, 1.732);
\draw[thin,dotted] (0.300, 2.598) -- (11.700, 2.598);
\draw[thin] (0.300, 3.464) -- (11.700, 3.464);
\draw[thin,dotted] (0.300, 4.330) -- (11.700, 4.330);
\draw[thin] (0.300, 5.196) -- (11.700, 5.196);
\draw[thin,dotted] (0.300, 6.062) -- (11.700, 6.062);
\draw[thin] (0.300, 6.928) -- (11.700, 6.928);
\draw[thin,dotted] (0.300, 7.794) -- (11.700, 7.794);
\draw[thin] (0.300, 8.660) -- (11.700, 8.660);
\draw[thin,dotted] (0.300, 9.526) -- (11.700, 9.526);
\draw[thin] (0.300, 10.392) -- (11.700, 10.392);
\draw[thin,dotted] (0.300, 11.258) -- (11.700, 11.258);		
		\draw(6,-0.75) node{\footnotesize $(r,c) = (1,0)$};	
	\end{tikzpicture}
\\
	\begin{tikzpicture}[scale = 0.4]
		\draw[color=black, fill=cyan] (6.000, 8.660) -- (3.500, 4.330) -- (9.500, 4.330) -- (7.000, 8.660) -- cycle;
		\draw[pattern={north east lines}] (4.000, 3.464) -- (10.000, 3.464) -- (7.000, 8.660) -- (5.000, 8.660) -- (3.000, 5.196) -- cycle;
		\draw[thin,dotted] (11.173, 0.300) -- (11.700, 1.212);
\draw[thin] (10.173, 0.300) -- (11.700, 2.944);
\draw[thin,dotted] (9.173, 0.300) -- (11.700, 4.677);
\draw[thin] (8.173, 0.300) -- (11.700, 6.409);
\draw[thin,dotted] (7.173, 0.300) -- (11.700, 8.141);
\draw[thin] (6.173, 0.300) -- (11.700, 9.873);
\draw[thin,dotted] (5.173, 0.300) -- (11.700, 11.605);
\draw[thin] (4.173, 0.300) -- (10.755, 11.700);
\draw[thin,dotted] (3.173, 0.300) -- (9.755, 11.700);
\draw[thin] (2.173, 0.300) -- (8.755, 11.700);
\draw[thin,dotted] (1.173, 0.300) -- (7.755, 11.700);
\draw[thin] (0.300, 0.520) -- (6.755, 11.700);
\draw[thin,dotted] (0.300, 2.252) -- (5.755, 11.700);
\draw[thin] (0.300, 3.984) -- (4.755, 11.700);
\draw[thin,dotted] (0.300, 5.716) -- (3.755, 11.700);
\draw[thin] (0.300, 7.448) -- (2.755, 11.700);
\draw[thin,dotted] (0.300, 9.180) -- (1.755, 11.700);
\draw[thin] (0.300, 10.912) -- (0.755, 11.700);
\draw[thin,dotted] (0.300, 1.212) -- (0.827, 0.300);
\draw[thin] (0.300, 2.944) -- (1.827, 0.300);
\draw[thin,dotted] (0.300, 4.677) -- (2.827, 0.300);
\draw[thin] (0.300, 6.409) -- (3.827, 0.300);
\draw[thin,dotted] (0.300, 8.141) -- (4.827, 0.300);
\draw[thin] (0.300, 9.873) -- (5.827, 0.300);
\draw[thin,dotted] (0.300, 11.605) -- (6.827, 0.300);
\draw[thin] (1.245, 11.700) -- (7.827, 0.300);
\draw[thin,dotted] (2.245, 11.700) -- (8.827, 0.300);
\draw[thin] (3.245, 11.700) -- (9.827, 0.300);
\draw[thin,dotted] (4.245, 11.700) -- (10.827, 0.300);
\draw[thin] (5.245, 11.700) -- (11.700, 0.520);
\draw[thin,dotted] (6.245, 11.700) -- (11.700, 2.252);
\draw[thin] (7.245, 11.700) -- (11.700, 3.984);
\draw[thin,dotted] (8.245, 11.700) -- (11.700, 5.716);
\draw[thin] (9.245, 11.700) -- (11.700, 7.448);
\draw[thin,dotted] (10.245, 11.700) -- (11.700, 9.180);
\draw[thin] (11.245, 11.700) -- (11.700, 10.912);
\draw[thin,dotted] (0.300, 0.866) -- (11.700, 0.866);
\draw[thin] (0.300, 1.732) -- (11.700, 1.732);
\draw[thin,dotted] (0.300, 2.598) -- (11.700, 2.598);
\draw[thin] (0.300, 3.464) -- (11.700, 3.464);
\draw[thin,dotted] (0.300, 4.330) -- (11.700, 4.330);
\draw[thin] (0.300, 5.196) -- (11.700, 5.196);
\draw[thin,dotted] (0.300, 6.062) -- (11.700, 6.062);
\draw[thin] (0.300, 6.928) -- (11.700, 6.928);
\draw[thin,dotted] (0.300, 7.794) -- (11.700, 7.794);
\draw[thin] (0.300, 8.660) -- (11.700, 8.660);
\draw[thin,dotted] (0.300, 9.526) -- (11.700, 9.526);
\draw[thin] (0.300, 10.392) -- (11.700, 10.392);
\draw[thin,dotted] (0.300, 11.258) -- (11.700, 11.258);		
		\draw(6,-0.75) node{\footnotesize $(r,c) = (1,2)$};	
	\end{tikzpicture}
&
	\begin{tikzpicture}[scale = 0.4]
		\draw[color=black, fill=cyan] (5.500, 7.794) -- (3.000, 3.464) -- (9.000, 3.464) -- (6.500, 7.794) -- cycle;
		\draw[pattern={north east lines}] (5.000, 8.660) -- (2.000, 3.464) -- (10.000, 3.464) -- (7.000, 8.660) -- cycle;
		\draw[thin,dotted] (11.173, 0.300) -- (11.700, 1.212);
\draw[thin] (10.173, 0.300) -- (11.700, 2.944);
\draw[thin,dotted] (9.173, 0.300) -- (11.700, 4.677);
\draw[thin] (8.173, 0.300) -- (11.700, 6.409);
\draw[thin,dotted] (7.173, 0.300) -- (11.700, 8.141);
\draw[thin] (6.173, 0.300) -- (11.700, 9.873);
\draw[thin,dotted] (5.173, 0.300) -- (11.700, 11.605);
\draw[thin] (4.173, 0.300) -- (10.755, 11.700);
\draw[thin,dotted] (3.173, 0.300) -- (9.755, 11.700);
\draw[thin] (2.173, 0.300) -- (8.755, 11.700);
\draw[thin,dotted] (1.173, 0.300) -- (7.755, 11.700);
\draw[thin] (0.300, 0.520) -- (6.755, 11.700);
\draw[thin,dotted] (0.300, 2.252) -- (5.755, 11.700);
\draw[thin] (0.300, 3.984) -- (4.755, 11.700);
\draw[thin,dotted] (0.300, 5.716) -- (3.755, 11.700);
\draw[thin] (0.300, 7.448) -- (2.755, 11.700);
\draw[thin,dotted] (0.300, 9.180) -- (1.755, 11.700);
\draw[thin] (0.300, 10.912) -- (0.755, 11.700);
\draw[thin,dotted] (0.300, 1.212) -- (0.827, 0.300);
\draw[thin] (0.300, 2.944) -- (1.827, 0.300);
\draw[thin,dotted] (0.300, 4.677) -- (2.827, 0.300);
\draw[thin] (0.300, 6.409) -- (3.827, 0.300);
\draw[thin,dotted] (0.300, 8.141) -- (4.827, 0.300);
\draw[thin] (0.300, 9.873) -- (5.827, 0.300);
\draw[thin,dotted] (0.300, 11.605) -- (6.827, 0.300);
\draw[thin] (1.245, 11.700) -- (7.827, 0.300);
\draw[thin,dotted] (2.245, 11.700) -- (8.827, 0.300);
\draw[thin] (3.245, 11.700) -- (9.827, 0.300);
\draw[thin,dotted] (4.245, 11.700) -- (10.827, 0.300);
\draw[thin] (5.245, 11.700) -- (11.700, 0.520);
\draw[thin,dotted] (6.245, 11.700) -- (11.700, 2.252);
\draw[thin] (7.245, 11.700) -- (11.700, 3.984);
\draw[thin,dotted] (8.245, 11.700) -- (11.700, 5.716);
\draw[thin] (9.245, 11.700) -- (11.700, 7.448);
\draw[thin,dotted] (10.245, 11.700) -- (11.700, 9.180);
\draw[thin] (11.245, 11.700) -- (11.700, 10.912);
\draw[thin,dotted] (0.300, 0.866) -- (11.700, 0.866);
\draw[thin] (0.300, 1.732) -- (11.700, 1.732);
\draw[thin,dotted] (0.300, 2.598) -- (11.700, 2.598);
\draw[thin] (0.300, 3.464) -- (11.700, 3.464);
\draw[thin,dotted] (0.300, 4.330) -- (11.700, 4.330);
\draw[thin] (0.300, 5.196) -- (11.700, 5.196);
\draw[thin,dotted] (0.300, 6.062) -- (11.700, 6.062);
\draw[thin] (0.300, 6.928) -- (11.700, 6.928);
\draw[thin,dotted] (0.300, 7.794) -- (11.700, 7.794);
\draw[thin] (0.300, 8.660) -- (11.700, 8.660);
\draw[thin,dotted] (0.300, 9.526) -- (11.700, 9.526);
\draw[thin] (0.300, 10.392) -- (11.700, 10.392);
\draw[thin,dotted] (0.300, 11.258) -- (11.700, 11.258);
		\draw(6,-0.75) node{\footnotesize $(r,c) = (2,2)$};	
		\end{tikzpicture}
	\end{tabular}
\end{center}
\caption{The blue shaded regions represent elements of $P_{r,c}(\mytikz{0.07}{\protect\tritop},T, 2n)$, for the different values of $(r,c)$. The illustrations show how these are modified to reach the hatched regions, from which we can deduce the recursive expression in \eqref{eq:recBeven}.}
\label{fig:recBeven}
\end{figure}

\begin{equation}
\label{eq:recBodd}
\begin{split}
B_{2n+1} 
&= -6 + \sum_{(r,c)\in I} |P_{r,c}(\mytikz{0.08}{\tritop},T, 2n+1)| \\
&= -6 + A_{n+1} + B_{n+1} + B_{n+1} + D_{n+2}.
\end{split}
\end{equation}
See Figure~\ref{fig:recBodd} for a visualisation of the deduction of this recursion.

\begin{figure}[ht]
\begin{center}
\begin{tabular}{cc}
	\begin{tikzpicture}[scale = 0.4]
		\draw[color=black, fill=cyan] (5.500, 9.526) -- (2.500, 4.330) -- (9.500, 4.330) -- (6.500, 9.526) -- cycle;
		\draw[pattern={north east lines}] (6.000, 10.392) -- (2.000, 3.464) -- (10.000, 3.464) -- cycle;
		\draw[thin,dotted] (11.173, 0.300) -- (11.700, 1.212);
\draw[thin] (10.173, 0.300) -- (11.700, 2.944);
\draw[thin,dotted] (9.173, 0.300) -- (11.700, 4.677);
\draw[thin] (8.173, 0.300) -- (11.700, 6.409);
\draw[thin,dotted] (7.173, 0.300) -- (11.700, 8.141);
\draw[thin] (6.173, 0.300) -- (11.700, 9.873);
\draw[thin,dotted] (5.173, 0.300) -- (11.700, 11.605);
\draw[thin] (4.173, 0.300) -- (10.755, 11.700);
\draw[thin,dotted] (3.173, 0.300) -- (9.755, 11.700);
\draw[thin] (2.173, 0.300) -- (8.755, 11.700);
\draw[thin,dotted] (1.173, 0.300) -- (7.755, 11.700);
\draw[thin] (0.300, 0.520) -- (6.755, 11.700);
\draw[thin,dotted] (0.300, 2.252) -- (5.755, 11.700);
\draw[thin] (0.300, 3.984) -- (4.755, 11.700);
\draw[thin,dotted] (0.300, 5.716) -- (3.755, 11.700);
\draw[thin] (0.300, 7.448) -- (2.755, 11.700);
\draw[thin,dotted] (0.300, 9.180) -- (1.755, 11.700);
\draw[thin] (0.300, 10.912) -- (0.755, 11.700);
\draw[thin,dotted] (0.300, 1.212) -- (0.827, 0.300);
\draw[thin] (0.300, 2.944) -- (1.827, 0.300);
\draw[thin,dotted] (0.300, 4.677) -- (2.827, 0.300);
\draw[thin] (0.300, 6.409) -- (3.827, 0.300);
\draw[thin,dotted] (0.300, 8.141) -- (4.827, 0.300);
\draw[thin] (0.300, 9.873) -- (5.827, 0.300);
\draw[thin,dotted] (0.300, 11.605) -- (6.827, 0.300);
\draw[thin] (1.245, 11.700) -- (7.827, 0.300);
\draw[thin,dotted] (2.245, 11.700) -- (8.827, 0.300);
\draw[thin] (3.245, 11.700) -- (9.827, 0.300);
\draw[thin,dotted] (4.245, 11.700) -- (10.827, 0.300);
\draw[thin] (5.245, 11.700) -- (11.700, 0.520);
\draw[thin,dotted] (6.245, 11.700) -- (11.700, 2.252);
\draw[thin] (7.245, 11.700) -- (11.700, 3.984);
\draw[thin,dotted] (8.245, 11.700) -- (11.700, 5.716);
\draw[thin] (9.245, 11.700) -- (11.700, 7.448);
\draw[thin,dotted] (10.245, 11.700) -- (11.700, 9.180);
\draw[thin] (11.245, 11.700) -- (11.700, 10.912);
\draw[thin,dotted] (0.300, 0.866) -- (11.700, 0.866);
\draw[thin] (0.300, 1.732) -- (11.700, 1.732);
\draw[thin,dotted] (0.300, 2.598) -- (11.700, 2.598);
\draw[thin] (0.300, 3.464) -- (11.700, 3.464);
\draw[thin,dotted] (0.300, 4.330) -- (11.700, 4.330);
\draw[thin] (0.300, 5.196) -- (11.700, 5.196);
\draw[thin,dotted] (0.300, 6.062) -- (11.700, 6.062);
\draw[thin] (0.300, 6.928) -- (11.700, 6.928);
\draw[thin,dotted] (0.300, 7.794) -- (11.700, 7.794);
\draw[thin] (0.300, 8.660) -- (11.700, 8.660);
\draw[thin,dotted] (0.300, 9.526) -- (11.700, 9.526);
\draw[thin] (0.300, 10.392) -- (11.700, 10.392);
\draw[thin,dotted] (0.300, 11.258) -- (11.700, 11.258);		
		\draw(6,-0.75) node{\footnotesize $(r,c) = (0,0)$};	
	\end{tikzpicture}
&
	\begin{tikzpicture}[scale = 0.4]
		\draw[color=black, fill=cyan] (5.000, 8.660) -- (2.000, 3.464) -- (9.000, 3.464) -- (6.000, 8.660) -- cycle;
		\draw[pattern={north east lines}] (5.000, 8.660) -- (2.000, 3.464) -- (10.000, 3.464) -- (7.000, 8.660) -- cycle;
		\draw[thin,dotted] (11.173, 0.300) -- (11.700, 1.212);
\draw[thin] (10.173, 0.300) -- (11.700, 2.944);
\draw[thin,dotted] (9.173, 0.300) -- (11.700, 4.677);
\draw[thin] (8.173, 0.300) -- (11.700, 6.409);
\draw[thin,dotted] (7.173, 0.300) -- (11.700, 8.141);
\draw[thin] (6.173, 0.300) -- (11.700, 9.873);
\draw[thin,dotted] (5.173, 0.300) -- (11.700, 11.605);
\draw[thin] (4.173, 0.300) -- (10.755, 11.700);
\draw[thin,dotted] (3.173, 0.300) -- (9.755, 11.700);
\draw[thin] (2.173, 0.300) -- (8.755, 11.700);
\draw[thin,dotted] (1.173, 0.300) -- (7.755, 11.700);
\draw[thin] (0.300, 0.520) -- (6.755, 11.700);
\draw[thin,dotted] (0.300, 2.252) -- (5.755, 11.700);
\draw[thin] (0.300, 3.984) -- (4.755, 11.700);
\draw[thin,dotted] (0.300, 5.716) -- (3.755, 11.700);
\draw[thin] (0.300, 7.448) -- (2.755, 11.700);
\draw[thin,dotted] (0.300, 9.180) -- (1.755, 11.700);
\draw[thin] (0.300, 10.912) -- (0.755, 11.700);
\draw[thin,dotted] (0.300, 1.212) -- (0.827, 0.300);
\draw[thin] (0.300, 2.944) -- (1.827, 0.300);
\draw[thin,dotted] (0.300, 4.677) -- (2.827, 0.300);
\draw[thin] (0.300, 6.409) -- (3.827, 0.300);
\draw[thin,dotted] (0.300, 8.141) -- (4.827, 0.300);
\draw[thin] (0.300, 9.873) -- (5.827, 0.300);
\draw[thin,dotted] (0.300, 11.605) -- (6.827, 0.300);
\draw[thin] (1.245, 11.700) -- (7.827, 0.300);
\draw[thin,dotted] (2.245, 11.700) -- (8.827, 0.300);
\draw[thin] (3.245, 11.700) -- (9.827, 0.300);
\draw[thin,dotted] (4.245, 11.700) -- (10.827, 0.300);
\draw[thin] (5.245, 11.700) -- (11.700, 0.520);
\draw[thin,dotted] (6.245, 11.700) -- (11.700, 2.252);
\draw[thin] (7.245, 11.700) -- (11.700, 3.984);
\draw[thin,dotted] (8.245, 11.700) -- (11.700, 5.716);
\draw[thin] (9.245, 11.700) -- (11.700, 7.448);
\draw[thin,dotted] (10.245, 11.700) -- (11.700, 9.180);
\draw[thin] (11.245, 11.700) -- (11.700, 10.912);
\draw[thin,dotted] (0.300, 0.866) -- (11.700, 0.866);
\draw[thin] (0.300, 1.732) -- (11.700, 1.732);
\draw[thin,dotted] (0.300, 2.598) -- (11.700, 2.598);
\draw[thin] (0.300, 3.464) -- (11.700, 3.464);
\draw[thin,dotted] (0.300, 4.330) -- (11.700, 4.330);
\draw[thin] (0.300, 5.196) -- (11.700, 5.196);
\draw[thin,dotted] (0.300, 6.062) -- (11.700, 6.062);
\draw[thin] (0.300, 6.928) -- (11.700, 6.928);
\draw[thin,dotted] (0.300, 7.794) -- (11.700, 7.794);
\draw[thin] (0.300, 8.660) -- (11.700, 8.660);
\draw[thin,dotted] (0.300, 9.526) -- (11.700, 9.526);
\draw[thin] (0.300, 10.392) -- (11.700, 10.392);
\draw[thin,dotted] (0.300, 11.258) -- (11.700, 11.258);		
		\draw(6,-0.75) node{\footnotesize $(r,c) = (1,0)$};	
	\end{tikzpicture}
\\
	\begin{tikzpicture}[scale = 0.4]
		\draw[color=black, fill=cyan] (6.000, 8.660) -- (3.000, 3.464) -- (10.000, 3.464) -- (7.000, 8.660) -- cycle;
		\draw[pattern={north east lines}] (5.000, 8.660) -- (2.000, 3.464) -- (10.000, 3.464) -- (7.000, 8.660) -- cycle;
		\draw[thin,dotted] (11.173, 0.300) -- (11.700, 1.212);
\draw[thin] (10.173, 0.300) -- (11.700, 2.944);
\draw[thin,dotted] (9.173, 0.300) -- (11.700, 4.677);
\draw[thin] (8.173, 0.300) -- (11.700, 6.409);
\draw[thin,dotted] (7.173, 0.300) -- (11.700, 8.141);
\draw[thin] (6.173, 0.300) -- (11.700, 9.873);
\draw[thin,dotted] (5.173, 0.300) -- (11.700, 11.605);
\draw[thin] (4.173, 0.300) -- (10.755, 11.700);
\draw[thin,dotted] (3.173, 0.300) -- (9.755, 11.700);
\draw[thin] (2.173, 0.300) -- (8.755, 11.700);
\draw[thin,dotted] (1.173, 0.300) -- (7.755, 11.700);
\draw[thin] (0.300, 0.520) -- (6.755, 11.700);
\draw[thin,dotted] (0.300, 2.252) -- (5.755, 11.700);
\draw[thin] (0.300, 3.984) -- (4.755, 11.700);
\draw[thin,dotted] (0.300, 5.716) -- (3.755, 11.700);
\draw[thin] (0.300, 7.448) -- (2.755, 11.700);
\draw[thin,dotted] (0.300, 9.180) -- (1.755, 11.700);
\draw[thin] (0.300, 10.912) -- (0.755, 11.700);
\draw[thin,dotted] (0.300, 1.212) -- (0.827, 0.300);
\draw[thin] (0.300, 2.944) -- (1.827, 0.300);
\draw[thin,dotted] (0.300, 4.677) -- (2.827, 0.300);
\draw[thin] (0.300, 6.409) -- (3.827, 0.300);
\draw[thin,dotted] (0.300, 8.141) -- (4.827, 0.300);
\draw[thin] (0.300, 9.873) -- (5.827, 0.300);
\draw[thin,dotted] (0.300, 11.605) -- (6.827, 0.300);
\draw[thin] (1.245, 11.700) -- (7.827, 0.300);
\draw[thin,dotted] (2.245, 11.700) -- (8.827, 0.300);
\draw[thin] (3.245, 11.700) -- (9.827, 0.300);
\draw[thin,dotted] (4.245, 11.700) -- (10.827, 0.300);
\draw[thin] (5.245, 11.700) -- (11.700, 0.520);
\draw[thin,dotted] (6.245, 11.700) -- (11.700, 2.252);
\draw[thin] (7.245, 11.700) -- (11.700, 3.984);
\draw[thin,dotted] (8.245, 11.700) -- (11.700, 5.716);
\draw[thin] (9.245, 11.700) -- (11.700, 7.448);
\draw[thin,dotted] (10.245, 11.700) -- (11.700, 9.180);
\draw[thin] (11.245, 11.700) -- (11.700, 10.912);
\draw[thin,dotted] (0.300, 0.866) -- (11.700, 0.866);
\draw[thin] (0.300, 1.732) -- (11.700, 1.732);
\draw[thin,dotted] (0.300, 2.598) -- (11.700, 2.598);
\draw[thin] (0.300, 3.464) -- (11.700, 3.464);
\draw[thin,dotted] (0.300, 4.330) -- (11.700, 4.330);
\draw[thin] (0.300, 5.196) -- (11.700, 5.196);
\draw[thin,dotted] (0.300, 6.062) -- (11.700, 6.062);
\draw[thin] (0.300, 6.928) -- (11.700, 6.928);
\draw[thin,dotted] (0.300, 7.794) -- (11.700, 7.794);
\draw[thin] (0.300, 8.660) -- (11.700, 8.660);
\draw[thin,dotted] (0.300, 9.526) -- (11.700, 9.526);
\draw[thin] (0.300, 10.392) -- (11.700, 10.392);
\draw[thin,dotted] (0.300, 11.258) -- (11.700, 11.258);		
		\draw(6,-0.75) node{\footnotesize $(r,c) = (1,2)$};	
	\end{tikzpicture}
&
	\begin{tikzpicture}[scale = 0.4]
		\draw[color=black, fill=cyan] (5.500, 7.794) -- (2.500, 2.598) -- (9.500, 2.598) -- (6.500, 7.794) -- cycle;
		\draw[pattern={north east lines}] (5.000, 8.660) -- (2.000, 3.464) -- (3.000, 1.732) -- (9.000, 1.732) -- (10.000, 3.464) -- (7.000, 8.660) -- cycle;
		\draw[thin,dotted] (11.173, 0.300) -- (11.700, 1.212);
\draw[thin] (10.173, 0.300) -- (11.700, 2.944);
\draw[thin,dotted] (9.173, 0.300) -- (11.700, 4.677);
\draw[thin] (8.173, 0.300) -- (11.700, 6.409);
\draw[thin,dotted] (7.173, 0.300) -- (11.700, 8.141);
\draw[thin] (6.173, 0.300) -- (11.700, 9.873);
\draw[thin,dotted] (5.173, 0.300) -- (11.700, 11.605);
\draw[thin] (4.173, 0.300) -- (10.755, 11.700);
\draw[thin,dotted] (3.173, 0.300) -- (9.755, 11.700);
\draw[thin] (2.173, 0.300) -- (8.755, 11.700);
\draw[thin,dotted] (1.173, 0.300) -- (7.755, 11.700);
\draw[thin] (0.300, 0.520) -- (6.755, 11.700);
\draw[thin,dotted] (0.300, 2.252) -- (5.755, 11.700);
\draw[thin] (0.300, 3.984) -- (4.755, 11.700);
\draw[thin,dotted] (0.300, 5.716) -- (3.755, 11.700);
\draw[thin] (0.300, 7.448) -- (2.755, 11.700);
\draw[thin,dotted] (0.300, 9.180) -- (1.755, 11.700);
\draw[thin] (0.300, 10.912) -- (0.755, 11.700);
\draw[thin,dotted] (0.300, 1.212) -- (0.827, 0.300);
\draw[thin] (0.300, 2.944) -- (1.827, 0.300);
\draw[thin,dotted] (0.300, 4.677) -- (2.827, 0.300);
\draw[thin] (0.300, 6.409) -- (3.827, 0.300);
\draw[thin,dotted] (0.300, 8.141) -- (4.827, 0.300);
\draw[thin] (0.300, 9.873) -- (5.827, 0.300);
\draw[thin,dotted] (0.300, 11.605) -- (6.827, 0.300);
\draw[thin] (1.245, 11.700) -- (7.827, 0.300);
\draw[thin,dotted] (2.245, 11.700) -- (8.827, 0.300);
\draw[thin] (3.245, 11.700) -- (9.827, 0.300);
\draw[thin,dotted] (4.245, 11.700) -- (10.827, 0.300);
\draw[thin] (5.245, 11.700) -- (11.700, 0.520);
\draw[thin,dotted] (6.245, 11.700) -- (11.700, 2.252);
\draw[thin] (7.245, 11.700) -- (11.700, 3.984);
\draw[thin,dotted] (8.245, 11.700) -- (11.700, 5.716);
\draw[thin] (9.245, 11.700) -- (11.700, 7.448);
\draw[thin,dotted] (10.245, 11.700) -- (11.700, 9.180);
\draw[thin] (11.245, 11.700) -- (11.700, 10.912);
\draw[thin,dotted] (0.300, 0.866) -- (11.700, 0.866);
\draw[thin] (0.300, 1.732) -- (11.700, 1.732);
\draw[thin,dotted] (0.300, 2.598) -- (11.700, 2.598);
\draw[thin] (0.300, 3.464) -- (11.700, 3.464);
\draw[thin,dotted] (0.300, 4.330) -- (11.700, 4.330);
\draw[thin] (0.300, 5.196) -- (11.700, 5.196);
\draw[thin,dotted] (0.300, 6.062) -- (11.700, 6.062);
\draw[thin] (0.300, 6.928) -- (11.700, 6.928);
\draw[thin,dotted] (0.300, 7.794) -- (11.700, 7.794);
\draw[thin] (0.300, 8.660) -- (11.700, 8.660);
\draw[thin,dotted] (0.300, 9.526) -- (11.700, 9.526);
\draw[thin] (0.300, 10.392) -- (11.700, 10.392);
\draw[thin,dotted] (0.300, 11.258) -- (11.700, 11.258);
		\draw(6,-0.75) node{\footnotesize $(r,c) = (2,2)$};	
	\end{tikzpicture}
\end{tabular}
\end{center}
\caption{The blue shaded shaded regions represent elements of $P_{r,c}(\mytikz{0.07}{\protect\tritop},T, 2n+1)$, for the different values of $(r,c)$. The illustrations show how these are modified to reach the hatched regions, from which we can deduce the recursive expression in \eqref{eq:recBodd}.}
\label{fig:recBodd}
\end{figure}

\begin{equation}
\label{eq:recCeven}
\begin{split}
C_{2n} 
&= -6 + \sum_{(r,c)\in I} |P_{r,c}(\mytikz{0.08}{\trilower},T, 2n)| \\
&= -6 + A_{n} + C_{n+1} + C_{n+1} + D_{n+1}.
\end{split}
\end{equation}
See Figure~\ref{fig:recCeven} for a visualisation of the deduction of this recursion. 

\begin{figure}[ht]
\begin{center}
\begin{tabular}{cc}
	\begin{tikzpicture}[scale = 0.4]
		\draw[color=black, fill=cyan] (6.000, 10.392) -- (3.500, 6.062) -- (4.000, 5.196) -- (8.000, 5.196) -- (8.500, 6.062) -- cycle;
		\draw[pattern={north east lines}] (6.000, 10.392) -- (3.000, 5.196) -- (9.000, 5.196) -- cycle;
		\draw[thin,dotted] (11.173, 0.300) -- (11.700, 1.212);
\draw[thin] (10.173, 0.300) -- (11.700, 2.944);
\draw[thin,dotted] (9.173, 0.300) -- (11.700, 4.677);
\draw[thin] (8.173, 0.300) -- (11.700, 6.409);
\draw[thin,dotted] (7.173, 0.300) -- (11.700, 8.141);
\draw[thin] (6.173, 0.300) -- (11.700, 9.873);
\draw[thin,dotted] (5.173, 0.300) -- (11.700, 11.605);
\draw[thin] (4.173, 0.300) -- (10.755, 11.700);
\draw[thin,dotted] (3.173, 0.300) -- (9.755, 11.700);
\draw[thin] (2.173, 0.300) -- (8.755, 11.700);
\draw[thin,dotted] (1.173, 0.300) -- (7.755, 11.700);
\draw[thin] (0.300, 0.520) -- (6.755, 11.700);
\draw[thin,dotted] (0.300, 2.252) -- (5.755, 11.700);
\draw[thin] (0.300, 3.984) -- (4.755, 11.700);
\draw[thin,dotted] (0.300, 5.716) -- (3.755, 11.700);
\draw[thin] (0.300, 7.448) -- (2.755, 11.700);
\draw[thin,dotted] (0.300, 9.180) -- (1.755, 11.700);
\draw[thin] (0.300, 10.912) -- (0.755, 11.700);
\draw[thin,dotted] (0.300, 1.212) -- (0.827, 0.300);
\draw[thin] (0.300, 2.944) -- (1.827, 0.300);
\draw[thin,dotted] (0.300, 4.677) -- (2.827, 0.300);
\draw[thin] (0.300, 6.409) -- (3.827, 0.300);
\draw[thin,dotted] (0.300, 8.141) -- (4.827, 0.300);
\draw[thin] (0.300, 9.873) -- (5.827, 0.300);
\draw[thin,dotted] (0.300, 11.605) -- (6.827, 0.300);
\draw[thin] (1.245, 11.700) -- (7.827, 0.300);
\draw[thin,dotted] (2.245, 11.700) -- (8.827, 0.300);
\draw[thin] (3.245, 11.700) -- (9.827, 0.300);
\draw[thin,dotted] (4.245, 11.700) -- (10.827, 0.300);
\draw[thin] (5.245, 11.700) -- (11.700, 0.520);
\draw[thin,dotted] (6.245, 11.700) -- (11.700, 2.252);
\draw[thin] (7.245, 11.700) -- (11.700, 3.984);
\draw[thin,dotted] (8.245, 11.700) -- (11.700, 5.716);
\draw[thin] (9.245, 11.700) -- (11.700, 7.448);
\draw[thin,dotted] (10.245, 11.700) -- (11.700, 9.180);
\draw[thin] (11.245, 11.700) -- (11.700, 10.912);
\draw[thin,dotted] (0.300, 0.866) -- (11.700, 0.866);
\draw[thin] (0.300, 1.732) -- (11.700, 1.732);
\draw[thin,dotted] (0.300, 2.598) -- (11.700, 2.598);
\draw[thin] (0.300, 3.464) -- (11.700, 3.464);
\draw[thin,dotted] (0.300, 4.330) -- (11.700, 4.330);
\draw[thin] (0.300, 5.196) -- (11.700, 5.196);
\draw[thin,dotted] (0.300, 6.062) -- (11.700, 6.062);
\draw[thin] (0.300, 6.928) -- (11.700, 6.928);
\draw[thin,dotted] (0.300, 7.794) -- (11.700, 7.794);
\draw[thin] (0.300, 8.660) -- (11.700, 8.660);
\draw[thin,dotted] (0.300, 9.526) -- (11.700, 9.526);
\draw[thin] (0.300, 10.392) -- (11.700, 10.392);
\draw[thin,dotted] (0.300, 11.258) -- (11.700, 11.258);		
		\draw(6,-0.75) node{\footnotesize $(r,c) = (0,0)$};	
	\end{tikzpicture}
&
	\begin{tikzpicture}[scale = 0.4]
		\draw[color=black, fill=cyan] (5.500, 9.526) -- (3.000, 5.196) -- (3.500, 4.330) -- (7.500, 4.330) -- (8.000, 5.196) -- cycle;
		\draw[pattern={north east lines}] (6.000, 10.392) -- (3.000, 5.196) -- (4.000, 3.464) -- (8.000, 3.464) -- (9.000, 5.196) -- cycle;
		\draw[thin,dotted] (11.173, 0.300) -- (11.700, 1.212);
\draw[thin] (10.173, 0.300) -- (11.700, 2.944);
\draw[thin,dotted] (9.173, 0.300) -- (11.700, 4.677);
\draw[thin] (8.173, 0.300) -- (11.700, 6.409);
\draw[thin,dotted] (7.173, 0.300) -- (11.700, 8.141);
\draw[thin] (6.173, 0.300) -- (11.700, 9.873);
\draw[thin,dotted] (5.173, 0.300) -- (11.700, 11.605);
\draw[thin] (4.173, 0.300) -- (10.755, 11.700);
\draw[thin,dotted] (3.173, 0.300) -- (9.755, 11.700);
\draw[thin] (2.173, 0.300) -- (8.755, 11.700);
\draw[thin,dotted] (1.173, 0.300) -- (7.755, 11.700);
\draw[thin] (0.300, 0.520) -- (6.755, 11.700);
\draw[thin,dotted] (0.300, 2.252) -- (5.755, 11.700);
\draw[thin] (0.300, 3.984) -- (4.755, 11.700);
\draw[thin,dotted] (0.300, 5.716) -- (3.755, 11.700);
\draw[thin] (0.300, 7.448) -- (2.755, 11.700);
\draw[thin,dotted] (0.300, 9.180) -- (1.755, 11.700);
\draw[thin] (0.300, 10.912) -- (0.755, 11.700);
\draw[thin,dotted] (0.300, 1.212) -- (0.827, 0.300);
\draw[thin] (0.300, 2.944) -- (1.827, 0.300);
\draw[thin,dotted] (0.300, 4.677) -- (2.827, 0.300);
\draw[thin] (0.300, 6.409) -- (3.827, 0.300);
\draw[thin,dotted] (0.300, 8.141) -- (4.827, 0.300);
\draw[thin] (0.300, 9.873) -- (5.827, 0.300);
\draw[thin,dotted] (0.300, 11.605) -- (6.827, 0.300);
\draw[thin] (1.245, 11.700) -- (7.827, 0.300);
\draw[thin,dotted] (2.245, 11.700) -- (8.827, 0.300);
\draw[thin] (3.245, 11.700) -- (9.827, 0.300);
\draw[thin,dotted] (4.245, 11.700) -- (10.827, 0.300);
\draw[thin] (5.245, 11.700) -- (11.700, 0.520);
\draw[thin,dotted] (6.245, 11.700) -- (11.700, 2.252);
\draw[thin] (7.245, 11.700) -- (11.700, 3.984);
\draw[thin,dotted] (8.245, 11.700) -- (11.700, 5.716);
\draw[thin] (9.245, 11.700) -- (11.700, 7.448);
\draw[thin,dotted] (10.245, 11.700) -- (11.700, 9.180);
\draw[thin] (11.245, 11.700) -- (11.700, 10.912);
\draw[thin,dotted] (0.300, 0.866) -- (11.700, 0.866);
\draw[thin] (0.300, 1.732) -- (11.700, 1.732);
\draw[thin,dotted] (0.300, 2.598) -- (11.700, 2.598);
\draw[thin] (0.300, 3.464) -- (11.700, 3.464);
\draw[thin,dotted] (0.300, 4.330) -- (11.700, 4.330);
\draw[thin] (0.300, 5.196) -- (11.700, 5.196);
\draw[thin,dotted] (0.300, 6.062) -- (11.700, 6.062);
\draw[thin] (0.300, 6.928) -- (11.700, 6.928);
\draw[thin,dotted] (0.300, 7.794) -- (11.700, 7.794);
\draw[thin] (0.300, 8.660) -- (11.700, 8.660);
\draw[thin,dotted] (0.300, 9.526) -- (11.700, 9.526);
\draw[thin] (0.300, 10.392) -- (11.700, 10.392);
\draw[thin,dotted] (0.300, 11.258) -- (11.700, 11.258);		
		\draw(6,-0.75) node{\footnotesize $(r,c) = (1,0)$};	
	\end{tikzpicture}
\\
	\begin{tikzpicture}[scale = 0.4]
		\draw[color=black, fill=cyan] (6.500, 9.526) -- (4.000, 5.196) -- (4.500, 4.330) -- (8.500, 4.330) -- (9.000, 5.196) -- cycle;
		\draw[pattern={north east lines}] (6.000, 10.392) -- (3.000, 5.196) -- (4.000, 3.464) -- (8.000, 3.464) -- (9.000, 5.196) -- cycle;
		\draw[thin,dotted] (11.173, 0.300) -- (11.700, 1.212);
\draw[thin] (10.173, 0.300) -- (11.700, 2.944);
\draw[thin,dotted] (9.173, 0.300) -- (11.700, 4.677);
\draw[thin] (8.173, 0.300) -- (11.700, 6.409);
\draw[thin,dotted] (7.173, 0.300) -- (11.700, 8.141);
\draw[thin] (6.173, 0.300) -- (11.700, 9.873);
\draw[thin,dotted] (5.173, 0.300) -- (11.700, 11.605);
\draw[thin] (4.173, 0.300) -- (10.755, 11.700);
\draw[thin,dotted] (3.173, 0.300) -- (9.755, 11.700);
\draw[thin] (2.173, 0.300) -- (8.755, 11.700);
\draw[thin,dotted] (1.173, 0.300) -- (7.755, 11.700);
\draw[thin] (0.300, 0.520) -- (6.755, 11.700);
\draw[thin,dotted] (0.300, 2.252) -- (5.755, 11.700);
\draw[thin] (0.300, 3.984) -- (4.755, 11.700);
\draw[thin,dotted] (0.300, 5.716) -- (3.755, 11.700);
\draw[thin] (0.300, 7.448) -- (2.755, 11.700);
\draw[thin,dotted] (0.300, 9.180) -- (1.755, 11.700);
\draw[thin] (0.300, 10.912) -- (0.755, 11.700);
\draw[thin,dotted] (0.300, 1.212) -- (0.827, 0.300);
\draw[thin] (0.300, 2.944) -- (1.827, 0.300);
\draw[thin,dotted] (0.300, 4.677) -- (2.827, 0.300);
\draw[thin] (0.300, 6.409) -- (3.827, 0.300);
\draw[thin,dotted] (0.300, 8.141) -- (4.827, 0.300);
\draw[thin] (0.300, 9.873) -- (5.827, 0.300);
\draw[thin,dotted] (0.300, 11.605) -- (6.827, 0.300);
\draw[thin] (1.245, 11.700) -- (7.827, 0.300);
\draw[thin,dotted] (2.245, 11.700) -- (8.827, 0.300);
\draw[thin] (3.245, 11.700) -- (9.827, 0.300);
\draw[thin,dotted] (4.245, 11.700) -- (10.827, 0.300);
\draw[thin] (5.245, 11.700) -- (11.700, 0.520);
\draw[thin,dotted] (6.245, 11.700) -- (11.700, 2.252);
\draw[thin] (7.245, 11.700) -- (11.700, 3.984);
\draw[thin,dotted] (8.245, 11.700) -- (11.700, 5.716);
\draw[thin] (9.245, 11.700) -- (11.700, 7.448);
\draw[thin,dotted] (10.245, 11.700) -- (11.700, 9.180);
\draw[thin] (11.245, 11.700) -- (11.700, 10.912);
\draw[thin,dotted] (0.300, 0.866) -- (11.700, 0.866);
\draw[thin] (0.300, 1.732) -- (11.700, 1.732);
\draw[thin,dotted] (0.300, 2.598) -- (11.700, 2.598);
\draw[thin] (0.300, 3.464) -- (11.700, 3.464);
\draw[thin,dotted] (0.300, 4.330) -- (11.700, 4.330);
\draw[thin] (0.300, 5.196) -- (11.700, 5.196);
\draw[thin,dotted] (0.300, 6.062) -- (11.700, 6.062);
\draw[thin] (0.300, 6.928) -- (11.700, 6.928);
\draw[thin,dotted] (0.300, 7.794) -- (11.700, 7.794);
\draw[thin] (0.300, 8.660) -- (11.700, 8.660);
\draw[thin,dotted] (0.300, 9.526) -- (11.700, 9.526);
\draw[thin] (0.300, 10.392) -- (11.700, 10.392);
\draw[thin,dotted] (0.300, 11.258) -- (11.700, 11.258);		
		\draw(6,-0.75) node{\footnotesize $(r,c) = (1,2)$};	
	\end{tikzpicture}
&
	\begin{tikzpicture}[scale = 0.4]
		\draw[color=black, fill=cyan] (6.000, 8.660) -- (3.500, 4.330) -- (4.000, 3.464) -- (8.000, 3.464) -- (8.500, 4.330) -- cycle;
		\draw[pattern={north east lines}] (5.000, 8.660) -- (3.000, 5.196) -- (4.000, 3.464) -- (8.000, 3.464) -- (9.000, 5.196) -- (7.000, 8.660) -- cycle;
		\draw[thin,dotted] (11.173, 0.300) -- (11.700, 1.212);
\draw[thin] (10.173, 0.300) -- (11.700, 2.944);
\draw[thin,dotted] (9.173, 0.300) -- (11.700, 4.677);
\draw[thin] (8.173, 0.300) -- (11.700, 6.409);
\draw[thin,dotted] (7.173, 0.300) -- (11.700, 8.141);
\draw[thin] (6.173, 0.300) -- (11.700, 9.873);
\draw[thin,dotted] (5.173, 0.300) -- (11.700, 11.605);
\draw[thin] (4.173, 0.300) -- (10.755, 11.700);
\draw[thin,dotted] (3.173, 0.300) -- (9.755, 11.700);
\draw[thin] (2.173, 0.300) -- (8.755, 11.700);
\draw[thin,dotted] (1.173, 0.300) -- (7.755, 11.700);
\draw[thin] (0.300, 0.520) -- (6.755, 11.700);
\draw[thin,dotted] (0.300, 2.252) -- (5.755, 11.700);
\draw[thin] (0.300, 3.984) -- (4.755, 11.700);
\draw[thin,dotted] (0.300, 5.716) -- (3.755, 11.700);
\draw[thin] (0.300, 7.448) -- (2.755, 11.700);
\draw[thin,dotted] (0.300, 9.180) -- (1.755, 11.700);
\draw[thin] (0.300, 10.912) -- (0.755, 11.700);
\draw[thin,dotted] (0.300, 1.212) -- (0.827, 0.300);
\draw[thin] (0.300, 2.944) -- (1.827, 0.300);
\draw[thin,dotted] (0.300, 4.677) -- (2.827, 0.300);
\draw[thin] (0.300, 6.409) -- (3.827, 0.300);
\draw[thin,dotted] (0.300, 8.141) -- (4.827, 0.300);
\draw[thin] (0.300, 9.873) -- (5.827, 0.300);
\draw[thin,dotted] (0.300, 11.605) -- (6.827, 0.300);
\draw[thin] (1.245, 11.700) -- (7.827, 0.300);
\draw[thin,dotted] (2.245, 11.700) -- (8.827, 0.300);
\draw[thin] (3.245, 11.700) -- (9.827, 0.300);
\draw[thin,dotted] (4.245, 11.700) -- (10.827, 0.300);
\draw[thin] (5.245, 11.700) -- (11.700, 0.520);
\draw[thin,dotted] (6.245, 11.700) -- (11.700, 2.252);
\draw[thin] (7.245, 11.700) -- (11.700, 3.984);
\draw[thin,dotted] (8.245, 11.700) -- (11.700, 5.716);
\draw[thin] (9.245, 11.700) -- (11.700, 7.448);
\draw[thin,dotted] (10.245, 11.700) -- (11.700, 9.180);
\draw[thin] (11.245, 11.700) -- (11.700, 10.912);
\draw[thin,dotted] (0.300, 0.866) -- (11.700, 0.866);
\draw[thin] (0.300, 1.732) -- (11.700, 1.732);
\draw[thin,dotted] (0.300, 2.598) -- (11.700, 2.598);
\draw[thin] (0.300, 3.464) -- (11.700, 3.464);
\draw[thin,dotted] (0.300, 4.330) -- (11.700, 4.330);
\draw[thin] (0.300, 5.196) -- (11.700, 5.196);
\draw[thin,dotted] (0.300, 6.062) -- (11.700, 6.062);
\draw[thin] (0.300, 6.928) -- (11.700, 6.928);
\draw[thin,dotted] (0.300, 7.794) -- (11.700, 7.794);
\draw[thin] (0.300, 8.660) -- (11.700, 8.660);
\draw[thin,dotted] (0.300, 9.526) -- (11.700, 9.526);
\draw[thin] (0.300, 10.392) -- (11.700, 10.392);
\draw[thin,dotted] (0.300, 11.258) -- (11.700, 11.258);
		\draw(6,-0.75) node{\footnotesize $(r,c) = (2,2)$};	
	\end{tikzpicture}
\end{tabular}
\end{center}
\caption{The blue shaded regions represent elements of $P_{r,c}(\mytikz{0.07}{\protect\trilower},T, 2n)$, for the different values of $(r,c)$. The illustrations show how these are modified to reach the hatched regions, from which we can deduce the recursive expression in \eqref{eq:recCeven}.}
\label{fig:recCeven}
\end{figure}

\begin{equation}
\label{eq:recCodd}
\begin{split}
C_{2n+1} 
&= -6 + \sum_{(r,c)\in I} |P_{r,c}(\mytikz{0.08}{\trilower},T, 2n+1)| \\
&= -6 + C_{n+1} + B_{n+1} + B_{n+1} + D_{n+2}.
\end{split}
\end{equation}
See Figure~\ref{fig:recCodd} for a visualisation of the deduction of this recursion.

\begin{figure}[ht]
\begin{center}
\begin{tabular}{cc}
	\begin{tikzpicture}[scale = 0.4]
		\draw[color=black, fill=cyan] (6.000, 10.392) -- (3.000, 5.196) -- (3.500, 4.330) -- (8.500, 4.330) -- (9.000, 5.196) -- cycle;
		\draw[pattern={north east lines}] (6.000, 10.392) -- (3.000, 5.196) -- (4.000, 3.464) -- (8.000, 3.464) -- (9.000, 5.196) -- cycle;
		\draw[thin,dotted] (11.173, 0.300) -- (11.700, 1.212);
\draw[thin] (10.173, 0.300) -- (11.700, 2.944);
\draw[thin,dotted] (9.173, 0.300) -- (11.700, 4.677);
\draw[thin] (8.173, 0.300) -- (11.700, 6.409);
\draw[thin,dotted] (7.173, 0.300) -- (11.700, 8.141);
\draw[thin] (6.173, 0.300) -- (11.700, 9.873);
\draw[thin,dotted] (5.173, 0.300) -- (11.700, 11.605);
\draw[thin] (4.173, 0.300) -- (10.755, 11.700);
\draw[thin,dotted] (3.173, 0.300) -- (9.755, 11.700);
\draw[thin] (2.173, 0.300) -- (8.755, 11.700);
\draw[thin,dotted] (1.173, 0.300) -- (7.755, 11.700);
\draw[thin] (0.300, 0.520) -- (6.755, 11.700);
\draw[thin,dotted] (0.300, 2.252) -- (5.755, 11.700);
\draw[thin] (0.300, 3.984) -- (4.755, 11.700);
\draw[thin,dotted] (0.300, 5.716) -- (3.755, 11.700);
\draw[thin] (0.300, 7.448) -- (2.755, 11.700);
\draw[thin,dotted] (0.300, 9.180) -- (1.755, 11.700);
\draw[thin] (0.300, 10.912) -- (0.755, 11.700);
\draw[thin,dotted] (0.300, 1.212) -- (0.827, 0.300);
\draw[thin] (0.300, 2.944) -- (1.827, 0.300);
\draw[thin,dotted] (0.300, 4.677) -- (2.827, 0.300);
\draw[thin] (0.300, 6.409) -- (3.827, 0.300);
\draw[thin,dotted] (0.300, 8.141) -- (4.827, 0.300);
\draw[thin] (0.300, 9.873) -- (5.827, 0.300);
\draw[thin,dotted] (0.300, 11.605) -- (6.827, 0.300);
\draw[thin] (1.245, 11.700) -- (7.827, 0.300);
\draw[thin,dotted] (2.245, 11.700) -- (8.827, 0.300);
\draw[thin] (3.245, 11.700) -- (9.827, 0.300);
\draw[thin,dotted] (4.245, 11.700) -- (10.827, 0.300);
\draw[thin] (5.245, 11.700) -- (11.700, 0.520);
\draw[thin,dotted] (6.245, 11.700) -- (11.700, 2.252);
\draw[thin] (7.245, 11.700) -- (11.700, 3.984);
\draw[thin,dotted] (8.245, 11.700) -- (11.700, 5.716);
\draw[thin] (9.245, 11.700) -- (11.700, 7.448);
\draw[thin,dotted] (10.245, 11.700) -- (11.700, 9.180);
\draw[thin] (11.245, 11.700) -- (11.700, 10.912);
\draw[thin,dotted] (0.300, 0.866) -- (11.700, 0.866);
\draw[thin] (0.300, 1.732) -- (11.700, 1.732);
\draw[thin,dotted] (0.300, 2.598) -- (11.700, 2.598);
\draw[thin] (0.300, 3.464) -- (11.700, 3.464);
\draw[thin,dotted] (0.300, 4.330) -- (11.700, 4.330);
\draw[thin] (0.300, 5.196) -- (11.700, 5.196);
\draw[thin,dotted] (0.300, 6.062) -- (11.700, 6.062);
\draw[thin] (0.300, 6.928) -- (11.700, 6.928);
\draw[thin,dotted] (0.300, 7.794) -- (11.700, 7.794);
\draw[thin] (0.300, 8.660) -- (11.700, 8.660);
\draw[thin,dotted] (0.300, 9.526) -- (11.700, 9.526);
\draw[thin] (0.300, 10.392) -- (11.700, 10.392);
\draw[thin,dotted] (0.300, 11.258) -- (11.700, 11.258);		
		\draw(6,-0.75) node{\footnotesize $(r,c) = (0,0)$};	
	\end{tikzpicture}
&
	\begin{tikzpicture}[scale = 0.4]
		\draw[color=black, fill=cyan] (5.500, 9.526) -- (2.500, 4.330) -- (3.000, 3.464) -- (8.000, 3.464) -- (8.500, 4.330) -- cycle;
		\draw[pattern={north east lines}] (6.000, 10.392) -- (2.000, 3.464) -- (8.000, 3.464) -- (9.000, 5.196) -- cycle;
		\draw[thin,dotted] (11.173, 0.300) -- (11.700, 1.212);
\draw[thin] (10.173, 0.300) -- (11.700, 2.944);
\draw[thin,dotted] (9.173, 0.300) -- (11.700, 4.677);
\draw[thin] (8.173, 0.300) -- (11.700, 6.409);
\draw[thin,dotted] (7.173, 0.300) -- (11.700, 8.141);
\draw[thin] (6.173, 0.300) -- (11.700, 9.873);
\draw[thin,dotted] (5.173, 0.300) -- (11.700, 11.605);
\draw[thin] (4.173, 0.300) -- (10.755, 11.700);
\draw[thin,dotted] (3.173, 0.300) -- (9.755, 11.700);
\draw[thin] (2.173, 0.300) -- (8.755, 11.700);
\draw[thin,dotted] (1.173, 0.300) -- (7.755, 11.700);
\draw[thin] (0.300, 0.520) -- (6.755, 11.700);
\draw[thin,dotted] (0.300, 2.252) -- (5.755, 11.700);
\draw[thin] (0.300, 3.984) -- (4.755, 11.700);
\draw[thin,dotted] (0.300, 5.716) -- (3.755, 11.700);
\draw[thin] (0.300, 7.448) -- (2.755, 11.700);
\draw[thin,dotted] (0.300, 9.180) -- (1.755, 11.700);
\draw[thin] (0.300, 10.912) -- (0.755, 11.700);
\draw[thin,dotted] (0.300, 1.212) -- (0.827, 0.300);
\draw[thin] (0.300, 2.944) -- (1.827, 0.300);
\draw[thin,dotted] (0.300, 4.677) -- (2.827, 0.300);
\draw[thin] (0.300, 6.409) -- (3.827, 0.300);
\draw[thin,dotted] (0.300, 8.141) -- (4.827, 0.300);
\draw[thin] (0.300, 9.873) -- (5.827, 0.300);
\draw[thin,dotted] (0.300, 11.605) -- (6.827, 0.300);
\draw[thin] (1.245, 11.700) -- (7.827, 0.300);
\draw[thin,dotted] (2.245, 11.700) -- (8.827, 0.300);
\draw[thin] (3.245, 11.700) -- (9.827, 0.300);
\draw[thin,dotted] (4.245, 11.700) -- (10.827, 0.300);
\draw[thin] (5.245, 11.700) -- (11.700, 0.520);
\draw[thin,dotted] (6.245, 11.700) -- (11.700, 2.252);
\draw[thin] (7.245, 11.700) -- (11.700, 3.984);
\draw[thin,dotted] (8.245, 11.700) -- (11.700, 5.716);
\draw[thin] (9.245, 11.700) -- (11.700, 7.448);
\draw[thin,dotted] (10.245, 11.700) -- (11.700, 9.180);
\draw[thin] (11.245, 11.700) -- (11.700, 10.912);
\draw[thin,dotted] (0.300, 0.866) -- (11.700, 0.866);
\draw[thin] (0.300, 1.732) -- (11.700, 1.732);
\draw[thin,dotted] (0.300, 2.598) -- (11.700, 2.598);
\draw[thin] (0.300, 3.464) -- (11.700, 3.464);
\draw[thin,dotted] (0.300, 4.330) -- (11.700, 4.330);
\draw[thin] (0.300, 5.196) -- (11.700, 5.196);
\draw[thin,dotted] (0.300, 6.062) -- (11.700, 6.062);
\draw[thin] (0.300, 6.928) -- (11.700, 6.928);
\draw[thin,dotted] (0.300, 7.794) -- (11.700, 7.794);
\draw[thin] (0.300, 8.660) -- (11.700, 8.660);
\draw[thin,dotted] (0.300, 9.526) -- (11.700, 9.526);
\draw[thin] (0.300, 10.392) -- (11.700, 10.392);
\draw[thin,dotted] (0.300, 11.258) -- (11.700, 11.258);		
		\draw(6,-0.75) node{\footnotesize $(r,c) = (1,0)$};	
	\end{tikzpicture}
\\
	\begin{tikzpicture}[scale = 0.4]
		\draw[color=black, fill=cyan] (6.500, 9.526) -- (3.500, 4.330) -- (4.000, 3.464) -- (9.000, 3.464) -- (9.500, 4.330) -- cycle;
		\draw[pattern={north east lines}] (6.000, 10.392) -- (3.000, 5.196) -- (4.000, 3.464) -- (10.000, 3.464) -- cycle;
		\draw[thin,dotted] (11.173, 0.300) -- (11.700, 1.212);
\draw[thin] (10.173, 0.300) -- (11.700, 2.944);
\draw[thin,dotted] (9.173, 0.300) -- (11.700, 4.677);
\draw[thin] (8.173, 0.300) -- (11.700, 6.409);
\draw[thin,dotted] (7.173, 0.300) -- (11.700, 8.141);
\draw[thin] (6.173, 0.300) -- (11.700, 9.873);
\draw[thin,dotted] (5.173, 0.300) -- (11.700, 11.605);
\draw[thin] (4.173, 0.300) -- (10.755, 11.700);
\draw[thin,dotted] (3.173, 0.300) -- (9.755, 11.700);
\draw[thin] (2.173, 0.300) -- (8.755, 11.700);
\draw[thin,dotted] (1.173, 0.300) -- (7.755, 11.700);
\draw[thin] (0.300, 0.520) -- (6.755, 11.700);
\draw[thin,dotted] (0.300, 2.252) -- (5.755, 11.700);
\draw[thin] (0.300, 3.984) -- (4.755, 11.700);
\draw[thin,dotted] (0.300, 5.716) -- (3.755, 11.700);
\draw[thin] (0.300, 7.448) -- (2.755, 11.700);
\draw[thin,dotted] (0.300, 9.180) -- (1.755, 11.700);
\draw[thin] (0.300, 10.912) -- (0.755, 11.700);
\draw[thin,dotted] (0.300, 1.212) -- (0.827, 0.300);
\draw[thin] (0.300, 2.944) -- (1.827, 0.300);
\draw[thin,dotted] (0.300, 4.677) -- (2.827, 0.300);
\draw[thin] (0.300, 6.409) -- (3.827, 0.300);
\draw[thin,dotted] (0.300, 8.141) -- (4.827, 0.300);
\draw[thin] (0.300, 9.873) -- (5.827, 0.300);
\draw[thin,dotted] (0.300, 11.605) -- (6.827, 0.300);
\draw[thin] (1.245, 11.700) -- (7.827, 0.300);
\draw[thin,dotted] (2.245, 11.700) -- (8.827, 0.300);
\draw[thin] (3.245, 11.700) -- (9.827, 0.300);
\draw[thin,dotted] (4.245, 11.700) -- (10.827, 0.300);
\draw[thin] (5.245, 11.700) -- (11.700, 0.520);
\draw[thin,dotted] (6.245, 11.700) -- (11.700, 2.252);
\draw[thin] (7.245, 11.700) -- (11.700, 3.984);
\draw[thin,dotted] (8.245, 11.700) -- (11.700, 5.716);
\draw[thin] (9.245, 11.700) -- (11.700, 7.448);
\draw[thin,dotted] (10.245, 11.700) -- (11.700, 9.180);
\draw[thin] (11.245, 11.700) -- (11.700, 10.912);
\draw[thin,dotted] (0.300, 0.866) -- (11.700, 0.866);
\draw[thin] (0.300, 1.732) -- (11.700, 1.732);
\draw[thin,dotted] (0.300, 2.598) -- (11.700, 2.598);
\draw[thin] (0.300, 3.464) -- (11.700, 3.464);
\draw[thin,dotted] (0.300, 4.330) -- (11.700, 4.330);
\draw[thin] (0.300, 5.196) -- (11.700, 5.196);
\draw[thin,dotted] (0.300, 6.062) -- (11.700, 6.062);
\draw[thin] (0.300, 6.928) -- (11.700, 6.928);
\draw[thin,dotted] (0.300, 7.794) -- (11.700, 7.794);
\draw[thin] (0.300, 8.660) -- (11.700, 8.660);
\draw[thin,dotted] (0.300, 9.526) -- (11.700, 9.526);
\draw[thin] (0.300, 10.392) -- (11.700, 10.392);
\draw[thin,dotted] (0.300, 11.258) -- (11.700, 11.258);		
		\draw(6,-0.75) node{\footnotesize $(r,c) = (1,2)$};	
	\end{tikzpicture}
&
	\begin{tikzpicture}[scale = 0.4]
		\draw[color=black, fill=cyan] (6.000, 8.660) -- (3.000, 3.464) -- (3.500, 2.598) -- (8.500, 2.598) -- (9.000, 3.464) -- cycle;
		\draw[pattern={north east lines}] (5.000, 8.660) -- (2.000, 3.464) -- (3.000, 1.732) -- (9.000, 1.732) -- (10.000, 3.464) -- (7.000, 8.660) -- cycle;
		\draw[thin,dotted] (11.173, 0.300) -- (11.700, 1.212);
\draw[thin] (10.173, 0.300) -- (11.700, 2.944);
\draw[thin,dotted] (9.173, 0.300) -- (11.700, 4.677);
\draw[thin] (8.173, 0.300) -- (11.700, 6.409);
\draw[thin,dotted] (7.173, 0.300) -- (11.700, 8.141);
\draw[thin] (6.173, 0.300) -- (11.700, 9.873);
\draw[thin,dotted] (5.173, 0.300) -- (11.700, 11.605);
\draw[thin] (4.173, 0.300) -- (10.755, 11.700);
\draw[thin,dotted] (3.173, 0.300) -- (9.755, 11.700);
\draw[thin] (2.173, 0.300) -- (8.755, 11.700);
\draw[thin,dotted] (1.173, 0.300) -- (7.755, 11.700);
\draw[thin] (0.300, 0.520) -- (6.755, 11.700);
\draw[thin,dotted] (0.300, 2.252) -- (5.755, 11.700);
\draw[thin] (0.300, 3.984) -- (4.755, 11.700);
\draw[thin,dotted] (0.300, 5.716) -- (3.755, 11.700);
\draw[thin] (0.300, 7.448) -- (2.755, 11.700);
\draw[thin,dotted] (0.300, 9.180) -- (1.755, 11.700);
\draw[thin] (0.300, 10.912) -- (0.755, 11.700);
\draw[thin,dotted] (0.300, 1.212) -- (0.827, 0.300);
\draw[thin] (0.300, 2.944) -- (1.827, 0.300);
\draw[thin,dotted] (0.300, 4.677) -- (2.827, 0.300);
\draw[thin] (0.300, 6.409) -- (3.827, 0.300);
\draw[thin,dotted] (0.300, 8.141) -- (4.827, 0.300);
\draw[thin] (0.300, 9.873) -- (5.827, 0.300);
\draw[thin,dotted] (0.300, 11.605) -- (6.827, 0.300);
\draw[thin] (1.245, 11.700) -- (7.827, 0.300);
\draw[thin,dotted] (2.245, 11.700) -- (8.827, 0.300);
\draw[thin] (3.245, 11.700) -- (9.827, 0.300);
\draw[thin,dotted] (4.245, 11.700) -- (10.827, 0.300);
\draw[thin] (5.245, 11.700) -- (11.700, 0.520);
\draw[thin,dotted] (6.245, 11.700) -- (11.700, 2.252);
\draw[thin] (7.245, 11.700) -- (11.700, 3.984);
\draw[thin,dotted] (8.245, 11.700) -- (11.700, 5.716);
\draw[thin] (9.245, 11.700) -- (11.700, 7.448);
\draw[thin,dotted] (10.245, 11.700) -- (11.700, 9.180);
\draw[thin] (11.245, 11.700) -- (11.700, 10.912);
\draw[thin,dotted] (0.300, 0.866) -- (11.700, 0.866);
\draw[thin] (0.300, 1.732) -- (11.700, 1.732);
\draw[thin,dotted] (0.300, 2.598) -- (11.700, 2.598);
\draw[thin] (0.300, 3.464) -- (11.700, 3.464);
\draw[thin,dotted] (0.300, 4.330) -- (11.700, 4.330);
\draw[thin] (0.300, 5.196) -- (11.700, 5.196);
\draw[thin,dotted] (0.300, 6.062) -- (11.700, 6.062);
\draw[thin] (0.300, 6.928) -- (11.700, 6.928);
\draw[thin,dotted] (0.300, 7.794) -- (11.700, 7.794);
\draw[thin] (0.300, 8.660) -- (11.700, 8.660);
\draw[thin,dotted] (0.300, 9.526) -- (11.700, 9.526);
\draw[thin] (0.300, 10.392) -- (11.700, 10.392);
\draw[thin,dotted] (0.300, 11.258) -- (11.700, 11.258);
		\draw(6,-0.75) node{\footnotesize $(r,c) = (2,2)$};	
	\end{tikzpicture}
\end{tabular}
\end{center}
\caption{The blue shaded regions represent elements of $P_{r,c}(\mytikz{0.07}{\protect\trilower},T, 2n+1)$, for the different values of $(r,c)$. The illustrations show how these are modified to reach the hatched regions, from which we can deduce the recursive expression in \eqref{eq:recCodd}.}
\label{fig:recCodd}
\end{figure}

\begin{equation}
\label{eq:recDeven}
\begin{split}
D_{2n} 
&= -6 + \sum_{(r,c)\in I} |P_{r,c}(\mytikz{0.08}{\triall},T, 2n)| \\
&= -6 + A_{n} + D_{n+1} + D_{n+1} + D_{n+1}.
\end{split}
\end{equation}
See Figure~\ref{fig:recDeven} for a visualisation of the deduction of this recursion.

\begin{figure}[ht]
\begin{center}
\begin{tabular}{cc}
	\begin{tikzpicture}[scale = 0.4]
		\draw[color=black, fill=cyan] (5.500, 9.526) -- (3.500, 6.062) -- (4.000, 5.196) -- (8.000, 5.196) -- (8.500, 6.062) -- (6.500, 9.526) -- cycle;
		\draw[pattern={north east lines}] (6.000, 10.392) -- (3.000, 5.196) -- (9.000, 5.196) -- cycle;
		\draw[thin,dotted] (11.173, 0.300) -- (11.700, 1.212);
\draw[thin] (10.173, 0.300) -- (11.700, 2.944);
\draw[thin,dotted] (9.173, 0.300) -- (11.700, 4.677);
\draw[thin] (8.173, 0.300) -- (11.700, 6.409);
\draw[thin,dotted] (7.173, 0.300) -- (11.700, 8.141);
\draw[thin] (6.173, 0.300) -- (11.700, 9.873);
\draw[thin,dotted] (5.173, 0.300) -- (11.700, 11.605);
\draw[thin] (4.173, 0.300) -- (10.755, 11.700);
\draw[thin,dotted] (3.173, 0.300) -- (9.755, 11.700);
\draw[thin] (2.173, 0.300) -- (8.755, 11.700);
\draw[thin,dotted] (1.173, 0.300) -- (7.755, 11.700);
\draw[thin] (0.300, 0.520) -- (6.755, 11.700);
\draw[thin,dotted] (0.300, 2.252) -- (5.755, 11.700);
\draw[thin] (0.300, 3.984) -- (4.755, 11.700);
\draw[thin,dotted] (0.300, 5.716) -- (3.755, 11.700);
\draw[thin] (0.300, 7.448) -- (2.755, 11.700);
\draw[thin,dotted] (0.300, 9.180) -- (1.755, 11.700);
\draw[thin] (0.300, 10.912) -- (0.755, 11.700);
\draw[thin,dotted] (0.300, 1.212) -- (0.827, 0.300);
\draw[thin] (0.300, 2.944) -- (1.827, 0.300);
\draw[thin,dotted] (0.300, 4.677) -- (2.827, 0.300);
\draw[thin] (0.300, 6.409) -- (3.827, 0.300);
\draw[thin,dotted] (0.300, 8.141) -- (4.827, 0.300);
\draw[thin] (0.300, 9.873) -- (5.827, 0.300);
\draw[thin,dotted] (0.300, 11.605) -- (6.827, 0.300);
\draw[thin] (1.245, 11.700) -- (7.827, 0.300);
\draw[thin,dotted] (2.245, 11.700) -- (8.827, 0.300);
\draw[thin] (3.245, 11.700) -- (9.827, 0.300);
\draw[thin,dotted] (4.245, 11.700) -- (10.827, 0.300);
\draw[thin] (5.245, 11.700) -- (11.700, 0.520);
\draw[thin,dotted] (6.245, 11.700) -- (11.700, 2.252);
\draw[thin] (7.245, 11.700) -- (11.700, 3.984);
\draw[thin,dotted] (8.245, 11.700) -- (11.700, 5.716);
\draw[thin] (9.245, 11.700) -- (11.700, 7.448);
\draw[thin,dotted] (10.245, 11.700) -- (11.700, 9.180);
\draw[thin] (11.245, 11.700) -- (11.700, 10.912);
\draw[thin,dotted] (0.300, 0.866) -- (11.700, 0.866);
\draw[thin] (0.300, 1.732) -- (11.700, 1.732);
\draw[thin,dotted] (0.300, 2.598) -- (11.700, 2.598);
\draw[thin] (0.300, 3.464) -- (11.700, 3.464);
\draw[thin,dotted] (0.300, 4.330) -- (11.700, 4.330);
\draw[thin] (0.300, 5.196) -- (11.700, 5.196);
\draw[thin,dotted] (0.300, 6.062) -- (11.700, 6.062);
\draw[thin] (0.300, 6.928) -- (11.700, 6.928);
\draw[thin,dotted] (0.300, 7.794) -- (11.700, 7.794);
\draw[thin] (0.300, 8.660) -- (11.700, 8.660);
\draw[thin,dotted] (0.300, 9.526) -- (11.700, 9.526);
\draw[thin] (0.300, 10.392) -- (11.700, 10.392);
\draw[thin,dotted] (0.300, 11.258) -- (11.700, 11.258);		
		\draw(6,-0.75) node{\footnotesize $(r,c) = (0,0)$};	
	\end{tikzpicture}
&
	\begin{tikzpicture}[scale = 0.4]
		\draw[color=black, fill=cyan] (5.000, 8.660) -- (3.000, 5.196) -- (3.500, 4.330) -- (7.500, 4.330) -- (8.000, 5.196) -- (6.000, 8.660) -- cycle;
		\draw[pattern={north east lines}] (5.000, 8.660) -- (3.000, 5.196) -- (4.000, 3.464) -- (8.000, 3.464) -- (9.000, 5.196) -- (7.000, 8.660) -- cycle;
		\draw[thin,dotted] (11.173, 0.300) -- (11.700, 1.212);
\draw[thin] (10.173, 0.300) -- (11.700, 2.944);
\draw[thin,dotted] (9.173, 0.300) -- (11.700, 4.677);
\draw[thin] (8.173, 0.300) -- (11.700, 6.409);
\draw[thin,dotted] (7.173, 0.300) -- (11.700, 8.141);
\draw[thin] (6.173, 0.300) -- (11.700, 9.873);
\draw[thin,dotted] (5.173, 0.300) -- (11.700, 11.605);
\draw[thin] (4.173, 0.300) -- (10.755, 11.700);
\draw[thin,dotted] (3.173, 0.300) -- (9.755, 11.700);
\draw[thin] (2.173, 0.300) -- (8.755, 11.700);
\draw[thin,dotted] (1.173, 0.300) -- (7.755, 11.700);
\draw[thin] (0.300, 0.520) -- (6.755, 11.700);
\draw[thin,dotted] (0.300, 2.252) -- (5.755, 11.700);
\draw[thin] (0.300, 3.984) -- (4.755, 11.700);
\draw[thin,dotted] (0.300, 5.716) -- (3.755, 11.700);
\draw[thin] (0.300, 7.448) -- (2.755, 11.700);
\draw[thin,dotted] (0.300, 9.180) -- (1.755, 11.700);
\draw[thin] (0.300, 10.912) -- (0.755, 11.700);
\draw[thin,dotted] (0.300, 1.212) -- (0.827, 0.300);
\draw[thin] (0.300, 2.944) -- (1.827, 0.300);
\draw[thin,dotted] (0.300, 4.677) -- (2.827, 0.300);
\draw[thin] (0.300, 6.409) -- (3.827, 0.300);
\draw[thin,dotted] (0.300, 8.141) -- (4.827, 0.300);
\draw[thin] (0.300, 9.873) -- (5.827, 0.300);
\draw[thin,dotted] (0.300, 11.605) -- (6.827, 0.300);
\draw[thin] (1.245, 11.700) -- (7.827, 0.300);
\draw[thin,dotted] (2.245, 11.700) -- (8.827, 0.300);
\draw[thin] (3.245, 11.700) -- (9.827, 0.300);
\draw[thin,dotted] (4.245, 11.700) -- (10.827, 0.300);
\draw[thin] (5.245, 11.700) -- (11.700, 0.520);
\draw[thin,dotted] (6.245, 11.700) -- (11.700, 2.252);
\draw[thin] (7.245, 11.700) -- (11.700, 3.984);
\draw[thin,dotted] (8.245, 11.700) -- (11.700, 5.716);
\draw[thin] (9.245, 11.700) -- (11.700, 7.448);
\draw[thin,dotted] (10.245, 11.700) -- (11.700, 9.180);
\draw[thin] (11.245, 11.700) -- (11.700, 10.912);
\draw[thin,dotted] (0.300, 0.866) -- (11.700, 0.866);
\draw[thin] (0.300, 1.732) -- (11.700, 1.732);
\draw[thin,dotted] (0.300, 2.598) -- (11.700, 2.598);
\draw[thin] (0.300, 3.464) -- (11.700, 3.464);
\draw[thin,dotted] (0.300, 4.330) -- (11.700, 4.330);
\draw[thin] (0.300, 5.196) -- (11.700, 5.196);
\draw[thin,dotted] (0.300, 6.062) -- (11.700, 6.062);
\draw[thin] (0.300, 6.928) -- (11.700, 6.928);
\draw[thin,dotted] (0.300, 7.794) -- (11.700, 7.794);
\draw[thin] (0.300, 8.660) -- (11.700, 8.660);
\draw[thin,dotted] (0.300, 9.526) -- (11.700, 9.526);
\draw[thin] (0.300, 10.392) -- (11.700, 10.392);
\draw[thin,dotted] (0.300, 11.258) -- (11.700, 11.258);		
		\draw(6,-0.75) node{\footnotesize $(r,c) = (1,0)$};	
	\end{tikzpicture}
\\
	\begin{tikzpicture}[scale = 0.4]
		\draw[color=black, fill=cyan] (6.000, 8.660) -- (4.000, 5.196) -- (4.500, 4.330) -- (8.500, 4.330) -- (9.000, 5.196) -- (7.000, 8.660) -- cycle;
		\draw[pattern={north east lines}] (5.000, 8.660) -- (3.000, 5.196) -- (4.000, 3.464) -- (8.000, 3.464) -- (9.000, 5.196) -- (7.000, 8.660) -- cycle;
		\draw[thin,dotted] (11.173, 0.300) -- (11.700, 1.212);
\draw[thin] (10.173, 0.300) -- (11.700, 2.944);
\draw[thin,dotted] (9.173, 0.300) -- (11.700, 4.677);
\draw[thin] (8.173, 0.300) -- (11.700, 6.409);
\draw[thin,dotted] (7.173, 0.300) -- (11.700, 8.141);
\draw[thin] (6.173, 0.300) -- (11.700, 9.873);
\draw[thin,dotted] (5.173, 0.300) -- (11.700, 11.605);
\draw[thin] (4.173, 0.300) -- (10.755, 11.700);
\draw[thin,dotted] (3.173, 0.300) -- (9.755, 11.700);
\draw[thin] (2.173, 0.300) -- (8.755, 11.700);
\draw[thin,dotted] (1.173, 0.300) -- (7.755, 11.700);
\draw[thin] (0.300, 0.520) -- (6.755, 11.700);
\draw[thin,dotted] (0.300, 2.252) -- (5.755, 11.700);
\draw[thin] (0.300, 3.984) -- (4.755, 11.700);
\draw[thin,dotted] (0.300, 5.716) -- (3.755, 11.700);
\draw[thin] (0.300, 7.448) -- (2.755, 11.700);
\draw[thin,dotted] (0.300, 9.180) -- (1.755, 11.700);
\draw[thin] (0.300, 10.912) -- (0.755, 11.700);
\draw[thin,dotted] (0.300, 1.212) -- (0.827, 0.300);
\draw[thin] (0.300, 2.944) -- (1.827, 0.300);
\draw[thin,dotted] (0.300, 4.677) -- (2.827, 0.300);
\draw[thin] (0.300, 6.409) -- (3.827, 0.300);
\draw[thin,dotted] (0.300, 8.141) -- (4.827, 0.300);
\draw[thin] (0.300, 9.873) -- (5.827, 0.300);
\draw[thin,dotted] (0.300, 11.605) -- (6.827, 0.300);
\draw[thin] (1.245, 11.700) -- (7.827, 0.300);
\draw[thin,dotted] (2.245, 11.700) -- (8.827, 0.300);
\draw[thin] (3.245, 11.700) -- (9.827, 0.300);
\draw[thin,dotted] (4.245, 11.700) -- (10.827, 0.300);
\draw[thin] (5.245, 11.700) -- (11.700, 0.520);
\draw[thin,dotted] (6.245, 11.700) -- (11.700, 2.252);
\draw[thin] (7.245, 11.700) -- (11.700, 3.984);
\draw[thin,dotted] (8.245, 11.700) -- (11.700, 5.716);
\draw[thin] (9.245, 11.700) -- (11.700, 7.448);
\draw[thin,dotted] (10.245, 11.700) -- (11.700, 9.180);
\draw[thin] (11.245, 11.700) -- (11.700, 10.912);
\draw[thin,dotted] (0.300, 0.866) -- (11.700, 0.866);
\draw[thin] (0.300, 1.732) -- (11.700, 1.732);
\draw[thin,dotted] (0.300, 2.598) -- (11.700, 2.598);
\draw[thin] (0.300, 3.464) -- (11.700, 3.464);
\draw[thin,dotted] (0.300, 4.330) -- (11.700, 4.330);
\draw[thin] (0.300, 5.196) -- (11.700, 5.196);
\draw[thin,dotted] (0.300, 6.062) -- (11.700, 6.062);
\draw[thin] (0.300, 6.928) -- (11.700, 6.928);
\draw[thin,dotted] (0.300, 7.794) -- (11.700, 7.794);
\draw[thin] (0.300, 8.660) -- (11.700, 8.660);
\draw[thin,dotted] (0.300, 9.526) -- (11.700, 9.526);
\draw[thin] (0.300, 10.392) -- (11.700, 10.392);
\draw[thin,dotted] (0.300, 11.258) -- (11.700, 11.258);		
		\draw(6,-0.75) node{\footnotesize $(r,c) = (1,2)$};	
	\end{tikzpicture}
&
	\begin{tikzpicture}[scale = 0.4]
		\draw[color=black, fill=cyan] (5.500, 7.794) -- (3.500, 4.330) -- (4.000, 3.464) -- (8.000, 3.464) -- (8.500, 4.330) -- (6.500, 7.794) -- cycle;
		\draw[pattern={north east lines}] (5.000, 8.660) -- (3.000, 5.196) -- (4.000, 3.464) -- (8.000, 3.464) -- (9.000, 5.196) -- (7.000, 8.660) -- cycle;
		\draw[thin,dotted] (11.173, 0.300) -- (11.700, 1.212);
\draw[thin] (10.173, 0.300) -- (11.700, 2.944);
\draw[thin,dotted] (9.173, 0.300) -- (11.700, 4.677);
\draw[thin] (8.173, 0.300) -- (11.700, 6.409);
\draw[thin,dotted] (7.173, 0.300) -- (11.700, 8.141);
\draw[thin] (6.173, 0.300) -- (11.700, 9.873);
\draw[thin,dotted] (5.173, 0.300) -- (11.700, 11.605);
\draw[thin] (4.173, 0.300) -- (10.755, 11.700);
\draw[thin,dotted] (3.173, 0.300) -- (9.755, 11.700);
\draw[thin] (2.173, 0.300) -- (8.755, 11.700);
\draw[thin,dotted] (1.173, 0.300) -- (7.755, 11.700);
\draw[thin] (0.300, 0.520) -- (6.755, 11.700);
\draw[thin,dotted] (0.300, 2.252) -- (5.755, 11.700);
\draw[thin] (0.300, 3.984) -- (4.755, 11.700);
\draw[thin,dotted] (0.300, 5.716) -- (3.755, 11.700);
\draw[thin] (0.300, 7.448) -- (2.755, 11.700);
\draw[thin,dotted] (0.300, 9.180) -- (1.755, 11.700);
\draw[thin] (0.300, 10.912) -- (0.755, 11.700);
\draw[thin,dotted] (0.300, 1.212) -- (0.827, 0.300);
\draw[thin] (0.300, 2.944) -- (1.827, 0.300);
\draw[thin,dotted] (0.300, 4.677) -- (2.827, 0.300);
\draw[thin] (0.300, 6.409) -- (3.827, 0.300);
\draw[thin,dotted] (0.300, 8.141) -- (4.827, 0.300);
\draw[thin] (0.300, 9.873) -- (5.827, 0.300);
\draw[thin,dotted] (0.300, 11.605) -- (6.827, 0.300);
\draw[thin] (1.245, 11.700) -- (7.827, 0.300);
\draw[thin,dotted] (2.245, 11.700) -- (8.827, 0.300);
\draw[thin] (3.245, 11.700) -- (9.827, 0.300);
\draw[thin,dotted] (4.245, 11.700) -- (10.827, 0.300);
\draw[thin] (5.245, 11.700) -- (11.700, 0.520);
\draw[thin,dotted] (6.245, 11.700) -- (11.700, 2.252);
\draw[thin] (7.245, 11.700) -- (11.700, 3.984);
\draw[thin,dotted] (8.245, 11.700) -- (11.700, 5.716);
\draw[thin] (9.245, 11.700) -- (11.700, 7.448);
\draw[thin,dotted] (10.245, 11.700) -- (11.700, 9.180);
\draw[thin] (11.245, 11.700) -- (11.700, 10.912);
\draw[thin,dotted] (0.300, 0.866) -- (11.700, 0.866);
\draw[thin] (0.300, 1.732) -- (11.700, 1.732);
\draw[thin,dotted] (0.300, 2.598) -- (11.700, 2.598);
\draw[thin] (0.300, 3.464) -- (11.700, 3.464);
\draw[thin,dotted] (0.300, 4.330) -- (11.700, 4.330);
\draw[thin] (0.300, 5.196) -- (11.700, 5.196);
\draw[thin,dotted] (0.300, 6.062) -- (11.700, 6.062);
\draw[thin] (0.300, 6.928) -- (11.700, 6.928);
\draw[thin,dotted] (0.300, 7.794) -- (11.700, 7.794);
\draw[thin] (0.300, 8.660) -- (11.700, 8.660);
\draw[thin,dotted] (0.300, 9.526) -- (11.700, 9.526);
\draw[thin] (0.300, 10.392) -- (11.700, 10.392);
\draw[thin,dotted] (0.300, 11.258) -- (11.700, 11.258);
		\draw(6,-0.75) node{\footnotesize $(r,c) = (2,2)$};	
		\end{tikzpicture}
\end{tabular}
\end{center}
\caption{The blue shaded regions represent elements of $P_{r,c}(\mytikz{0.07}{\protect\triall},T, 2n)$, for the different values of $(r,c)$. The illustrations show how these are modified to reach the hatched regions, from which we can deduce the recursive expression in \eqref{eq:recDeven}.}
\label{fig:recDeven}
\end{figure}

\begin{equation}
\label{eq:recDodd}
\begin{split}
D_{2n+1} 
&= -6 + \sum_{(r,c)\in I} |P_{r,c}(\mytikz{0.08}{\triall},T, 2n+1)| \\
&= -6 + C_{n+1} + C_{n+1} + C_{n+1} + D_{n+2}.
\end{split}
\end{equation}
See Figure~\ref{fig:recDodd} for a visualisation of the deduction of this recursion.

\begin{figure}[ht]
\begin{center}
\begin{tabular}{cc}
	\begin{tikzpicture}[scale = 0.4]
		\draw[color=black, fill=cyan] (5.500, 9.526) -- (3.000, 5.196) -- (3.500, 4.330) -- (8.500, 4.330) -- (9.000, 5.196) -- (6.500, 9.526) -- cycle;
		\draw[pattern={north east lines}] (6.000, 10.392) -- (3.000, 5.196) -- (4.000, 3.464) -- (8.000, 3.464) -- (9.000, 5.196) -- cycle;
		\draw[thin,dotted] (11.173, 0.300) -- (11.700, 1.212);
\draw[thin] (10.173, 0.300) -- (11.700, 2.944);
\draw[thin,dotted] (9.173, 0.300) -- (11.700, 4.677);
\draw[thin] (8.173, 0.300) -- (11.700, 6.409);
\draw[thin,dotted] (7.173, 0.300) -- (11.700, 8.141);
\draw[thin] (6.173, 0.300) -- (11.700, 9.873);
\draw[thin,dotted] (5.173, 0.300) -- (11.700, 11.605);
\draw[thin] (4.173, 0.300) -- (10.755, 11.700);
\draw[thin,dotted] (3.173, 0.300) -- (9.755, 11.700);
\draw[thin] (2.173, 0.300) -- (8.755, 11.700);
\draw[thin,dotted] (1.173, 0.300) -- (7.755, 11.700);
\draw[thin] (0.300, 0.520) -- (6.755, 11.700);
\draw[thin,dotted] (0.300, 2.252) -- (5.755, 11.700);
\draw[thin] (0.300, 3.984) -- (4.755, 11.700);
\draw[thin,dotted] (0.300, 5.716) -- (3.755, 11.700);
\draw[thin] (0.300, 7.448) -- (2.755, 11.700);
\draw[thin,dotted] (0.300, 9.180) -- (1.755, 11.700);
\draw[thin] (0.300, 10.912) -- (0.755, 11.700);
\draw[thin,dotted] (0.300, 1.212) -- (0.827, 0.300);
\draw[thin] (0.300, 2.944) -- (1.827, 0.300);
\draw[thin,dotted] (0.300, 4.677) -- (2.827, 0.300);
\draw[thin] (0.300, 6.409) -- (3.827, 0.300);
\draw[thin,dotted] (0.300, 8.141) -- (4.827, 0.300);
\draw[thin] (0.300, 9.873) -- (5.827, 0.300);
\draw[thin,dotted] (0.300, 11.605) -- (6.827, 0.300);
\draw[thin] (1.245, 11.700) -- (7.827, 0.300);
\draw[thin,dotted] (2.245, 11.700) -- (8.827, 0.300);
\draw[thin] (3.245, 11.700) -- (9.827, 0.300);
\draw[thin,dotted] (4.245, 11.700) -- (10.827, 0.300);
\draw[thin] (5.245, 11.700) -- (11.700, 0.520);
\draw[thin,dotted] (6.245, 11.700) -- (11.700, 2.252);
\draw[thin] (7.245, 11.700) -- (11.700, 3.984);
\draw[thin,dotted] (8.245, 11.700) -- (11.700, 5.716);
\draw[thin] (9.245, 11.700) -- (11.700, 7.448);
\draw[thin,dotted] (10.245, 11.700) -- (11.700, 9.180);
\draw[thin] (11.245, 11.700) -- (11.700, 10.912);
\draw[thin,dotted] (0.300, 0.866) -- (11.700, 0.866);
\draw[thin] (0.300, 1.732) -- (11.700, 1.732);
\draw[thin,dotted] (0.300, 2.598) -- (11.700, 2.598);
\draw[thin] (0.300, 3.464) -- (11.700, 3.464);
\draw[thin,dotted] (0.300, 4.330) -- (11.700, 4.330);
\draw[thin] (0.300, 5.196) -- (11.700, 5.196);
\draw[thin,dotted] (0.300, 6.062) -- (11.700, 6.062);
\draw[thin] (0.300, 6.928) -- (11.700, 6.928);
\draw[thin,dotted] (0.300, 7.794) -- (11.700, 7.794);
\draw[thin] (0.300, 8.660) -- (11.700, 8.660);
\draw[thin,dotted] (0.300, 9.526) -- (11.700, 9.526);
\draw[thin] (0.300, 10.392) -- (11.700, 10.392);
\draw[thin,dotted] (0.300, 11.258) -- (11.700, 11.258);		
		\draw(6,-0.75) node{\footnotesize $(r,c) = (0,0)$};	
	\end{tikzpicture}
&
	\begin{tikzpicture}[scale = 0.4]
		\draw[color=black, fill=cyan] (5.000, 8.660) -- (2.500, 4.330) -- (3.000, 3.464) -- (8.000, 3.464) -- (8.500, 4.330) -- (6.000, 8.660) -- cycle;
		\draw[pattern={north east lines}] (5.000, 8.660) -- (2.000, 3.464) -- (8.000, 3.464) -- (9.000, 5.196) -- (7.000, 8.660) -- cycle;
		\draw[thin,dotted] (11.173, 0.300) -- (11.700, 1.212);
\draw[thin] (10.173, 0.300) -- (11.700, 2.944);
\draw[thin,dotted] (9.173, 0.300) -- (11.700, 4.677);
\draw[thin] (8.173, 0.300) -- (11.700, 6.409);
\draw[thin,dotted] (7.173, 0.300) -- (11.700, 8.141);
\draw[thin] (6.173, 0.300) -- (11.700, 9.873);
\draw[thin,dotted] (5.173, 0.300) -- (11.700, 11.605);
\draw[thin] (4.173, 0.300) -- (10.755, 11.700);
\draw[thin,dotted] (3.173, 0.300) -- (9.755, 11.700);
\draw[thin] (2.173, 0.300) -- (8.755, 11.700);
\draw[thin,dotted] (1.173, 0.300) -- (7.755, 11.700);
\draw[thin] (0.300, 0.520) -- (6.755, 11.700);
\draw[thin,dotted] (0.300, 2.252) -- (5.755, 11.700);
\draw[thin] (0.300, 3.984) -- (4.755, 11.700);
\draw[thin,dotted] (0.300, 5.716) -- (3.755, 11.700);
\draw[thin] (0.300, 7.448) -- (2.755, 11.700);
\draw[thin,dotted] (0.300, 9.180) -- (1.755, 11.700);
\draw[thin] (0.300, 10.912) -- (0.755, 11.700);
\draw[thin,dotted] (0.300, 1.212) -- (0.827, 0.300);
\draw[thin] (0.300, 2.944) -- (1.827, 0.300);
\draw[thin,dotted] (0.300, 4.677) -- (2.827, 0.300);
\draw[thin] (0.300, 6.409) -- (3.827, 0.300);
\draw[thin,dotted] (0.300, 8.141) -- (4.827, 0.300);
\draw[thin] (0.300, 9.873) -- (5.827, 0.300);
\draw[thin,dotted] (0.300, 11.605) -- (6.827, 0.300);
\draw[thin] (1.245, 11.700) -- (7.827, 0.300);
\draw[thin,dotted] (2.245, 11.700) -- (8.827, 0.300);
\draw[thin] (3.245, 11.700) -- (9.827, 0.300);
\draw[thin,dotted] (4.245, 11.700) -- (10.827, 0.300);
\draw[thin] (5.245, 11.700) -- (11.700, 0.520);
\draw[thin,dotted] (6.245, 11.700) -- (11.700, 2.252);
\draw[thin] (7.245, 11.700) -- (11.700, 3.984);
\draw[thin,dotted] (8.245, 11.700) -- (11.700, 5.716);
\draw[thin] (9.245, 11.700) -- (11.700, 7.448);
\draw[thin,dotted] (10.245, 11.700) -- (11.700, 9.180);
\draw[thin] (11.245, 11.700) -- (11.700, 10.912);
\draw[thin,dotted] (0.300, 0.866) -- (11.700, 0.866);
\draw[thin] (0.300, 1.732) -- (11.700, 1.732);
\draw[thin,dotted] (0.300, 2.598) -- (11.700, 2.598);
\draw[thin] (0.300, 3.464) -- (11.700, 3.464);
\draw[thin,dotted] (0.300, 4.330) -- (11.700, 4.330);
\draw[thin] (0.300, 5.196) -- (11.700, 5.196);
\draw[thin,dotted] (0.300, 6.062) -- (11.700, 6.062);
\draw[thin] (0.300, 6.928) -- (11.700, 6.928);
\draw[thin,dotted] (0.300, 7.794) -- (11.700, 7.794);
\draw[thin] (0.300, 8.660) -- (11.700, 8.660);
\draw[thin,dotted] (0.300, 9.526) -- (11.700, 9.526);
\draw[thin] (0.300, 10.392) -- (11.700, 10.392);
\draw[thin,dotted] (0.300, 11.258) -- (11.700, 11.258);		
		\draw(6,-0.75) node{\footnotesize $(r,c) = (1,0)$};	
	\end{tikzpicture}
\\
	\begin{tikzpicture}[scale = 0.4]
		\draw[color=black, fill=cyan] (6.000, 8.660) -- (3.500, 4.330) -- (4.000, 3.464) -- (9.000, 3.464) -- (9.500, 4.330) -- (7.000, 8.660) -- cycle;
		\draw[pattern={north east lines}] (5.000, 8.660) -- (3.000, 5.196) -- (4.000, 3.464) -- (10.000, 3.464) -- (7.000, 8.660)-- cycle;
		\draw[thin,dotted] (11.173, 0.300) -- (11.700, 1.212);
\draw[thin] (10.173, 0.300) -- (11.700, 2.944);
\draw[thin,dotted] (9.173, 0.300) -- (11.700, 4.677);
\draw[thin] (8.173, 0.300) -- (11.700, 6.409);
\draw[thin,dotted] (7.173, 0.300) -- (11.700, 8.141);
\draw[thin] (6.173, 0.300) -- (11.700, 9.873);
\draw[thin,dotted] (5.173, 0.300) -- (11.700, 11.605);
\draw[thin] (4.173, 0.300) -- (10.755, 11.700);
\draw[thin,dotted] (3.173, 0.300) -- (9.755, 11.700);
\draw[thin] (2.173, 0.300) -- (8.755, 11.700);
\draw[thin,dotted] (1.173, 0.300) -- (7.755, 11.700);
\draw[thin] (0.300, 0.520) -- (6.755, 11.700);
\draw[thin,dotted] (0.300, 2.252) -- (5.755, 11.700);
\draw[thin] (0.300, 3.984) -- (4.755, 11.700);
\draw[thin,dotted] (0.300, 5.716) -- (3.755, 11.700);
\draw[thin] (0.300, 7.448) -- (2.755, 11.700);
\draw[thin,dotted] (0.300, 9.180) -- (1.755, 11.700);
\draw[thin] (0.300, 10.912) -- (0.755, 11.700);
\draw[thin,dotted] (0.300, 1.212) -- (0.827, 0.300);
\draw[thin] (0.300, 2.944) -- (1.827, 0.300);
\draw[thin,dotted] (0.300, 4.677) -- (2.827, 0.300);
\draw[thin] (0.300, 6.409) -- (3.827, 0.300);
\draw[thin,dotted] (0.300, 8.141) -- (4.827, 0.300);
\draw[thin] (0.300, 9.873) -- (5.827, 0.300);
\draw[thin,dotted] (0.300, 11.605) -- (6.827, 0.300);
\draw[thin] (1.245, 11.700) -- (7.827, 0.300);
\draw[thin,dotted] (2.245, 11.700) -- (8.827, 0.300);
\draw[thin] (3.245, 11.700) -- (9.827, 0.300);
\draw[thin,dotted] (4.245, 11.700) -- (10.827, 0.300);
\draw[thin] (5.245, 11.700) -- (11.700, 0.520);
\draw[thin,dotted] (6.245, 11.700) -- (11.700, 2.252);
\draw[thin] (7.245, 11.700) -- (11.700, 3.984);
\draw[thin,dotted] (8.245, 11.700) -- (11.700, 5.716);
\draw[thin] (9.245, 11.700) -- (11.700, 7.448);
\draw[thin,dotted] (10.245, 11.700) -- (11.700, 9.180);
\draw[thin] (11.245, 11.700) -- (11.700, 10.912);
\draw[thin,dotted] (0.300, 0.866) -- (11.700, 0.866);
\draw[thin] (0.300, 1.732) -- (11.700, 1.732);
\draw[thin,dotted] (0.300, 2.598) -- (11.700, 2.598);
\draw[thin] (0.300, 3.464) -- (11.700, 3.464);
\draw[thin,dotted] (0.300, 4.330) -- (11.700, 4.330);
\draw[thin] (0.300, 5.196) -- (11.700, 5.196);
\draw[thin,dotted] (0.300, 6.062) -- (11.700, 6.062);
\draw[thin] (0.300, 6.928) -- (11.700, 6.928);
\draw[thin,dotted] (0.300, 7.794) -- (11.700, 7.794);
\draw[thin] (0.300, 8.660) -- (11.700, 8.660);
\draw[thin,dotted] (0.300, 9.526) -- (11.700, 9.526);
\draw[thin] (0.300, 10.392) -- (11.700, 10.392);
\draw[thin,dotted] (0.300, 11.258) -- (11.700, 11.258);		
		\draw(6,-0.75) node{\footnotesize $(r,c) = (1,2)$};	
	\end{tikzpicture}
&
	\begin{tikzpicture}[scale = 0.4]
		\draw[color=black, fill=cyan] (5.500, 7.794) -- (3.000, 3.464) -- (3.500, 2.598) -- (8.500, 2.598) -- (9.000, 3.464) -- (6.500, 7.794) -- cycle;
		\draw[pattern={north east lines}] (5.000, 8.660) -- (2.000, 3.464) -- (3.000, 1.732) -- (9.000, 1.732) -- (10.000, 3.464) -- (7.000, 8.660) -- cycle;
		\draw[thin,dotted] (11.173, 0.300) -- (11.700, 1.212);
\draw[thin] (10.173, 0.300) -- (11.700, 2.944);
\draw[thin,dotted] (9.173, 0.300) -- (11.700, 4.677);
\draw[thin] (8.173, 0.300) -- (11.700, 6.409);
\draw[thin,dotted] (7.173, 0.300) -- (11.700, 8.141);
\draw[thin] (6.173, 0.300) -- (11.700, 9.873);
\draw[thin,dotted] (5.173, 0.300) -- (11.700, 11.605);
\draw[thin] (4.173, 0.300) -- (10.755, 11.700);
\draw[thin,dotted] (3.173, 0.300) -- (9.755, 11.700);
\draw[thin] (2.173, 0.300) -- (8.755, 11.700);
\draw[thin,dotted] (1.173, 0.300) -- (7.755, 11.700);
\draw[thin] (0.300, 0.520) -- (6.755, 11.700);
\draw[thin,dotted] (0.300, 2.252) -- (5.755, 11.700);
\draw[thin] (0.300, 3.984) -- (4.755, 11.700);
\draw[thin,dotted] (0.300, 5.716) -- (3.755, 11.700);
\draw[thin] (0.300, 7.448) -- (2.755, 11.700);
\draw[thin,dotted] (0.300, 9.180) -- (1.755, 11.700);
\draw[thin] (0.300, 10.912) -- (0.755, 11.700);
\draw[thin,dotted] (0.300, 1.212) -- (0.827, 0.300);
\draw[thin] (0.300, 2.944) -- (1.827, 0.300);
\draw[thin,dotted] (0.300, 4.677) -- (2.827, 0.300);
\draw[thin] (0.300, 6.409) -- (3.827, 0.300);
\draw[thin,dotted] (0.300, 8.141) -- (4.827, 0.300);
\draw[thin] (0.300, 9.873) -- (5.827, 0.300);
\draw[thin,dotted] (0.300, 11.605) -- (6.827, 0.300);
\draw[thin] (1.245, 11.700) -- (7.827, 0.300);
\draw[thin,dotted] (2.245, 11.700) -- (8.827, 0.300);
\draw[thin] (3.245, 11.700) -- (9.827, 0.300);
\draw[thin,dotted] (4.245, 11.700) -- (10.827, 0.300);
\draw[thin] (5.245, 11.700) -- (11.700, 0.520);
\draw[thin,dotted] (6.245, 11.700) -- (11.700, 2.252);
\draw[thin] (7.245, 11.700) -- (11.700, 3.984);
\draw[thin,dotted] (8.245, 11.700) -- (11.700, 5.716);
\draw[thin] (9.245, 11.700) -- (11.700, 7.448);
\draw[thin,dotted] (10.245, 11.700) -- (11.700, 9.180);
\draw[thin] (11.245, 11.700) -- (11.700, 10.912);
\draw[thin,dotted] (0.300, 0.866) -- (11.700, 0.866);
\draw[thin] (0.300, 1.732) -- (11.700, 1.732);
\draw[thin,dotted] (0.300, 2.598) -- (11.700, 2.598);
\draw[thin] (0.300, 3.464) -- (11.700, 3.464);
\draw[thin,dotted] (0.300, 4.330) -- (11.700, 4.330);
\draw[thin] (0.300, 5.196) -- (11.700, 5.196);
\draw[thin,dotted] (0.300, 6.062) -- (11.700, 6.062);
\draw[thin] (0.300, 6.928) -- (11.700, 6.928);
\draw[thin,dotted] (0.300, 7.794) -- (11.700, 7.794);
\draw[thin] (0.300, 8.660) -- (11.700, 8.660);
\draw[thin,dotted] (0.300, 9.526) -- (11.700, 9.526);
\draw[thin] (0.300, 10.392) -- (11.700, 10.392);
\draw[thin,dotted] (0.300, 11.258) -- (11.700, 11.258);
		\draw(6,-0.75) node{\footnotesize $(r,c) = (2,2)$};	
	\end{tikzpicture}
\end{tabular}
\end{center}
\caption{The blue shaded regions represent elements of $P_{r,c}(\mytikz{0.07}{\protect\triall},T, 2n+1)$, for the different values of $(r,c)$. The illustrations show how these are modified to reach the hatched regions, from which we can deduce the recursive expression in \eqref{eq:recDodd}.}
\label{fig:recDodd}
\end{figure}

\begin{equation}
\label{eq:recAprimEven}
\begin{split}
A'_{2n} 
&= -6 + \sum_{(r,c)\in I} |P_{r,c}(\mytikz{0.08}{\tridown},T, 2n)| \\
&= -6 + A'_{n} + A'_{n+1} + A'_{n+1} + A'_{n+1}. 
\end{split}
\end{equation}
See Figure~\ref{fig:recAprimEven} for a visualisation of the deduction of this recursion.

\begin{figure}[ht]
\centering
\begin{tabular}{cc}
	\begin{tikzpicture}[scale = 0.4]
		\draw[color=black, fill=cyan] (6.000, 3.464) -- (3.000, 8.660) -- (9.000, 8.660) -- cycle;
		\draw[pattern={north east lines}](6.000, 3.464) -- (3.000, 8.660) -- (9.000, 8.660) -- cycle;
		\draw[thin] (13.173, -1.432) -- (14.700, 1.212);
\draw[thin,dotted] (12.173, -1.432) -- (14.700, 2.944);
\draw[thin] (11.173, -1.432) -- (14.700, 4.677);
\draw[thin,dotted] (10.173, -1.432) -- (14.700, 6.409);
\draw[thin] (9.173, -1.432) -- (14.700, 8.141);
\draw[thin,dotted] (8.173, -1.432) -- (14.700, 9.873);
\draw[thin] (7.173, -1.432) -- (13.755, 9.968);
\draw[thin,dotted] (6.173, -1.432) -- (12.755, 9.968);
\draw[thin] (5.173, -1.432) -- (11.755, 9.968);
\draw[thin,dotted] (4.173, -1.432) -- (10.755, 9.968);
\draw[thin] (3.173, -1.432) -- (9.755, 9.968);
\draw[thin,dotted] (2.173, -1.432) -- (8.755, 9.968);
\draw[thin] (1.300, -1.212) -- (7.755, 9.968);
\draw[thin,dotted] (1.300, 0.520) -- (6.755, 9.968);
\draw[thin] (1.300, 2.252) -- (5.755, 9.968);
\draw[thin,dotted] (1.300, 3.984) -- (4.755, 9.968);
\draw[thin] (1.300, 5.716) -- (3.755, 9.968);
\draw[thin,dotted] (1.300, 7.448) -- (2.755, 9.968);
\draw[thin] (1.300, 9.180) -- (1.755, 9.968);
\draw[thin,dotted] (1.300, -0.520) -- (1.827, -1.432);
\draw[thin] (1.300, 1.212) -- (2.827, -1.432);
\draw[thin,dotted] (1.300, 2.944) -- (3.827, -1.432);
\draw[thin] (1.300, 4.677) -- (4.827, -1.432);
\draw[thin,dotted] (1.300, 6.409) -- (5.827, -1.432);
\draw[thin] (1.300, 8.141) -- (6.827, -1.432);
\draw[thin,dotted] (1.300, 9.873) -- (7.827, -1.432);
\draw[thin] (2.245, 9.968) -- (8.827, -1.432);
\draw[thin,dotted] (3.245, 9.968) -- (9.827, -1.432);
\draw[thin] (4.245, 9.968) -- (10.827, -1.432);
\draw[thin,dotted] (5.245, 9.968) -- (11.827, -1.432);
\draw[thin] (6.245, 9.968) -- (12.827, -1.432);
\draw[thin,dotted] (7.245, 9.968) -- (13.827, -1.432);
\draw[thin] (8.245, 9.968) -- (14.700, -1.212);
\draw[thin,dotted] (9.245, 9.968) -- (14.700, 0.520);
\draw[thin] (10.245, 9.968) -- (14.700, 2.252);
\draw[thin,dotted] (11.245, 9.968) -- (14.700, 3.984);
\draw[thin] (12.245, 9.968) -- (14.700, 5.716);
\draw[thin,dotted] (13.245, 9.968) -- (14.700, 7.448);
\draw[thin] (14.245, 9.968) -- (14.700, 9.180);
\draw[thin,dotted] (1.300, -0.866) -- (14.700, -0.866);
\draw[thin] (1.300, 0.000) -- (14.700, 0.000);
\draw[thin,dotted] (1.300, 0.866) -- (14.700, 0.866);
\draw[thin] (1.300, 1.732) -- (14.700, 1.732);
\draw[thin,dotted] (1.300, 2.598) -- (14.700, 2.598);
\draw[thin] (1.300, 3.464) -- (14.700, 3.464);
\draw[thin,dotted] (1.300, 4.330) -- (14.700, 4.330);
\draw[thin] (1.300, 5.196) -- (14.700, 5.196);
\draw[thin,dotted] (1.300, 6.062) -- (14.700, 6.062);
\draw[thin] (1.300, 6.928) -- (14.700, 6.928);
\draw[thin,dotted] (1.300, 7.794) -- (14.700, 7.794);
\draw[thin] (1.300, 8.660) -- (14.700, 8.660);
\draw[thin,dotted] (1.300, 9.526) -- (14.700, 9.526);
		\draw(8, -2.25) node{\footnotesize $(r,c) = (0,0)$};	
	\end{tikzpicture}
&
	\begin{tikzpicture}[scale = 0.4]
		\draw[color=black, fill=cyan] (7.000, 3.464) -- (4.000, 8.660) -- (10.000, 8.660) -- cycle;
		\draw[pattern={north east lines}] (3.000, 8.660) -- (7.000, 1.732) -- (11.000, 8.660) -- cycle;
		\draw[thin] (13.173, -1.432) -- (14.700, 1.212);
\draw[thin,dotted] (12.173, -1.432) -- (14.700, 2.944);
\draw[thin] (11.173, -1.432) -- (14.700, 4.677);
\draw[thin,dotted] (10.173, -1.432) -- (14.700, 6.409);
\draw[thin] (9.173, -1.432) -- (14.700, 8.141);
\draw[thin,dotted] (8.173, -1.432) -- (14.700, 9.873);
\draw[thin] (7.173, -1.432) -- (13.755, 9.968);
\draw[thin,dotted] (6.173, -1.432) -- (12.755, 9.968);
\draw[thin] (5.173, -1.432) -- (11.755, 9.968);
\draw[thin,dotted] (4.173, -1.432) -- (10.755, 9.968);
\draw[thin] (3.173, -1.432) -- (9.755, 9.968);
\draw[thin,dotted] (2.173, -1.432) -- (8.755, 9.968);
\draw[thin] (1.300, -1.212) -- (7.755, 9.968);
\draw[thin,dotted] (1.300, 0.520) -- (6.755, 9.968);
\draw[thin] (1.300, 2.252) -- (5.755, 9.968);
\draw[thin,dotted] (1.300, 3.984) -- (4.755, 9.968);
\draw[thin] (1.300, 5.716) -- (3.755, 9.968);
\draw[thin,dotted] (1.300, 7.448) -- (2.755, 9.968);
\draw[thin] (1.300, 9.180) -- (1.755, 9.968);
\draw[thin,dotted] (1.300, -0.520) -- (1.827, -1.432);
\draw[thin] (1.300, 1.212) -- (2.827, -1.432);
\draw[thin,dotted] (1.300, 2.944) -- (3.827, -1.432);
\draw[thin] (1.300, 4.677) -- (4.827, -1.432);
\draw[thin,dotted] (1.300, 6.409) -- (5.827, -1.432);
\draw[thin] (1.300, 8.141) -- (6.827, -1.432);
\draw[thin,dotted] (1.300, 9.873) -- (7.827, -1.432);
\draw[thin] (2.245, 9.968) -- (8.827, -1.432);
\draw[thin,dotted] (3.245, 9.968) -- (9.827, -1.432);
\draw[thin] (4.245, 9.968) -- (10.827, -1.432);
\draw[thin,dotted] (5.245, 9.968) -- (11.827, -1.432);
\draw[thin] (6.245, 9.968) -- (12.827, -1.432);
\draw[thin,dotted] (7.245, 9.968) -- (13.827, -1.432);
\draw[thin] (8.245, 9.968) -- (14.700, -1.212);
\draw[thin,dotted] (9.245, 9.968) -- (14.700, 0.520);
\draw[thin] (10.245, 9.968) -- (14.700, 2.252);
\draw[thin,dotted] (11.245, 9.968) -- (14.700, 3.984);
\draw[thin] (12.245, 9.968) -- (14.700, 5.716);
\draw[thin,dotted] (13.245, 9.968) -- (14.700, 7.448);
\draw[thin] (14.245, 9.968) -- (14.700, 9.180);
\draw[thin,dotted] (1.300, -0.866) -- (14.700, -0.866);
\draw[thin] (1.300, 0.000) -- (14.700, 0.000);
\draw[thin,dotted] (1.300, 0.866) -- (14.700, 0.866);
\draw[thin] (1.300, 1.732) -- (14.700, 1.732);
\draw[thin,dotted] (1.300, 2.598) -- (14.700, 2.598);
\draw[thin] (1.300, 3.464) -- (14.700, 3.464);
\draw[thin,dotted] (1.300, 4.330) -- (14.700, 4.330);
\draw[thin] (1.300, 5.196) -- (14.700, 5.196);
\draw[thin,dotted] (1.300, 6.062) -- (14.700, 6.062);
\draw[thin] (1.300, 6.928) -- (14.700, 6.928);
\draw[thin,dotted] (1.300, 7.794) -- (14.700, 7.794);
\draw[thin] (1.300, 8.660) -- (14.700, 8.660);
\draw[thin,dotted] (1.300, 9.526) -- (14.700, 9.526);
		\draw(8, -2.25) node{\footnotesize $(r,c) = (1,0)$};	
	\end{tikzpicture}
\\
	\begin{tikzpicture}[scale = 0.4]
		\draw[color=black, fill=cyan] (6.500, 2.598) -- (3.500, 7.794) -- (9.500, 7.794) -- cycle;
		\draw[pattern={north east lines}] (3.000, 8.660) -- (7.000, 1.732) -- (11.000, 8.660) -- cycle;
		\draw[thin] (13.173, -1.432) -- (14.700, 1.212);
\draw[thin,dotted] (12.173, -1.432) -- (14.700, 2.944);
\draw[thin] (11.173, -1.432) -- (14.700, 4.677);
\draw[thin,dotted] (10.173, -1.432) -- (14.700, 6.409);
\draw[thin] (9.173, -1.432) -- (14.700, 8.141);
\draw[thin,dotted] (8.173, -1.432) -- (14.700, 9.873);
\draw[thin] (7.173, -1.432) -- (13.755, 9.968);
\draw[thin,dotted] (6.173, -1.432) -- (12.755, 9.968);
\draw[thin] (5.173, -1.432) -- (11.755, 9.968);
\draw[thin,dotted] (4.173, -1.432) -- (10.755, 9.968);
\draw[thin] (3.173, -1.432) -- (9.755, 9.968);
\draw[thin,dotted] (2.173, -1.432) -- (8.755, 9.968);
\draw[thin] (1.300, -1.212) -- (7.755, 9.968);
\draw[thin,dotted] (1.300, 0.520) -- (6.755, 9.968);
\draw[thin] (1.300, 2.252) -- (5.755, 9.968);
\draw[thin,dotted] (1.300, 3.984) -- (4.755, 9.968);
\draw[thin] (1.300, 5.716) -- (3.755, 9.968);
\draw[thin,dotted] (1.300, 7.448) -- (2.755, 9.968);
\draw[thin] (1.300, 9.180) -- (1.755, 9.968);
\draw[thin,dotted] (1.300, -0.520) -- (1.827, -1.432);
\draw[thin] (1.300, 1.212) -- (2.827, -1.432);
\draw[thin,dotted] (1.300, 2.944) -- (3.827, -1.432);
\draw[thin] (1.300, 4.677) -- (4.827, -1.432);
\draw[thin,dotted] (1.300, 6.409) -- (5.827, -1.432);
\draw[thin] (1.300, 8.141) -- (6.827, -1.432);
\draw[thin,dotted] (1.300, 9.873) -- (7.827, -1.432);
\draw[thin] (2.245, 9.968) -- (8.827, -1.432);
\draw[thin,dotted] (3.245, 9.968) -- (9.827, -1.432);
\draw[thin] (4.245, 9.968) -- (10.827, -1.432);
\draw[thin,dotted] (5.245, 9.968) -- (11.827, -1.432);
\draw[thin] (6.245, 9.968) -- (12.827, -1.432);
\draw[thin,dotted] (7.245, 9.968) -- (13.827, -1.432);
\draw[thin] (8.245, 9.968) -- (14.700, -1.212);
\draw[thin,dotted] (9.245, 9.968) -- (14.700, 0.520);
\draw[thin] (10.245, 9.968) -- (14.700, 2.252);
\draw[thin,dotted] (11.245, 9.968) -- (14.700, 3.984);
\draw[thin] (12.245, 9.968) -- (14.700, 5.716);
\draw[thin,dotted] (13.245, 9.968) -- (14.700, 7.448);
\draw[thin] (14.245, 9.968) -- (14.700, 9.180);
\draw[thin,dotted] (1.300, -0.866) -- (14.700, -0.866);
\draw[thin] (1.300, 0.000) -- (14.700, 0.000);
\draw[thin,dotted] (1.300, 0.866) -- (14.700, 0.866);
\draw[thin] (1.300, 1.732) -- (14.700, 1.732);
\draw[thin,dotted] (1.300, 2.598) -- (14.700, 2.598);
\draw[thin] (1.300, 3.464) -- (14.700, 3.464);
\draw[thin,dotted] (1.300, 4.330) -- (14.700, 4.330);
\draw[thin] (1.300, 5.196) -- (14.700, 5.196);
\draw[thin,dotted] (1.300, 6.062) -- (14.700, 6.062);
\draw[thin] (1.300, 6.928) -- (14.700, 6.928);
\draw[thin,dotted] (1.300, 7.794) -- (14.700, 7.794);
\draw[thin] (1.300, 8.660) -- (14.700, 8.660);
\draw[thin,dotted] (1.300, 9.526) -- (14.700, 9.526);
		\draw(8, -2.25) node{\footnotesize $(r,c) = (1,2)$};	
	\end{tikzpicture}
&
	\begin{tikzpicture}[scale = 0.4]
		\draw[color=black, fill=cyan] (7.500, 2.598) -- (4.500, 7.794) -- (10.500, 7.794) -- cycle;
		\draw[pattern={north east lines}] (3.000, 8.660) -- (7.000, 1.732) -- (11.000, 8.660) -- cycle;
		\draw[thin] (13.173, -1.432) -- (14.700, 1.212);
\draw[thin,dotted] (12.173, -1.432) -- (14.700, 2.944);
\draw[thin] (11.173, -1.432) -- (14.700, 4.677);
\draw[thin,dotted] (10.173, -1.432) -- (14.700, 6.409);
\draw[thin] (9.173, -1.432) -- (14.700, 8.141);
\draw[thin,dotted] (8.173, -1.432) -- (14.700, 9.873);
\draw[thin] (7.173, -1.432) -- (13.755, 9.968);
\draw[thin,dotted] (6.173, -1.432) -- (12.755, 9.968);
\draw[thin] (5.173, -1.432) -- (11.755, 9.968);
\draw[thin,dotted] (4.173, -1.432) -- (10.755, 9.968);
\draw[thin] (3.173, -1.432) -- (9.755, 9.968);
\draw[thin,dotted] (2.173, -1.432) -- (8.755, 9.968);
\draw[thin] (1.300, -1.212) -- (7.755, 9.968);
\draw[thin,dotted] (1.300, 0.520) -- (6.755, 9.968);
\draw[thin] (1.300, 2.252) -- (5.755, 9.968);
\draw[thin,dotted] (1.300, 3.984) -- (4.755, 9.968);
\draw[thin] (1.300, 5.716) -- (3.755, 9.968);
\draw[thin,dotted] (1.300, 7.448) -- (2.755, 9.968);
\draw[thin] (1.300, 9.180) -- (1.755, 9.968);
\draw[thin,dotted] (1.300, -0.520) -- (1.827, -1.432);
\draw[thin] (1.300, 1.212) -- (2.827, -1.432);
\draw[thin,dotted] (1.300, 2.944) -- (3.827, -1.432);
\draw[thin] (1.300, 4.677) -- (4.827, -1.432);
\draw[thin,dotted] (1.300, 6.409) -- (5.827, -1.432);
\draw[thin] (1.300, 8.141) -- (6.827, -1.432);
\draw[thin,dotted] (1.300, 9.873) -- (7.827, -1.432);
\draw[thin] (2.245, 9.968) -- (8.827, -1.432);
\draw[thin,dotted] (3.245, 9.968) -- (9.827, -1.432);
\draw[thin] (4.245, 9.968) -- (10.827, -1.432);
\draw[thin,dotted] (5.245, 9.968) -- (11.827, -1.432);
\draw[thin] (6.245, 9.968) -- (12.827, -1.432);
\draw[thin,dotted] (7.245, 9.968) -- (13.827, -1.432);
\draw[thin] (8.245, 9.968) -- (14.700, -1.212);
\draw[thin,dotted] (9.245, 9.968) -- (14.700, 0.520);
\draw[thin] (10.245, 9.968) -- (14.700, 2.252);
\draw[thin,dotted] (11.245, 9.968) -- (14.700, 3.984);
\draw[thin] (12.245, 9.968) -- (14.700, 5.716);
\draw[thin,dotted] (13.245, 9.968) -- (14.700, 7.448);
\draw[thin] (14.245, 9.968) -- (14.700, 9.180);
\draw[thin,dotted] (1.300, -0.866) -- (14.700, -0.866);
\draw[thin] (1.300, 0.000) -- (14.700, 0.000);
\draw[thin,dotted] (1.300, 0.866) -- (14.700, 0.866);
\draw[thin] (1.300, 1.732) -- (14.700, 1.732);
\draw[thin,dotted] (1.300, 2.598) -- (14.700, 2.598);
\draw[thin] (1.300, 3.464) -- (14.700, 3.464);
\draw[thin,dotted] (1.300, 4.330) -- (14.700, 4.330);
\draw[thin] (1.300, 5.196) -- (14.700, 5.196);
\draw[thin,dotted] (1.300, 6.062) -- (14.700, 6.062);
\draw[thin] (1.300, 6.928) -- (14.700, 6.928);
\draw[thin,dotted] (1.300, 7.794) -- (14.700, 7.794);
\draw[thin] (1.300, 8.660) -- (14.700, 8.660);
\draw[thin,dotted] (1.300, 9.526) -- (14.700, 9.526);
		\draw(8, -2.25) node{\footnotesize $(r,c) = (2,2)$};	
	\end{tikzpicture}
\end{tabular}
\caption{The blue shaded regions represent elements of $P_{r,c}(\mytikz{0.07}{\protect\tridown},T, 2n)$, for the different values of $(r,c)$. The illustrations show how these are modified to reach the hatched regions, from which we can deduce the recursive expression in \eqref{eq:recAprimEven}.}
\label{fig:recAprimEven}
\end{figure}

\begin{align}
\label{eq:recAprimOdd}
\begin{split}
A'_{2n+1} 
&= -6 + \sum_{(r,c)\in I} |P_{r,c}(\mytikz{0.08}{\tridown},T, 2n+1)| \\
&= -6 + A'_{n+1} + A'_{n+1} + A'_{n+1} + A'_{n+2}.
\end{split}
\end{align}
See Figure~\ref{fig:recAprimOdd} for a visualisation of the deduction of this recursion.

\begin{figure}[ht]
\centering
\begin{tabular}{cc}
	\begin{tikzpicture}[scale = 0.4]
		\draw[color=black, fill=cyan] (6.500, 2.598) -- (3.000, 8.660) -- (10.000, 8.660) -- cycle;
		\draw[pattern={north east lines}] (3.000, 8.660) -- (7.000, 1.732) -- (11.000, 8.660) -- cycle;
		\draw[thin] (13.173, -1.432) -- (14.700, 1.212);
\draw[thin,dotted] (12.173, -1.432) -- (14.700, 2.944);
\draw[thin] (11.173, -1.432) -- (14.700, 4.677);
\draw[thin,dotted] (10.173, -1.432) -- (14.700, 6.409);
\draw[thin] (9.173, -1.432) -- (14.700, 8.141);
\draw[thin,dotted] (8.173, -1.432) -- (14.700, 9.873);
\draw[thin] (7.173, -1.432) -- (13.755, 9.968);
\draw[thin,dotted] (6.173, -1.432) -- (12.755, 9.968);
\draw[thin] (5.173, -1.432) -- (11.755, 9.968);
\draw[thin,dotted] (4.173, -1.432) -- (10.755, 9.968);
\draw[thin] (3.173, -1.432) -- (9.755, 9.968);
\draw[thin,dotted] (2.173, -1.432) -- (8.755, 9.968);
\draw[thin] (1.300, -1.212) -- (7.755, 9.968);
\draw[thin,dotted] (1.300, 0.520) -- (6.755, 9.968);
\draw[thin] (1.300, 2.252) -- (5.755, 9.968);
\draw[thin,dotted] (1.300, 3.984) -- (4.755, 9.968);
\draw[thin] (1.300, 5.716) -- (3.755, 9.968);
\draw[thin,dotted] (1.300, 7.448) -- (2.755, 9.968);
\draw[thin] (1.300, 9.180) -- (1.755, 9.968);
\draw[thin,dotted] (1.300, -0.520) -- (1.827, -1.432);
\draw[thin] (1.300, 1.212) -- (2.827, -1.432);
\draw[thin,dotted] (1.300, 2.944) -- (3.827, -1.432);
\draw[thin] (1.300, 4.677) -- (4.827, -1.432);
\draw[thin,dotted] (1.300, 6.409) -- (5.827, -1.432);
\draw[thin] (1.300, 8.141) -- (6.827, -1.432);
\draw[thin,dotted] (1.300, 9.873) -- (7.827, -1.432);
\draw[thin] (2.245, 9.968) -- (8.827, -1.432);
\draw[thin,dotted] (3.245, 9.968) -- (9.827, -1.432);
\draw[thin] (4.245, 9.968) -- (10.827, -1.432);
\draw[thin,dotted] (5.245, 9.968) -- (11.827, -1.432);
\draw[thin] (6.245, 9.968) -- (12.827, -1.432);
\draw[thin,dotted] (7.245, 9.968) -- (13.827, -1.432);
\draw[thin] (8.245, 9.968) -- (14.700, -1.212);
\draw[thin,dotted] (9.245, 9.968) -- (14.700, 0.520);
\draw[thin] (10.245, 9.968) -- (14.700, 2.252);
\draw[thin,dotted] (11.245, 9.968) -- (14.700, 3.984);
\draw[thin] (12.245, 9.968) -- (14.700, 5.716);
\draw[thin,dotted] (13.245, 9.968) -- (14.700, 7.448);
\draw[thin] (14.245, 9.968) -- (14.700, 9.180);
\draw[thin,dotted] (1.300, -0.866) -- (14.700, -0.866);
\draw[thin] (1.300, 0.000) -- (14.700, 0.000);
\draw[thin,dotted] (1.300, 0.866) -- (14.700, 0.866);
\draw[thin] (1.300, 1.732) -- (14.700, 1.732);
\draw[thin,dotted] (1.300, 2.598) -- (14.700, 2.598);
\draw[thin] (1.300, 3.464) -- (14.700, 3.464);
\draw[thin,dotted] (1.300, 4.330) -- (14.700, 4.330);
\draw[thin] (1.300, 5.196) -- (14.700, 5.196);
\draw[thin,dotted] (1.300, 6.062) -- (14.700, 6.062);
\draw[thin] (1.300, 6.928) -- (14.700, 6.928);
\draw[thin,dotted] (1.300, 7.794) -- (14.700, 7.794);
\draw[thin] (1.300, 8.660) -- (14.700, 8.660);
\draw[thin,dotted] (1.300, 9.526) -- (14.700, 9.526);
		\draw(8, -2.25) node{\footnotesize $(r,c) = (0,0)$};	
	\end{tikzpicture}
&
	\begin{tikzpicture}[scale = 0.4]
		\draw[color=black, fill=cyan] (7.500, 2.598) -- (4.000, 8.660) -- (11.000, 8.660) -- cycle;
		\draw[pattern={north east lines}] (3.000, 8.660) -- (7.000, 1.732) -- (11.000, 8.660) -- cycle;
		\draw[thin] (13.173, -1.432) -- (14.700, 1.212);
\draw[thin,dotted] (12.173, -1.432) -- (14.700, 2.944);
\draw[thin] (11.173, -1.432) -- (14.700, 4.677);
\draw[thin,dotted] (10.173, -1.432) -- (14.700, 6.409);
\draw[thin] (9.173, -1.432) -- (14.700, 8.141);
\draw[thin,dotted] (8.173, -1.432) -- (14.700, 9.873);
\draw[thin] (7.173, -1.432) -- (13.755, 9.968);
\draw[thin,dotted] (6.173, -1.432) -- (12.755, 9.968);
\draw[thin] (5.173, -1.432) -- (11.755, 9.968);
\draw[thin,dotted] (4.173, -1.432) -- (10.755, 9.968);
\draw[thin] (3.173, -1.432) -- (9.755, 9.968);
\draw[thin,dotted] (2.173, -1.432) -- (8.755, 9.968);
\draw[thin] (1.300, -1.212) -- (7.755, 9.968);
\draw[thin,dotted] (1.300, 0.520) -- (6.755, 9.968);
\draw[thin] (1.300, 2.252) -- (5.755, 9.968);
\draw[thin,dotted] (1.300, 3.984) -- (4.755, 9.968);
\draw[thin] (1.300, 5.716) -- (3.755, 9.968);
\draw[thin,dotted] (1.300, 7.448) -- (2.755, 9.968);
\draw[thin] (1.300, 9.180) -- (1.755, 9.968);
\draw[thin,dotted] (1.300, -0.520) -- (1.827, -1.432);
\draw[thin] (1.300, 1.212) -- (2.827, -1.432);
\draw[thin,dotted] (1.300, 2.944) -- (3.827, -1.432);
\draw[thin] (1.300, 4.677) -- (4.827, -1.432);
\draw[thin,dotted] (1.300, 6.409) -- (5.827, -1.432);
\draw[thin] (1.300, 8.141) -- (6.827, -1.432);
\draw[thin,dotted] (1.300, 9.873) -- (7.827, -1.432);
\draw[thin] (2.245, 9.968) -- (8.827, -1.432);
\draw[thin,dotted] (3.245, 9.968) -- (9.827, -1.432);
\draw[thin] (4.245, 9.968) -- (10.827, -1.432);
\draw[thin,dotted] (5.245, 9.968) -- (11.827, -1.432);
\draw[thin] (6.245, 9.968) -- (12.827, -1.432);
\draw[thin,dotted] (7.245, 9.968) -- (13.827, -1.432);
\draw[thin] (8.245, 9.968) -- (14.700, -1.212);
\draw[thin,dotted] (9.245, 9.968) -- (14.700, 0.520);
\draw[thin] (10.245, 9.968) -- (14.700, 2.252);
\draw[thin,dotted] (11.245, 9.968) -- (14.700, 3.984);
\draw[thin] (12.245, 9.968) -- (14.700, 5.716);
\draw[thin,dotted] (13.245, 9.968) -- (14.700, 7.448);
\draw[thin] (14.245, 9.968) -- (14.700, 9.180);
\draw[thin,dotted] (1.300, -0.866) -- (14.700, -0.866);
\draw[thin] (1.300, 0.000) -- (14.700, 0.000);
\draw[thin,dotted] (1.300, 0.866) -- (14.700, 0.866);
\draw[thin] (1.300, 1.732) -- (14.700, 1.732);
\draw[thin,dotted] (1.300, 2.598) -- (14.700, 2.598);
\draw[thin] (1.300, 3.464) -- (14.700, 3.464);
\draw[thin,dotted] (1.300, 4.330) -- (14.700, 4.330);
\draw[thin] (1.300, 5.196) -- (14.700, 5.196);
\draw[thin,dotted] (1.300, 6.062) -- (14.700, 6.062);
\draw[thin] (1.300, 6.928) -- (14.700, 6.928);
\draw[thin,dotted] (1.300, 7.794) -- (14.700, 7.794);
\draw[thin] (1.300, 8.660) -- (14.700, 8.660);
\draw[thin,dotted] (1.300, 9.526) -- (14.700, 9.526);
		\draw(8, -2.25) node{\footnotesize $(r,c) = (0,2)$};	
	\end{tikzpicture}
\\
	\begin{tikzpicture}[scale = 0.4]
		\draw[color=black, fill=cyan] (7.000, 1.732) -- (3.500, 7.794) -- (10.500, 7.794) -- cycle;
		\draw[pattern={north east lines}] (3.000, 8.660) -- (7.000, 1.732) -- (11.000, 8.660) -- cycle;
		\draw[thin] (13.173, -1.432) -- (14.700, 1.212);
\draw[thin,dotted] (12.173, -1.432) -- (14.700, 2.944);
\draw[thin] (11.173, -1.432) -- (14.700, 4.677);
\draw[thin,dotted] (10.173, -1.432) -- (14.700, 6.409);
\draw[thin] (9.173, -1.432) -- (14.700, 8.141);
\draw[thin,dotted] (8.173, -1.432) -- (14.700, 9.873);
\draw[thin] (7.173, -1.432) -- (13.755, 9.968);
\draw[thin,dotted] (6.173, -1.432) -- (12.755, 9.968);
\draw[thin] (5.173, -1.432) -- (11.755, 9.968);
\draw[thin,dotted] (4.173, -1.432) -- (10.755, 9.968);
\draw[thin] (3.173, -1.432) -- (9.755, 9.968);
\draw[thin,dotted] (2.173, -1.432) -- (8.755, 9.968);
\draw[thin] (1.300, -1.212) -- (7.755, 9.968);
\draw[thin,dotted] (1.300, 0.520) -- (6.755, 9.968);
\draw[thin] (1.300, 2.252) -- (5.755, 9.968);
\draw[thin,dotted] (1.300, 3.984) -- (4.755, 9.968);
\draw[thin] (1.300, 5.716) -- (3.755, 9.968);
\draw[thin,dotted] (1.300, 7.448) -- (2.755, 9.968);
\draw[thin] (1.300, 9.180) -- (1.755, 9.968);
\draw[thin,dotted] (1.300, -0.520) -- (1.827, -1.432);
\draw[thin] (1.300, 1.212) -- (2.827, -1.432);
\draw[thin,dotted] (1.300, 2.944) -- (3.827, -1.432);
\draw[thin] (1.300, 4.677) -- (4.827, -1.432);
\draw[thin,dotted] (1.300, 6.409) -- (5.827, -1.432);
\draw[thin] (1.300, 8.141) -- (6.827, -1.432);
\draw[thin,dotted] (1.300, 9.873) -- (7.827, -1.432);
\draw[thin] (2.245, 9.968) -- (8.827, -1.432);
\draw[thin,dotted] (3.245, 9.968) -- (9.827, -1.432);
\draw[thin] (4.245, 9.968) -- (10.827, -1.432);
\draw[thin,dotted] (5.245, 9.968) -- (11.827, -1.432);
\draw[thin] (6.245, 9.968) -- (12.827, -1.432);
\draw[thin,dotted] (7.245, 9.968) -- (13.827, -1.432);
\draw[thin] (8.245, 9.968) -- (14.700, -1.212);
\draw[thin,dotted] (9.245, 9.968) -- (14.700, 0.520);
\draw[thin] (10.245, 9.968) -- (14.700, 2.252);
\draw[thin,dotted] (11.245, 9.968) -- (14.700, 3.984);
\draw[thin] (12.245, 9.968) -- (14.700, 5.716);
\draw[thin,dotted] (13.245, 9.968) -- (14.700, 7.448);
\draw[thin] (14.245, 9.968) -- (14.700, 9.180);
\draw[thin,dotted] (1.300, -0.866) -- (14.700, -0.866);
\draw[thin] (1.300, 0.000) -- (14.700, 0.000);
\draw[thin,dotted] (1.300, 0.866) -- (14.700, 0.866);
\draw[thin] (1.300, 1.732) -- (14.700, 1.732);
\draw[thin,dotted] (1.300, 2.598) -- (14.700, 2.598);
\draw[thin] (1.300, 3.464) -- (14.700, 3.464);
\draw[thin,dotted] (1.300, 4.330) -- (14.700, 4.330);
\draw[thin] (1.300, 5.196) -- (14.700, 5.196);
\draw[thin,dotted] (1.300, 6.062) -- (14.700, 6.062);
\draw[thin] (1.300, 6.928) -- (14.700, 6.928);
\draw[thin,dotted] (1.300, 7.794) -- (14.700, 7.794);
\draw[thin] (1.300, 8.660) -- (14.700, 8.660);
\draw[thin,dotted] (1.300, 9.526) -- (14.700, 9.526);
		\draw(8, -2.25) node{\footnotesize $(r,c) = (1,0)$};	
	\end{tikzpicture}
&
	\begin{tikzpicture}[scale = 0.4]
		\draw[color=black, fill=cyan] (8.000, 1.732) -- (4.500, 7.794) -- (11.500, 7.794) -- cycle;
		\draw[pattern={north east lines}] (3.000, 8.660) -- (8.000, 0.000) -- (13.000, 8.660) -- cycle;
		\draw[thin] (13.173, -1.432) -- (14.700, 1.212);
\draw[thin,dotted] (12.173, -1.432) -- (14.700, 2.944);
\draw[thin] (11.173, -1.432) -- (14.700, 4.677);
\draw[thin,dotted] (10.173, -1.432) -- (14.700, 6.409);
\draw[thin] (9.173, -1.432) -- (14.700, 8.141);
\draw[thin,dotted] (8.173, -1.432) -- (14.700, 9.873);
\draw[thin] (7.173, -1.432) -- (13.755, 9.968);
\draw[thin,dotted] (6.173, -1.432) -- (12.755, 9.968);
\draw[thin] (5.173, -1.432) -- (11.755, 9.968);
\draw[thin,dotted] (4.173, -1.432) -- (10.755, 9.968);
\draw[thin] (3.173, -1.432) -- (9.755, 9.968);
\draw[thin,dotted] (2.173, -1.432) -- (8.755, 9.968);
\draw[thin] (1.300, -1.212) -- (7.755, 9.968);
\draw[thin,dotted] (1.300, 0.520) -- (6.755, 9.968);
\draw[thin] (1.300, 2.252) -- (5.755, 9.968);
\draw[thin,dotted] (1.300, 3.984) -- (4.755, 9.968);
\draw[thin] (1.300, 5.716) -- (3.755, 9.968);
\draw[thin,dotted] (1.300, 7.448) -- (2.755, 9.968);
\draw[thin] (1.300, 9.180) -- (1.755, 9.968);
\draw[thin,dotted] (1.300, -0.520) -- (1.827, -1.432);
\draw[thin] (1.300, 1.212) -- (2.827, -1.432);
\draw[thin,dotted] (1.300, 2.944) -- (3.827, -1.432);
\draw[thin] (1.300, 4.677) -- (4.827, -1.432);
\draw[thin,dotted] (1.300, 6.409) -- (5.827, -1.432);
\draw[thin] (1.300, 8.141) -- (6.827, -1.432);
\draw[thin,dotted] (1.300, 9.873) -- (7.827, -1.432);
\draw[thin] (2.245, 9.968) -- (8.827, -1.432);
\draw[thin,dotted] (3.245, 9.968) -- (9.827, -1.432);
\draw[thin] (4.245, 9.968) -- (10.827, -1.432);
\draw[thin,dotted] (5.245, 9.968) -- (11.827, -1.432);
\draw[thin] (6.245, 9.968) -- (12.827, -1.432);
\draw[thin,dotted] (7.245, 9.968) -- (13.827, -1.432);
\draw[thin] (8.245, 9.968) -- (14.700, -1.212);
\draw[thin,dotted] (9.245, 9.968) -- (14.700, 0.520);
\draw[thin] (10.245, 9.968) -- (14.700, 2.252);
\draw[thin,dotted] (11.245, 9.968) -- (14.700, 3.984);
\draw[thin] (12.245, 9.968) -- (14.700, 5.716);
\draw[thin,dotted] (13.245, 9.968) -- (14.700, 7.448);
\draw[thin] (14.245, 9.968) -- (14.700, 9.180);
\draw[thin,dotted] (1.300, -0.866) -- (14.700, -0.866);
\draw[thin] (1.300, 0.000) -- (14.700, 0.000);
\draw[thin,dotted] (1.300, 0.866) -- (14.700, 0.866);
\draw[thin] (1.300, 1.732) -- (14.700, 1.732);
\draw[thin,dotted] (1.300, 2.598) -- (14.700, 2.598);
\draw[thin] (1.300, 3.464) -- (14.700, 3.464);
\draw[thin,dotted] (1.300, 4.330) -- (14.700, 4.330);
\draw[thin] (1.300, 5.196) -- (14.700, 5.196);
\draw[thin,dotted] (1.300, 6.062) -- (14.700, 6.062);
\draw[thin] (1.300, 6.928) -- (14.700, 6.928);
\draw[thin,dotted] (1.300, 7.794) -- (14.700, 7.794);
\draw[thin] (1.300, 8.660) -- (14.700, 8.660);
\draw[thin,dotted] (1.300, 9.526) -- (14.700, 9.526);
		\draw(8, -2.25) node{\footnotesize $(r,c) = (1,2)$};	
	\end{tikzpicture}
\end{tabular}
\caption{The blue shaded regions represent elements of $P_{r,c}(\mytikz{0.07}{\protect\tridown},T, 2n+1)$, for the different values of $(r,c)$. The illustrations show how these are modified to reach the hatched regions, from which we can deduce the recursive expression in \eqref{eq:recAprimOdd}.}
\label{fig:recAprimOdd}
\end{figure}

\section{Proof of Main Theorem}

In this section we provide the final steps the proof of Theorem~\ref{thm:main}. First, recall the definition of the quantities $A_n$, $B_n$, $C_n$, and $D_n$ from \eqref{eq:defofABCD}. By the help of the recursions in the previous section we can prove the following lemma. 

\begin{table}[ht]
\footnotesize
\centering
\begin{tabular}{c*{11}{r}}
\toprule
$n$ & 
\multicolumn{1}{p{7mm}}{\hfill 1} & 
\multicolumn{1}{p{7mm}}{\hfill 2} & 
\multicolumn{1}{p{7mm}}{\hfill 3} & 
\multicolumn{1}{p{7mm}}{\hfill 4} & 
\multicolumn{1}{p{7mm}}{\hfill 5} & 
\multicolumn{1}{p{7mm}}{\hfill 6} & 
\multicolumn{1}{p{7mm}}{\hfill 7} & 
\multicolumn{1}{p{7mm}}{\hfill 8} &
\multicolumn{1}{p{7mm}}{\hfill 9} & 
\multicolumn{1}{p{7mm}}{\hfill10} \\ 
\midrule
$A_n$ &    2 &    8 &   22 &   44 &   74 &  112 &  158 &  212 &  274 &  344 \\
$B_n$ &      &    4 &   14 &   32 &   58 &   92 &  134 &  184 &  242 &  308 \\
$C_n$ &      &    2 &    8 &   22 &   44 &   74 &  112 &  158 &  212 &  274 \\
$D_n$ &      &    1 &    4 &   14 &   32 &   58 &   92 &  134 &  184 &  242 \\
\bottomrule
\end{tabular}
\caption{Initial terms for $A$, $B$, $C$ and $D$, from \eqref{eq:defofABCD}. }
\label{table:initialvaluesABCD}
\end{table}

\begin{table}[ht]
\footnotesize
\centering
\begin{tabular}{c*{11}{r}}
\toprule
$n$ & 
\multicolumn{1}{p{7mm}}{\hfill 1} & 
\multicolumn{1}{p{7mm}}{\hfill 2} & 
\multicolumn{1}{p{7mm}}{\hfill 3} & 
\multicolumn{1}{p{7mm}}{\hfill 4} & 
\multicolumn{1}{p{7mm}}{\hfill 5} & 
\multicolumn{1}{p{7mm}}{\hfill 6} & 
\multicolumn{1}{p{7mm}}{\hfill 7} & 
\multicolumn{1}{p{7mm}}{\hfill 8} &
\multicolumn{1}{p{7mm}}{\hfill 9} & 
\multicolumn{1}{p{7mm}}{\hfill10} \\ 
\midrule
$A'_n$ &    1 &    2 &   4 &   8 &   14 &  22 &  32 &  44 &  58 &  74 \\
\bottomrule
\end{tabular}
\caption{Initial terms for $A'$, from \eqref{eq:defofAprim}.}
\label{table:initialvaluesAprim}
\end{table}

\begin{lemma}
\label{lemma:ABequalsCD}
Let $n\geq2$. Then 
\begin{equation}
\label{eq:ABequalsCD}
\begin{split}
	A_{n} &= C_{n+1}, \\
	B_{n} &= D_{n+1}. 
\end{split}
\end{equation}
\end{lemma}

\begin{proof}
We give a proof by induction. Let us consider the index $n$ in the equalities in \eqref{eq:ABequalsCD} depending on whether it is odd or even. The initial cases are directly seen from Table~\ref{table:initialvaluesABCD}. Assume for induction that the equalities in \eqref{eq:ABequalsCD} hold for $n<2p$. Then, for the induction step, we have from \eqref{eq:recAeven}, \eqref{eq:recCodd}, and the induction assumption
\begin{align*}
	A_{2p} - C_{2p+1} &= A_{p} + 3B_{p+1} - C_{p+1} - 2B_{p+1} - D_{p+2} = 0,
\end{align*}
and with the help of \eqref{eq:recAodd}, and \eqref{eq:recCeven} we get
\begin{align*}
	A_{2p+1} - C_{2p+2} &= 3A_{p+1} + D_{p+2} - A_{p+1} - 2C_{p+2} - D_{p+2} = 0.
\end{align*}
In the same way we obtain 
\begin{align*}
	B_{2p} - D_{2p+1} &= A_{p} + 2C_{p+1} + B_{p+1} - 3C_{p+1} - D_{p+2} = 0,
\end{align*}
and 
\begin{align*}
	B_{2p+1} - D_{2p+2} &= A_{p+1} + 2B_{p+1} + D_{p+2} - A_{p+1} - 3D_{p+2} = 0,
\end{align*}
which complete the induction and the proof. 
\end{proof}

By the help of Lemma~\ref{lemma:ABequalsCD} and the recursions from the previous section, we obtain
\begin{equation}
\label{eq:recAB}
\renewcommand{\arraystretch}{1.3}
\left\{\begin{array}{l@{ \ }l@{ \ }l@{ \ }l@{}l@{ \ }l@{ \ }l@{}l}
	A_{2n}   &= - 6 &+& &A_{n}   &+&3&B_{n+1}, \\
	A_{2n+1} &= - 6 &+&3&A_{n+1} &+& &B_{n+1}, \\
	B_{2n}   &= - 6 &+&3&A_{n}   &+& &B_{n+1}, \\
	B_{2n+1} &= - 6 &+& &A_{n+1} &+&3&B_{n+1}. 
\end{array}\right.
\end{equation}
Recall also the recursions for $A'_n$ from \eqref{eq:recAprimEven} and~\ref{eq:recAprimOdd}
\begin{equation}
\label{eq:recAprim}
\renewcommand{\arraystretch}{1.3}
\left\{\begin{array}{l@{ \ }l@{ \ }l@{ \ }l@{}l@{ \ }l@{ \ }l@{}l}
	A'_{2n}   &= - 6 &+& &A'_{n}   &+&3&A'_{n+1}, \\
	A'_{2n+1} &= - 6 &+&3&A'_{n+1} &+& &A'_{n+2}.
\end{array}\right.
\end{equation}
The initial values for \eqref{eq:recAB} and \eqref{eq:recAprim} are given in Table~\ref{table:initialvaluesABCD} and Table~\ref{table:initialvaluesAprim} respectively. The last step is now to give the proof of Theorem~\ref{thm:main}.

\begin{proof}[Proof of Theorem~\ref{thm:main}]
To prove \eqref{eq:MainUpwards}, we give simultaneously an explicit formula for $B_n$ from \eqref{eq:recAB}. That is, we claim the following 
\begin{align}
	A_n &= 4n^2 - 6n + 4, \label{eq:Aformula}\\
	B_n &= 4n^2 -10n + 8. \label{eq:Bformula}
\end{align}
We give a proof of the above claim by induction on $n$. The initial cases are directly from Table~\ref{table:initialvaluesABCD}. Assume for induction that \eqref{eq:Aformula} and \eqref{eq:Bformula} hold for $n<2p$. For the induction step we have by \eqref{eq:recAB}, \eqref{eq:Aformula}, \eqref{eq:Bformula}, and the induction assumption
\begin{align*}
	A_{2p} &- \big(4(2p)^2 - 6(2p) + 4\big) \\
	&= - 6+A_{p} + 3B_{p+1}  - \big(4(2p)^2 - 6(2p) + 4\big) \\
	&= - 6 + \big(4p^2 - 6p + 4\big) + 3\big(4(p+1)^2 -10(p+1) + 8\big) \\
	&\quad - \big(4(2p)^2 - 6(2p) + 4\big) \\
	&= 0. 
\end{align*}
The remaining cases for $A_{2p+1}$, $B_{2p}$, and $B_{2p+1}$ follow in the same way.

Similarly, the formula for $A'_n$ from \eqref{eq:MainDownwards} is easily verified for small $n$ from Table~\ref{table:initialvaluesAprim}. Assume for induction that \eqref{eq:MainDownwards} holds for $n<2p$. Then, in the induction step, we have with the help of \eqref{eq:recAprim}
\begin{align*}
	A'_{2p} & - \big( (2p)^2 - 3(2p) + 4\big) \\
	&= -6 + A'_{p} + 3A'_{p+1} - \big( (2p)^2 - 3(2p) + 4\big) \\
	&= -6 + \big( p^2 - 3p + 4\big) + 3 \big( (p+1)^2 - 3(p+1) + 4\big) \\
	& \quad - \big( (2p)^2 - 3(2p) + 4\big) \\
	&= 0.
\end{align*}
The case for $n=2p+1$ follows analogously.
\end{proof}

Let us close the paper with a small remark and question; from the formulas for $A_n$, $B_n$, and $A'_n$ from \eqref{eq:MainUpwards}, \eqref{eq:Bformula}, and \eqref{eq:MainDownwards} we see that there is a neat connection between the number of upward- and downward oriented triangular patterns, namely
\begin{equation}
\label{eq:ConnectionABandAprim}
\renewcommand{\arraystretch}{1.3}
\left\{\begin{array}{l@{ \ }l}
	A'_{2n}   &= A_n, \\
	A'_{2n-1} &= B_n, \\
\end{array}\right.
\end{equation}
for $n\geq2$. It would be interesting to see and find out if there is a direct geometrical or combinatorial argument leading to the connection in \eqref{eq:ConnectionABandAprim}, or if it is just a numerical coincidence.

\end{document}